\documentclass[final]{siamonline0516}
\pdfoutput=1
\usepackage{color}
\usepackage{amssymb,amstext,enumerate}
\usepackage{tikz,pgfplots}
\usetikzlibrary{shapes}
\usepackage{graphicx}
\usepackage{multirow}
\newcommand{\rottext}[1]{\rotatebox{90}{\hbox to 20mm{\hss #1\hss}}}
\newcommand{\til}{\widetilde}

\newcommand{\N}{\mathbb{N}}
\newcommand{\R}{\mathbb{R}}

\newcommand{\argmin}{\mathrm{argmin}}

\newcommand{\kron}{\otimes}

\newcommand{\diag}{\mathrm{diag}}

\renewcommand{\phi}{\mathbf{\varphi}}

\newcommand{\CB}{\mathcal{B}}

\newcommand{\CI}{\mathcal{I}}     

\newcommand{\CJ}{\mathcal{J}}

\newcommand{\CD}{\mathcal{D}}

\newcommand{\CS}{\mathcal{S}}                          

\newcommand{\bfd}{\mathbf{d}}

\newcommand{\bfA}{\mathbf{A}}
\newcommand{\bfC}{\mathbf{C}}

\newcommand{\bfU}{\mathbf{U}}
\newcommand{\bfP}{\mathbf{P}}
\newcommand{\bfI}{\mathbf{I}}
\newcommand{\bfD}{\mathbf{D}}

\newcommand{\bfb}{\mathbf{b}}
\newcommand{\bfc}{\mathbf{c}}
\newcommand{\bfH}{\mathbf{H}}

\newcommand{\bfx}{\mathbf{x}}

\newcommand{\bfe}{\mathbf{e}}
\newcommand{\bfu}{\mathbf{u}}
\newcommand{\bfy}{\mathbf{y}}

\newcommand{\bfM}{\mathbf{M}}
\newcommand{\bfG}{\mathbf{G}}
\newcommand{\bfL}{\mathbf{L}}
\newcommand{\bfw}{\mathbf{w}}

\newcommand{\bfz}{\mathbf{z}}
\newcommand{\bfp}{\mathbf{p}}
\newcommand{\bfg}{\mathbf{g}}
\newcommand{\bfh}{\mathbf{h}}
\newcommand{\bfr}{\mathbf{r}}

\newcommand{\bfv}{\mathbf{v}}

\newcommand{\bfJ}{\mathbf{J}}

\newcommand{\bfzero}{\mathbf{0}}

\newcommand{\x}{x}														
\renewcommand{\xi}[1]{\x_{#1}}								

\newcommand{\CL}{\mathcal{L}}                 

\newcommand{\hf}{\frac{1}{2}}

\newlength\iwidth
\newlength\iheight
\newcommand{\TheTitle}{Efficient Numerical Optimization For
Susceptibility Artifact Correction Of EPI-MRI} 
\newcommand{\TheAuthors}{Jan Macdonald and Lars Ruthotto}
\headers{Efficient Optimization for Susceptibility Artifact Correction}{\TheAuthors}
\title{{\TheTitle}}
\author{
  Jan Macdonald\thanks{Technische Universit\"at Berlin, Germany
  (\email{macdonald@math.tu-berlin.de}) and Emory University, Atlanta,
  GA, USA.}
  \and
  Lars Ruthotto\thanks{Emory University, Atlanta, GA, USA
  (\email{lruthotto@emory.edu}, corresponding author). The research of this author was in part supported by 
National Science Foundation (NSF) Grant DMS-1522599.}
}
\ifpdf
\hypersetup{
  pdftitle={\TheTitle},
  pdfauthor={\TheAuthors}
}
\fi
\begin{document}
\maketitle
%
\begin{abstract} 
	We present two efficient numerical methods for susceptibility
	artifact correction applicable in Echo Planar Imaging (EPI), an ultra fast Magnetic Resonance Imaging (MRI)
	technique widely used in clinical applications. Both methods address a major practical drawback of EPI, the so-called susceptibility artifacts, which consist of geometrical transformations and intensity modulations.
	We consider a tailored variational image registration problem that is based on a physical distortion model and aims at minimizing the distance of two oppositely distorted images subject to invertibility constraints. 
	We follow a discretize-then-optimize approach and present a novel face-staggered discretization yielding a separable structure in the discretized distance function and the invertibility constraints.
	The presence of a smoothness regularizer renders the overall optimization problem non-separable, but we present two optimization schemes that exploit the partial separability.
	First, we derive a block-Jacobi preconditioner to be used in a Gauss-Newton-PCG method.
	Second, we consider a splitting of the separable and non-separable
	part and solve the resulting problem using the Alternating Direction
	Method of Multipliers (ADMM). We provide a detailed convergence proof for ADMM for this non-convex optimization problem.
	Both schemes are of essentially linear complexity and are suitable for parallel computing.
	A considerable advantage of the proposed schemes over established methods is the reduced time-to-solution. In our numerical experiment using high-resolution 3D imaging data, our parallel implementation of the ADMM method solves a 3D problem with more than 5 million degrees
	of freedom in less than 50 seconds on a standard laptop, which is a considerable improvement over existing methods.
\end{abstract}
%
%
\section{Introduction} 
\label{sec:introduction}
Echo-Planar-Imaging (EPI) is an ultra fast Magnetic Resonance Imaging
(MRI) technique that is widely used in medical imaging
applications~\cite{StehlingEtAl1991}. For example, EPI is used in the
neuroscience to accelerate the acquisition of Diffusion Tensor Images
(DTI)~\cite{LeBihanEtAl2001} or intra-operatively to guide
surgery~\cite{DagaEtAl2014}. While offering a substantial reduction of scan
time, a drawback of EPI is its high sensitivity to small inhomogeneities
of the local magnetic field. In practical applications the magnetic
field is perturbed inevitably by susceptibility variations of the object
being imaged. The strength of the inhomogeneity is correlated with the
strength of the external magnetic field~\cite{deMunckEtAl1996} and, thus,
correcting for these artifacts becomes increasingly relevant for
high-resolution MRI. 
\par
A physical model for distortions caused by susceptibility variations was derived
in~\cite{ChangFitzpatrick1992}. It was shown that the distortion consists
of two components: a geometric displacement and a modulation of image
intensities. 
It is important to note that for EPI-MRI the displacement is practically limited 
to a fixed and a priori known direction, the so-called \emph{phase encoding direction}. 
Distortions in the other directions (frequency encoding and slice selection direction) 
are negligible. It is this particular structure that we exploit in this work to obtain a partially separable objective function and efficient optimization schemes. The intensity modulation is given by the Jacobian determinant
of the geometric transformation and ensures mass-preservation under the
assumption that the Jacobian determinant is strictly positive almost
everywhere. As mentioned in~\cite{ChangFitzpatrick1992} this property
needs to be ensured by judicious choice of the measurement parameters.
\par
In recent years many approaches for susceptibility artifact correction
were proposed; see, e.g.,~\cite{HollandEtAl2010} for an extensive overview. Most of them employ the physical distortion model in~\cite{ChangFitzpatrick1992}, and fall into one of two categories. One approach is to obtain a
\emph{field map}, which is an estimate of the field inhomogeneity, from a
reference scan and apply the physical model for susceptibility
artifacts~\cite{Jezzard2012,JezzardEtAl1995}. 
Alternatively, the field map can be estimated using an additional EPI image with reversed
phase-encoding gradients and thus opposite deformations.
The estimation problem can be phrased as a
nonlinear image registration problem as originally proposed
in~\cite{ChangFitzpatrick1992}.
This approach, commonly referred to as \emph{reversed gradient} method, is taken in the following. 
\par
There are several numerical implementations of the reversed gradient method, for
example,~\cite{IrfanogluEtAl2015,RuthottoEtAl2013hysco,RuthottoEtAlPMB2012,HollandEtAl2010,SkareAndersson2005,AnderssonEtAl2003}. 
Given the two images, the goal is to estimate the field inhomogeneity such that
the resulting deformations render both images as similar as possible to
one another. Recently, several studies have shown the (often superior
with respect to the field map approaches) quality of reversed gradient
approaches, e.g., in Arterial Spin Labeling
(ASL)~\cite{MadaiEtAl2016}, quantitative MRI~\cite{HongEtAl2015}, and
perfusion weighted MRI~\cite{VardalEtAl2013}. 
\par
Despite the increasing popularity of reversed gradient methods, relatively little attention 
has been paid to their efficient numerical implementation.
Although the methods in~\cite{IrfanogluEtAl2015,RuthottoEtAl2013hysco,RuthottoEtAlPMB2012,HollandEtAl2010,SkareAndersson2005,AnderssonEtAl2003}
are all based on the same physical distortion model, they employ
different discretizations and optimization strategies.
Hessian-based minimization schemes are used in~\cite{RuthottoEtAl2013hysco,RuthottoEtAlPMB2012,HollandEtAl2010,SkareAndersson2005,AnderssonEtAl2003}. 
As to be expected, the computationally most expensive step in these iterative approaches is
computing the search direction, which requires approximately solving a linear system.
An often neglected aspect is the impact of the chosen discretization on the complexity of this step.
As we show in this work, a careful numerical discretization that is motivated by the physical distortion model
can be exploited to substantially reduce the computational cost of this step.
\par
In this paper, we present and compare two novel fast and scalable numerical
optimization schemes that accelerate reversed gradient based
susceptibility artifact correction. Similar
to~\cite{RuthottoEtAl2013hysco,RuthottoEtAlPMB2012,HollandEtAl2010}, we
consider a variational formulation consisting of a distance functional
and a regularization functional that improves robustness against noise.
We use a discretize-then-optimize paradigm and follow the general
guidelines described in~\cite{Modersitzki2009}. In contrast to existing
works, we derive a face-staggered discretization that exploits the fact
that, in EPI correction, displacements are practically limited along one a priori
known direction. We show that this discretization leads to a separable
structure of the discrete distance function, which results in a block-diagonal Hessian, whereas the smoothness
regularizer yields global coupling but has exploitable structure. We propose two approaches to
exploit this structure for fast numerical optimization: First, we
construct a parallel block-Jacobi preconditioner for a Gauss-Newton
method. Second, we derive a completely parallelizable algorithm that aims at minimizing the non-convex objective function using the framework of the
Alternating Direction Method of Multipliers (ADMM)~\cite{BoydEtAl2011}. 
\par
The complexity of the first method is linear in the number of faces in the
computational mesh. The complexity of the second method is essentially
linear, however, using ADMM the problem decouples into
small-dimensional subproblems that can be solved efficiently and in
parallel. As to be expected and shown in our experiments, this comes at
the cost of an increased number of iterations and therefore the choice
of method depends on the computational platform employed.
\par
Our proposed schemes exploit the specific structure of the distortion
model in EPI-MRI, where displacements due to susceptibility artifacts
only occur in one spatial dimension. In contrast to that, displacements can occur in all
spatial dimensions in general image registration problems. This is why the numerical techniques in this paper differ from existing
efficient solvers for general image registration problems, such as,
e.g.\,~\cite{CachierEtAl2003,HakerEtAl2004,VercauterenEtAl2009,MangBiros2015InexactNewton}.
\par
The paper is organized as follows. \Cref{sec:math} introduces the
forward and inverse problem of susceptibility artifact correction.
\Cref{sec:discretization} describes the discretization using a
face-staggered grid for the displacement. \Cref{sec:numerical_optimization}
describes the optimization methods. \Cref{sec:experiments} outlines the
potential of the method using real-life data. Finally, \cref{sec:summary} concludes the paper.


\section{Mathematical Formulation} 
\label{sec:math}
In this section, we briefly review the physical distortion model derived in~\cite{ChangFitzpatrick1992} and the variational formulation of EPI susceptibility artifact correction also used in~\cite{RuthottoEtAl2013hysco,RuthottoEtAlPMB2012}. For clarity of presentation, we limit the discussion to the three-dimensional case, which is most relevant for our applications.
\par
Let us first derive the forward problem. Let $\Omega \subset \R^3$ be a
rectangular domain of interest, let $v\in\R^3$ denote the phase-encoding
direction, and let the magnitude of the field inhomogeneity at a point
$x \in\Omega$ be $b(x)$, where in the forward problem $b: \Omega \to
\R$ is assumed to be known. As derived in~\cite{ChangFitzpatrick1992},
the distorted measurement $\CI_v : \R^3 \to \R$ and the undistorted
image $\CI : \R^3  \to \R$ satisfy
\begin{equation}\label{eq:distortionModel}
	\CI(x) = \CI_v(x + b(x) v) \cdot \det \nabla (x + b(x) v) = \CI_v(x + b(x) v) \cdot (1+ v^\top\nabla b(x)),
\end{equation}
where $\det \nabla (x + b(x) v)$ denotes the Jacobian determinant of the transformation $x + b(x) v$.
As in~\cite{Modersitzki2009} we assume that the images are continuously differentiable and compactly supported in $\Omega$.
Note that the Jacobian determinant simplifies to a directional derivative since displacements are limited along one line. Similarly, let $\CI_{-v}$ denote a second image acquired with phase-encoding direction $-v$ but otherwise unchanged imaging parameters. Using the physical distortion model~\cref{eq:distortionModel}, we have
\begin{equation}\label{eq:reversedGradient}
	\CI(x) = \CI_v(x + b(x) v) \cdot (1+ v^\top\nabla b(x)) = \CI_{-v}(x - b(x) v) \cdot (1 - v^\top\nabla b(x)).
\end{equation}
\par
In the inverse problem, both the inhomogeneity $b$ and the undistorted
image $\CI$ are unknown. However, given two images $\CI_v$ and
$\CI_{-v}$ acquired with phase-encoding directions $v$ and $-v$,
respectively, the goal is to estimate  $b$ such that the second equality
in~\cref{eq:reversedGradient} holds approximately. Commonly a simple
$L^2$ distance term is used, i.e.,
\begin{equation}\label{eq:DistContinuous}
	\CD(b) = \hf \int_\Omega \left(\CI_v(x + b(x) v) \cdot (1+
	v^\top\nabla b(x)) - \CI_{-v}(x - b(x) v) \cdot
	(1-v^\top\nabla b(x))\right)^2 dx.
\end{equation}
Minimizing the distance term alone is an ill-posed problem and thus
regularization is added;
see~\cite{RuthottoEtAl2013hysco,RuthottoEtAlPMB2012,HollandEtAl2010}.
In~\cite{deMunckEtAl1996}, methods for computing the field inhomogeneity
$b$ from susceptibility properties of the brain are developed. In this
work it is shown that the field inhomogeneity should be in the Sobolev
space $H^1(\Omega)$ and thus, we consider the smoothness regularizer
\begin{equation}\label{eq:RegContinuous}
	\CS(b) = \hf \int_{\Omega} \left\| \nabla b(x) \right\|^2 dx.
\end{equation}
It is important to note that the physical distortion
model~\cref{eq:distortionModel} only holds if the Jacobian determinants
for both phase encoding directions are strictly positive for almost all $x\in\Omega$; see~\cite{ChangFitzpatrick1992}.  Therefore, as firstly suggested in~\cite{RuthottoEtAl2013hysco,RuthottoEtAlPMB2012}, we impose a constraint on the Jacobian determinant
\begin{equation}\label{eq:JacConstr}
	 -1 \leq  v^\top \nabla b(x)  \leq 1 \quad \text{ for almost all } \quad x \in \Omega.
\end{equation}
As also discussed~\cite[Sec. 3]{RuthottoMod2015}, we restrict the set of feasible field inhomogeneities to a closed ball $\CB$ with respect to the $L^\infty$-norm whose radius depends only on the diameter of $\Omega$. To this end, note that both the distance and the regularization functional vanish for any large enough constant field inhomogeneity $b$, due to the fact that $\CI_v$ and $\CI_{-v}$ are supported within the bounded set $\Omega$. However, these global minimizers of $\CD(b)$ and $\CS(b)$ would be implausible in practical applications.
This leads to the variational problem
\begin{equation}\label{eq:varProb}
	\min_{b\in\CB} \left\{ \CJ(b) = \CD(b) + \alpha \CS(b)\right\} \quad \text{ subject to } -1
	\leq v^\top\nabla b(x) \leq 1,\quad  \forall x \in \Omega,
\end{equation}
where the parameter $\alpha>0$ balances between minimizing the distance
and the regularity of the solution. There is no general rule for
choosing an ``optimal'' regularization parameter especially for nonlinear
inverse problems, however, several criteria such as generalized cross
validation~\cite{GolubEtAl1979,HaberOldenburg2000},
L-curve~\cite{Hansen1998}, or discrepancy principle~\cite{Vogel2002} are
commonly used. In this paper, we assume that $\alpha$ is chosen by the
user and in our numerical experiments we show the robustness of the
proposed optimization with respect to the choice of $\alpha$; see \cref{par:2d_example}. 
Following the guidelines
in~\cite{Modersitzki2009} we first discretize the variational problem,
see \cref{sec:discretization}, and then discuss numerical methods
for solving the discrete optimization problem in \cref{sec:numerical_optimization}.


\section{Discretization} 
\label{sec:discretization}
In this section, we derive a face-staggered discretization of the variational
problem~\cref{eq:varProb} that leads to a separable structure of the
discretized distance function. 
\par
Our notation follows the general guidelines in~\cite{Modersitzki2009}. For ease of presentation, we consider a rectangular domain  $\Omega =
(0,1)\times(0,1)\times(0,1) \subset \R^3$ and assume that the phase encoding direction and thus the direction of the distortion is aligned with the first coordinate axis. In other words, we assume \ $v = e_1$, where $e_1$ is the first unit vector.
In our experience, this is not a practical limitation since image data can be adequately
rearranged. To simplify our notation, we assume that $\Omega$ is divided
into $m^3$ voxels with edge length $h=1/m$ in all three coordinate
directions.  Our implementation supports arbitrary numbers of voxels and
anisotropic voxel sizes. The images $\CI_{v}$ and $\CI_{-v}$ are assumed
to be compactly supported and continuously differentiable functions. In
practice, a continuous image model is built from discrete data by using
interpolation; see~\cite[Sec.~3]{Modersitzki2009} for details.
\par
To obtain a separable structure of the discrete distance term, we
discretize the field inhomogeneity, $b$, by a vector $\bfb \in
\R^{(m+1)m^2}$ on the $x_1$-faces of a regular grid as visualized in \cref{fig:compare_grids}. For brevity, we denote the number of $x_1$-faces by  $n={(m+1)m^2}$. Clearly, the elements of the vector $\bfb$ can be accessed using linear indices or sub indices 
\begin{equation*}
	\bfb_{ijk} = b\left((i-1)h, (j-0.5)h,  (k-0.5)h\right),\quad
	\text{for}\quad i=1,\ldots,m+1\quad\text{and}\quad
	j,k=1,\ldots,m. 
\end{equation*}
The restriction of the field inhomogeneities $b$ to the $L^\infty$-ball $\CB$ is discretized by restricting $\bfb$ to a symmetric closed box $\Sigma\subseteq\R^{n}$ with edge length $2 \cdot {\rm{diam}(\Omega)}$ around the origin. In our numerical experiments, it was not necessary to enforce this constraint.
The distance functional~\cref{eq:DistContinuous} is approximated by a
midpoint rule. To this end, the geometric transformation and the
intensity modulation in~\cref{eq:distortionModel} are approximated in
the cell-centers by simple averaging and short finite differences,
respectively.
\begin{figure}[t]
  \begin{center}
    \input{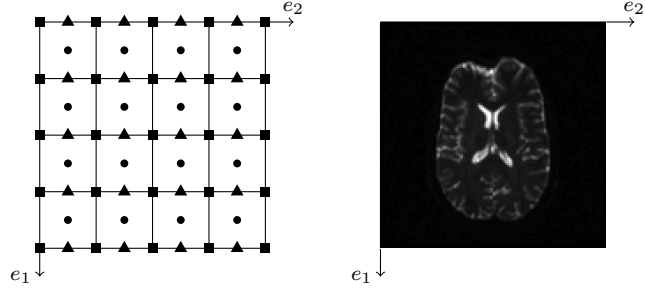}
  \end{center}
\caption{Cell-centered grid (circles), nodal grid (squares), and $e_1$-staggered grid
(triangles) on a
2-dimensional box-domain with $4\times 4$ voxels (left). Example of a
real-life deformed image (right). Data is courtesy of Harald Kugel,
University Hospital M\"unster, Germany, cf.~\cref{sec:experiments}.}
\label{fig:compare_grids}
\end{figure}
To compute the field inhomogeneity in the cell-centers given an $e_1$-staggered discretization, we use the averaging matrix
\[
	\bfA_1 = \bfI\left(m^2\right)\kron\til\bfA_1,\quad \til\bfA_1 =
\frac{1}{2}\begin{pmatrix} 1 & 1 &  & 0 \\  & \ddots & \ddots
  &  \\ 0 &  & 1 & 1\end{pmatrix}\in\R^{m\times m+1},
\]
where $\kron$ denotes the Kronecker product of two matrices and $\bfI(k)$ denotes the identity
matrix of size $k\times k$. From
that we obtain the discretized displacement in $e_1$-direction using 
\begin{equation}\label{eq:A}
	\bfA = \begin{pmatrix}
	\bfA_1\\\mathbf{0}\\\mathbf{0}\end{pmatrix}\in\R^{3m^3\times
	n},
\end{equation}
where $\mathbf{0}\in\R^{m^3 \times n}$ is a matrix of all zeroes.
\par
Similarly, we discretize the first partial differential operator using short finite differences matrices
\begin{equation}\label{eq:D1}
  \bfD_1 = \bfI\left(m^2\right) \kron\til \bfD(m+1,h), \text{ where } \quad \til\bfD(m,h) =
  \frac{1}{h}\begin{pmatrix} -1 & 1 & & 0\\ & \ddots & \ddots & \\ 0 &
  & -1 & 1\end{pmatrix} \in \R^{(m-1)\times m}.
\end{equation}
Combining~\cref{eq:A} and~\cref{eq:D1} the distance functional~\cref{eq:DistContinuous} is approximated by
\begin{equation}\label{eq:D}
	D(\bfb) = \frac{h^3}{2} \big\| \CI_{v}(\bfx + \bfA \bfb) \odot
	(\bfe + \bfD_1 \bfb) - \CI_{-v}(\bfx - \bfA \bfb) \odot (\bfe - \bfD_1 \bfb)  \big\|^2,
\end{equation}
where $\bfx$ are the cell-centers of a uniform grid, $\CI_{\pm v}(\bfx \pm \bfA \bfb)$ denote vectors of the image intensities at the shifted grid points, $\odot$ denotes the component-wise Hadamard product, and $\bfe\in\R^{m^3}$ is a vector of all ones.
\par
Similarly, we approximate the regularization functional~\cref{eq:RegContinuous} by
\begin{equation} \label{eq:S}
S(\bfb) = \frac{h^3}{2}\left(\|\bfD_1\bfb\|^2+\left\|\bfD_2\bfb\right\|^2+\left\|\bfD_3\bfb\right\|^2\right),
\end{equation}
where the discrete partial derivative operators are
\begin{equation}\label{eq:d2d3}
	\bfD_2 = \bfI(m) \kron \til\bfD(m,h) \kron \bfI(m+1)
	\quad \text{ and } \quad \bfD_3 = \til\bfD(m,h) \kron
	\bfI(m\cdot (m+1)).
\end{equation}
\par
Using~\cref{eq:D1} to discretize the constraint and combining~\cref{eq:D} and~\cref{eq:S}, we obtain the finite-dimensional optimization problem 
\begin{equation}\label{eq:discprob}
  \min_{\bfb \in \Sigma} \left\{J(\bfb) = D(\bfb) + \alpha S(\bfb) \right\} \quad \textrm{ subject to }
  -1\leq \bfD_1\bfb \leq 1.
\end{equation}
All components of the optimization problem are smooth, the regularizer is a convex quadratic, the constraints are convex, however, the distance function is in general non-convex.
The non-convexity is addressed using a multilevel strategy; see~\cref{sub:multilevel}.
In this paper, we consider two approaches for incorporating the linear inequality constraints.
First, we consider augmenting the objective function by the penalty term introduced in~\cite{RuthottoEtAlPMB2012}
\begin{equation}\label{eq:P}
	P(\bfb) = \beta  h^3  \bfe^\top \varphi(\bfD_1\bfb),\quad
	\varphi\colon\R\rightarrow\R\colon x \mapsto 
		\begin{cases}
			\frac{x^4}{(1-x^2)} &,\text{ if }-1 \leq x \leq 1\\
			\infty				&,\text{ else}
		\end{cases}
	,	
\end{equation}
where $\beta>0$ is a penalty parameter and $\varphi$ is applied
component-wise.
The penalty function can be interpreted as a simplification of the volume term of the hyperelastic regularization function in~\cite{BurgerEtAl2013} for transformations that are limited along one spatial directions.
As also proposed in~\cite{RuthottoEtAlPMB2012} we use an
inexact Gauss-Newton method for solving the resulting smooth and
box-constrained optimization problem
\begin{equation}\label{eq:JGN}
  \min_{\bfb \in \Sigma} \left\{J_{\rm GN}(\bfb) = D(\bfb) + \alpha S(\bfb) +P(\bfb)\right\}.
\end{equation}
In contrast to~\cite{RuthottoEtAlPMB2012} we exploit the separable structure of the
data term and the penalty to develop an effective and efficient
preconditioner. Second, we exploit the separable structure of the
distance term and the constraints in~\cref{eq:discprob} and derive an efficient implementation
of a Sequential Quadratic Programming (SQP) method. Both approaches are
described in more detail in \cref{sec:numerical_optimization}.
\par
In preparation for efficient, Hessian-based optimization, we quickly derive the gradients and approximate Hessians of the discretized distance, regularizer, and penalty term. 
Let us denote the residual of the distance term~\cref{eq:D} by
\[\bfr(\bfb) =
\CI_v(\bfx+\bfA\bfb)\odot(\bfe+\bfD_1\bfb)-\CI_{-v}(\bfx-\bfA\bfb)\odot(\bfe-\bfD_1\bfb).\]
Then, denoting the Jacobian matrix of the residual for a fixed $\bfb$ by $\bfJ_\bfr(\bfb) \in \R^{m^3 \times n}$, we obtain the gradient and Gauss-Newton approximation to the Hessian as
\begin{equation*}
	\nabla D(\bfb) = h^3\  \bfJ_{\bfr}(\bfb)^\top\bfr(\bfb)\quad \text{ and } \quad \nabla^2 D(\bfb) \approx \bfH_D(\bfb) = h^3\ \bfJ_{\bfr}(\bfb)^\top
	\bfJ_{\bfr}(\bfb).
\end{equation*}
The Jacobian  of the residual, is given by
\begin{align*}
  \bfJ_{\bfr}(\bfb) &= \diag\left(\partial_1 \CI_v(\bfx+\bfA\bfb)\odot(\bfe+\bfD_1\bfb)\right)\bfA_1+\diag\left(\CI_{v}(\bfx+\bfA\bfb)\right)\bfD_1\\
  &\quad + \diag\left(\partial_1 \CI_{-v}(\bfx-\bfA\bfb)\odot(\bfe-\bfD_1\bfb)\right)\bfA_1+\diag\left(\CI_{-v}(\bfx-\bfA\bfb)\right)\bfD_1,
\end{align*}
where $\partial_1 \CI_{\pm v}(\bfx\pm \bfA \bfb)$ denote vectors of the first partial derivatives of the images evaluated at the shifted grid points and $\diag(\bfv)\in\R^{n\times n}$ is a diagonal matrix with diagonal entries given by the vector $\bfv \in \R^n$.
\par
Unlike the distance function, the regularizer introduces coupling along all dimensions.
The Hessian and the gradient of the regularization function~\cref{eq:S} are
\begin{equation*}
	\nabla^2 S = h^3 \left(
	\bfD_1^\top\bfD_1+\bfD_2^\top\bfD_2+\bfD_3^\top\bfD_3\right)
	\quad \text{ and } \quad 
	\nabla S(\bfb) =\nabla^2 S\ \bfb.
\end{equation*}
It is important to note though that $\nabla^2 S$ is of Block-Toeplitz-Toeplitz-Block~(BTTB) structure. Finally, the gradient and Hessian of the penalty term~\cref{eq:P} are
\begin{equation*}
	\nabla P(\bfb) = \beta h^3 \bfD_1 \varphi^\prime(\bfD_1\bfb)
	\quad \text{ and } \quad \nabla^2 P(\bfb) = \beta h^3 \bfD_1^\top
	\diag(\varphi^{\prime\prime}(\bfD_1\bfb))\bfD_1,
\end{equation*}
for feasible $\bfb$, i.e., $-1 \leq \bfD_1\bfb \leq 1$. Here, the functions
\begin{equation}\label{eq:dphi}
	\varphi^\prime(x) = \frac{4x^3-2x^5}{(1-x^2)^2}
	\quad \text{ and } \quad
	\varphi^{\prime\prime}(x) = \frac{2x^2(x^4-3x^2+6)}{(1-x^2)^3}, \quad x \in [-1,1]
\end{equation}
are applied component-wise. Combining the above derivations, the
gradients and approximated Hessians of the discrete objective functions are
\begin{equation}\label{eq:dJ}
	\nabla J(\bfb) = \nabla D(\bfb) + \alpha \nabla S(\bfb) \text{ and } \bfH_J(\bfb) = \bfH_D (\bfb) + \alpha
	\nabla^2 S +\gamma \bfI(n)
\end{equation}
and
\begin{equation}\label{eq:dJGN}
  \nabla J_{\rm GN}(\bfb) = \nabla J(\bfb) + \nabla
  P(\bfb) \text{ and } \bfH_{J_{\rm GN}}(\bfb) = \bfH_J(\bfb) + \nabla^2 P(\bfb)
\end{equation}
respectively, where $\gamma>0$ ensures that the approximate Hessians are symmetric positive definite. In our numerical experiments, we choose $\gamma = 10^{-3}$.

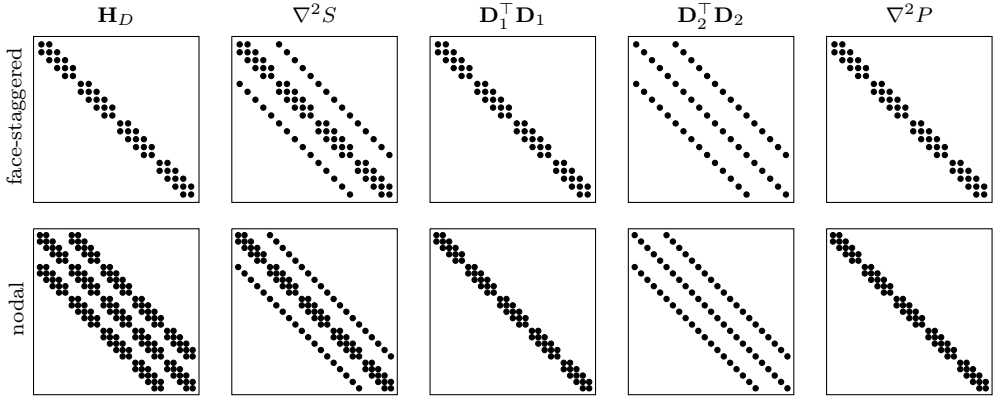
\begin{figure}
  \scriptsize
  \setlength\iwidth{22mm}
  \setlength\iheight{22mm}
  \renewcommand{\rottext}[1]{\rotatebox{90}{\hbox to 23mm{\hss #1\hss}}}
  \begin{center}
    \begin{tabular}{c@{\ }cccccc}
      & $\bfH_D$ & $\nabla^2S$ & $\bfD_1^\top\bfD_1$ &
      $\bfD_2^\top\bfD_2$ & $\nabla^2P$ \\
	 \rottext{face-staggered} &
%
%
\begin{tikzpicture}

\begin{axis}[%
width=\iwidth,
height=\iheight,
at={(0\iwidth,0\iheight)},
scale only axis,
xmin=0,
xmax=21,
xtick={\empty},
y dir=reverse,
ymin=0,
ymax=21,
ytick={\empty},
axis background/.style={fill=white}
]
\addplot [color=black,mark size=1pt,only marks,mark=*,mark options={solid},forget plot]
  table[row sep=crcr]{%
1	1\\
1	2\\
2	1\\
2	2\\
2	3\\
3	2\\
3	3\\
3	4\\
4	3\\
4	4\\
4	5\\
5	4\\
5	5\\
6	6\\
6	7\\
7	6\\
7	7\\
7	8\\
8	7\\
8	8\\
8	9\\
9	8\\
9	9\\
9	10\\
10	9\\
10	10\\
11	11\\
11	12\\
12	11\\
12	12\\
12	13\\
13	12\\
13	13\\
13	14\\
14	13\\
14	14\\
14	15\\
15	14\\
15	15\\
16	16\\
16	17\\
17	16\\
17	17\\
17	18\\
18	17\\
18	18\\
18	19\\
19	18\\
19	19\\
19	20\\
20	19\\
20	20\\
};
\end{axis}
\end{tikzpicture}%
 &
%
%
\begin{tikzpicture}

\begin{axis}[%
width=\iwidth,
height=\iheight,
at={(0\iwidth,0\iheight)},
scale only axis,
xmin=0,
xmax=21,
xtick={\empty},
y dir=reverse,
ymin=0,
ymax=21,
ytick={\empty},
axis background/.style={fill=white}
]
\addplot [color=black,mark size=1pt,only marks,mark=*,mark options={solid},forget plot]
  table[row sep=crcr]{%
1	1\\
1	2\\
1	6\\
2	1\\
2	2\\
2	3\\
2	7\\
3	2\\
3	3\\
3	4\\
3	8\\
4	3\\
4	4\\
4	5\\
4	9\\
5	4\\
5	5\\
5	10\\
6	1\\
6	6\\
6	7\\
6	11\\
7	2\\
7	6\\
7	7\\
7	8\\
7	12\\
8	3\\
8	7\\
8	8\\
8	9\\
8	13\\
9	4\\
9	8\\
9	9\\
9	10\\
9	14\\
10	5\\
10	9\\
10	10\\
10	15\\
11	6\\
11	11\\
11	12\\
11	16\\
12	7\\
12	11\\
12	12\\
12	13\\
12	17\\
13	8\\
13	12\\
13	13\\
13	14\\
13	18\\
14	9\\
14	13\\
14	14\\
14	15\\
14	19\\
15	10\\
15	14\\
15	15\\
15	20\\
16	11\\
16	16\\
16	17\\
17	12\\
17	16\\
17	17\\
17	18\\
18	13\\
18	17\\
18	18\\
18	19\\
19	14\\
19	18\\
19	19\\
19	20\\
20	15\\
20	19\\
20	20\\
};
\end{axis}
\end{tikzpicture}%
 &
%
%
\begin{tikzpicture}

\begin{axis}[%
width=\iwidth,
height=\iheight,
at={(0\iwidth,0\iheight)},
scale only axis,
xmin=0,
xmax=21,
xtick={\empty},
y dir=reverse,
ymin=0,
ymax=21,
ytick={\empty},
axis background/.style={fill=white}
]
\addplot [color=black,mark size=1pt,only marks,mark=*,mark options={solid},forget plot]
  table[row sep=crcr]{%
1	1\\
1	2\\
2	1\\
2	2\\
2	3\\
3	2\\
3	3\\
3	4\\
4	3\\
4	4\\
4	5\\
5	4\\
5	5\\
6	6\\
6	7\\
7	6\\
7	7\\
7	8\\
8	7\\
8	8\\
8	9\\
9	8\\
9	9\\
9	10\\
10	9\\
10	10\\
11	11\\
11	12\\
12	11\\
12	12\\
12	13\\
13	12\\
13	13\\
13	14\\
14	13\\
14	14\\
14	15\\
15	14\\
15	15\\
16	16\\
16	17\\
17	16\\
17	17\\
17	18\\
18	17\\
18	18\\
18	19\\
19	18\\
19	19\\
19	20\\
20	19\\
20	20\\
};
\end{axis}
\end{tikzpicture}%
 &
%
%
\begin{tikzpicture}

\begin{axis}[%
width=\iwidth,
height=\iheight,
at={(0\iwidth,0\iheight)},
scale only axis,
xmin=0,
xmax=21,
xtick={\empty},
y dir=reverse,
ymin=0,
ymax=21,
ytick={\empty},
axis background/.style={fill=white}
]
\addplot [color=black,mark size=1pt,only marks,mark=*,mark options={solid},forget plot]
  table[row sep=crcr]{%
1	1\\
1	6\\
2	2\\
2	7\\
3	3\\
3	8\\
4	4\\
4	9\\
5	5\\
5	10\\
6	1\\
6	6\\
6	11\\
7	2\\
7	7\\
7	12\\
8	3\\
8	8\\
8	13\\
9	4\\
9	9\\
9	14\\
10	5\\
10	10\\
10	15\\
11	6\\
11	11\\
11	16\\
12	7\\
12	12\\
12	17\\
13	8\\
13	13\\
13	18\\
14	9\\
14	14\\
14	19\\
15	10\\
15	15\\
15	20\\
16	11\\
16	16\\
17	12\\
17	17\\
18	13\\
18	18\\
19	14\\
19	19\\
20	15\\
20	20\\
};
\end{axis}
\end{tikzpicture}%
 &
%
%
\begin{tikzpicture}

\begin{axis}[%
width=\iwidth,
height=\iheight,
at={(0\iwidth,0\iheight)},
scale only axis,
xmin=0,
xmax=21,
xtick={\empty},
y dir=reverse,
ymin=0,
ymax=21,
ytick={\empty},
axis background/.style={fill=white}
]
\addplot [color=black,mark size=1pt,only marks,mark=*,mark options={solid},forget plot]
  table[row sep=crcr]{%
1	1\\
1	2\\
2	1\\
2	2\\
2	3\\
3	2\\
3	3\\
3	4\\
4	3\\
4	4\\
4	5\\
5	4\\
5	5\\
6	6\\
6	7\\
7	6\\
7	7\\
7	8\\
8	7\\
8	8\\
8	9\\
9	8\\
9	9\\
9	10\\
10	9\\
10	10\\
11	11\\
11	12\\
12	11\\
12	12\\
12	13\\
13	12\\
13	13\\
13	14\\
14	13\\
14	14\\
14	15\\
15	14\\
15	15\\
16	16\\
16	17\\
17	16\\
17	17\\
17	18\\
18	17\\
18	18\\
18	19\\
19	18\\
19	19\\
19	20\\
20	19\\
20	20\\
};
\end{axis}
\end{tikzpicture}%
 \\[0.5em]
	  \rottext{nodal} & 
%
%
\begin{tikzpicture}

\begin{axis}[%
width=\iwidth,
height=\iheight,
at={(0\iwidth,0\iheight)},
scale only axis,
xmin=0,
xmax=26,
xtick={\empty},
y dir=reverse,
ymin=0,
ymax=26,
ytick={\empty},
axis background/.style={fill=white}
]
\addplot [color=black,mark size=1pt,only marks,mark=*,mark options={solid},forget plot]
  table[row sep=crcr]{%
1	1\\
1	2\\
1	6\\
1	7\\
2	1\\
2	2\\
2	3\\
2	6\\
2	7\\
2	8\\
3	2\\
3	3\\
3	4\\
3	7\\
3	8\\
3	9\\
4	3\\
4	4\\
4	5\\
4	8\\
4	9\\
4	10\\
5	4\\
5	5\\
5	9\\
5	10\\
6	1\\
6	2\\
6	6\\
6	7\\
6	11\\
6	12\\
7	1\\
7	2\\
7	3\\
7	6\\
7	7\\
7	8\\
7	11\\
7	12\\
7	13\\
8	2\\
8	3\\
8	4\\
8	7\\
8	8\\
8	9\\
8	12\\
8	13\\
8	14\\
9	3\\
9	4\\
9	5\\
9	8\\
9	9\\
9	10\\
9	13\\
9	14\\
9	15\\
10	4\\
10	5\\
10	9\\
10	10\\
10	14\\
10	15\\
11	6\\
11	7\\
11	11\\
11	12\\
11	16\\
11	17\\
12	6\\
12	7\\
12	8\\
12	11\\
12	12\\
12	13\\
12	16\\
12	17\\
12	18\\
13	7\\
13	8\\
13	9\\
13	12\\
13	13\\
13	14\\
13	17\\
13	18\\
13	19\\
14	8\\
14	9\\
14	10\\
14	13\\
14	14\\
14	15\\
14	18\\
14	19\\
14	20\\
15	9\\
15	10\\
15	14\\
15	15\\
15	19\\
15	20\\
16	11\\
16	12\\
16	16\\
16	17\\
16	21\\
16	22\\
17	11\\
17	12\\
17	13\\
17	16\\
17	17\\
17	18\\
17	21\\
17	22\\
17	23\\
18	12\\
18	13\\
18	14\\
18	17\\
18	18\\
18	19\\
18	22\\
18	23\\
18	24\\
19	13\\
19	14\\
19	15\\
19	18\\
19	19\\
19	20\\
19	23\\
19	24\\
19	25\\
20	14\\
20	15\\
20	19\\
20	20\\
20	24\\
20	25\\
21	16\\
21	17\\
21	21\\
21	22\\
22	16\\
22	17\\
22	18\\
22	21\\
22	22\\
22	23\\
23	17\\
23	18\\
23	19\\
23	22\\
23	23\\
23	24\\
24	18\\
24	19\\
24	20\\
24	23\\
24	24\\
24	25\\
25	19\\
25	20\\
25	24\\
25	25\\
};
\end{axis}
\end{tikzpicture}%
 &
%
%
\begin{tikzpicture}

\begin{axis}[%
width=\iwidth,
height=\iheight,
at={(0\iwidth,0\iheight)},
scale only axis,
xmin=0,
xmax=26,
xtick={\empty},
y dir=reverse,
ymin=0,
ymax=26,
ytick={\empty},
axis background/.style={fill=white}
]
\addplot [color=black,mark size=1pt,only marks,mark=*,mark options={solid},forget plot]
  table[row sep=crcr]{%
1	1\\
1	2\\
1	6\\
2	1\\
2	2\\
2	3\\
2	7\\
3	2\\
3	3\\
3	4\\
3	8\\
4	3\\
4	4\\
4	5\\
4	9\\
5	4\\
5	5\\
5	10\\
6	1\\
6	6\\
6	7\\
6	11\\
7	2\\
7	6\\
7	7\\
7	8\\
7	12\\
8	3\\
8	7\\
8	8\\
8	9\\
8	13\\
9	4\\
9	8\\
9	9\\
9	10\\
9	14\\
10	5\\
10	9\\
10	10\\
10	15\\
11	6\\
11	11\\
11	12\\
11	16\\
12	7\\
12	11\\
12	12\\
12	13\\
12	17\\
13	8\\
13	12\\
13	13\\
13	14\\
13	18\\
14	9\\
14	13\\
14	14\\
14	15\\
14	19\\
15	10\\
15	14\\
15	15\\
15	20\\
16	11\\
16	16\\
16	17\\
16	21\\
17	12\\
17	16\\
17	17\\
17	18\\
17	22\\
18	13\\
18	17\\
18	18\\
18	19\\
18	23\\
19	14\\
19	18\\
19	19\\
19	20\\
19	24\\
20	15\\
20	19\\
20	20\\
20	25\\
21	16\\
21	21\\
21	22\\
22	17\\
22	21\\
22	22\\
22	23\\
23	18\\
23	22\\
23	23\\
23	24\\
24	19\\
24	23\\
24	24\\
24	25\\
25	20\\
25	24\\
25	25\\
};
\end{axis}
\end{tikzpicture}%
 &
%
%
\begin{tikzpicture}

\begin{axis}[%
width=\iwidth,
height=\iheight,
at={(0\iwidth,0\iheight)},
scale only axis,
xmin=0,
xmax=26,
xtick={\empty},
y dir=reverse,
ymin=0,
ymax=26,
ytick={\empty},
axis background/.style={fill=white}
]
\addplot [color=black,mark size=1pt,only marks,mark=*,mark options={solid},forget plot]
  table[row sep=crcr]{%
1	1\\
1	2\\
2	1\\
2	2\\
2	3\\
3	2\\
3	3\\
3	4\\
4	3\\
4	4\\
4	5\\
5	4\\
5	5\\
6	6\\
6	7\\
7	6\\
7	7\\
7	8\\
8	7\\
8	8\\
8	9\\
9	8\\
9	9\\
9	10\\
10	9\\
10	10\\
11	11\\
11	12\\
12	11\\
12	12\\
12	13\\
13	12\\
13	13\\
13	14\\
14	13\\
14	14\\
14	15\\
15	14\\
15	15\\
16	16\\
16	17\\
17	16\\
17	17\\
17	18\\
18	17\\
18	18\\
18	19\\
19	18\\
19	19\\
19	20\\
20	19\\
20	20\\
21	21\\
21	22\\
22	21\\
22	22\\
22	23\\
23	22\\
23	23\\
23	24\\
24	23\\
24	24\\
24	25\\
25	24\\
25	25\\
};
\end{axis}
\end{tikzpicture}%
 &
%
%
\begin{tikzpicture}

\begin{axis}[%
width=\iwidth,
height=\iheight,
at={(0\iwidth,0\iheight)},
scale only axis,
xmin=0,
xmax=26,
xtick={\empty},
y dir=reverse,
ymin=0,
ymax=26,
ytick={\empty},
axis background/.style={fill=white}
]
\addplot [color=black,mark size=1pt,only marks,mark=*,mark options={solid},forget plot]
  table[row sep=crcr]{%
1	1\\
1	6\\
2	2\\
2	7\\
3	3\\
3	8\\
4	4\\
4	9\\
5	5\\
5	10\\
6	1\\
6	6\\
6	11\\
7	2\\
7	7\\
7	12\\
8	3\\
8	8\\
8	13\\
9	4\\
9	9\\
9	14\\
10	5\\
10	10\\
10	15\\
11	6\\
11	11\\
11	16\\
12	7\\
12	12\\
12	17\\
13	8\\
13	13\\
13	18\\
14	9\\
14	14\\
14	19\\
15	10\\
15	15\\
15	20\\
16	11\\
16	16\\
16	21\\
17	12\\
17	17\\
17	22\\
18	13\\
18	18\\
18	23\\
19	14\\
19	19\\
19	24\\
20	15\\
20	20\\
20	25\\
21	16\\
21	21\\
22	17\\
22	22\\
23	18\\
23	23\\
24	19\\
24	24\\
25	20\\
25	25\\
};
\end{axis}
\end{tikzpicture}%
 &
%
%
\begin{tikzpicture}

\begin{axis}[%
width=\iwidth,
height=\iheight,
at={(0\iwidth,0\iheight)},
scale only axis,
xmin=0,
xmax=26,
xtick={\empty},
y dir=reverse,
ymin=0,
ymax=26,
ytick={\empty},
axis background/.style={fill=white}
]
\addplot [color=black,mark size=1pt,only marks,mark=*,mark options={solid},forget plot]
  table[row sep=crcr]{%
1	1\\
1	2\\
2	1\\
2	2\\
2	3\\
3	2\\
3	3\\
3	4\\
4	3\\
4	4\\
4	5\\
5	4\\
5	5\\
6	6\\
6	7\\
7	6\\
7	7\\
7	8\\
8	7\\
8	8\\
8	9\\
9	8\\
9	9\\
9	10\\
10	9\\
10	10\\
11	11\\
11	12\\
12	11\\
12	12\\
12	13\\
13	12\\
13	13\\
13	14\\
14	13\\
14	14\\
14	15\\
15	14\\
15	15\\
16	16\\
16	17\\
17	16\\
17	17\\
17	18\\
18	17\\
18	18\\
18	19\\
19	18\\
19	19\\
19	20\\
20	19\\
20	20\\
21	21\\
21	22\\
22	21\\
22	22\\
22	23\\
23	22\\
23	23\\
23	24\\
24	23\\
24	24\\
24	25\\
25	24\\
25	25\\
};
\end{axis}
\end{tikzpicture}%
    \end{tabular}
  \end{center}
\caption{Sparsity
patterns of the (approximated) Hessians of the distance function, smoothness
regularizer, its first and second term, and of the Hessian of the penalty function (from left to right) for the $e_1$-staggered discretization (top row) and the nodal
discretization used
in~\cite{RuthottoEtAl2013hysco,RuthottoEtAlPMB2012} (bottom row) for the
$4\times 4$ example shown in \cref{fig:compare_grids}. Note that the
Hessian of the distance term, the first term of the regularizer and the
Hessian of the penalty function are separable with respect to image
columns for the face-staggered discretization. Coupling is only introduced by the second term of the Hessian
of the smoothness regularizer.}
\label{fig:compare_sparsity}
\end{figure}
The optimization methods presented in the following section exploit the sparsity structure of the Hessian matrix. Due to the choice of the average operator $\bfA$ and the short finite difference operator $\bfD_1$, the Hessian of the distance function has a block-diagonal structure with
tridiagonal blocks of size $(m+1)\times(m+1)$; see also \cref{fig:compare_sparsity}. Thus, minimizing the distance term would decouple into several smaller optimization problems. 
The Hessian of the regularizer, $\nabla^2S$, is a discrete version of the negative Laplacian on
$\Omega$ with homogeneous Neumann boundary conditions and has a banded
structure. The term $\bfD_1^\top\bfD_1$ has the same block-diagonal
structure with tridiagonal blocks of size $(m+1)\times(m+1)$ as
$\nabla^2D$. So the coupling introduced by the regularizer comes from the
terms $\bfD_2^\top\bfD_2$ and $\bfD_3^\top\bfD_3$ only.
The Hessian of the penalty, $\nabla^2P$, has the same block-diagonal
structure with tridiagonal blocks as $\nabla^2D$.
\Cref{fig:compare_grids,fig:compare_sparsity} show a
comparison of the proposed face-staggered discretization and the nodal
discretization used in~\cite{RuthottoEtAl2013hysco,RuthottoEtAlPMB2012}
as well as the resulting sparsity patterns of the terms of the
approximate Hessians $\bfH_J$ and $\bfH_{J_{\rm GN}}$.

\section{Numerical Optimization} 
\label{sec:numerical_optimization}
In this section, we propose two efficient iterative methods for solving the discretized constrained optimization problem~\cref{eq:discprob}. Both methods exploit the separability achieved by
the face-staggered discretization and are scalable in the sense that
their complexity grows \emph{linearly} or \emph{essentially linearly} with
the number of unknowns in the discrete optimization problem.
In \cref{sub:gauss_newton_pcg} we  derive an inexact Gauss-Newton method
with a novel parallel block-Jacobi preconditioner and prove its
convergence. In \cref{sub:admm} we use
the Alternating Direction Method of Multipliers (ADMM) to decouple the
optimization into small subproblems that can be solved efficiently and
in parallel. Traditional convergence results for ADMM such as in
\cite{BoydEtAl2011} do not hold here, due to the non-convexity of the
objective function $J(\bfb)$. Thus, in \cref{sub:admmconv} we prove
convergence results for ADMM specific to problem~\cref{eq:discprob}. Finally, in \cref{sub:multilevel} we describe the coarse-to-fine multilevel strategy used in our experiments. Throughout this section, we denote the iteration counter by superscripts. 

\subsection{Gauss-Newton-PCG} 
\label{sub:gauss_newton_pcg}
The first approach is an inexact Gauss-Newton method (GN-PCG) for solving the
penalty formulation~\cref{eq:JGN}.
Starting with
$\bfb^0\equiv 0$, the $k$-th step of the iteration reads
\begin{equation}\label{eq:GNstep}
  \bfb^{k+1} = \bfb^k - \lambda^k
  \bfH_{J_{\rm GN}}(\bfb^k)^{-1}\nabla J_{\rm GN}(\bfb^k),
\end{equation}
where $\bfH_{J_{\rm GN}}(\bfb^k)$ and $\nabla J_{\rm GN}(\bfb^k)$ are computed according
to~\cref{eq:dJGN} and $\lambda^k$ is a step-size satisfying the
Wolfe conditions \cite[Ch. 3]{NocedalWright2006}.
This step-size choice guarantees the convergence of the
GN-PCG method to a stationary point even if the linear system in~\cref{eq:GNstep} is solved only approximately. More
precisely, instead of~\cref{eq:GNstep} we consider the iterative scheme
\begin{equation}\label{eq:inexactGNstep}
  \bfb^{k+1} = \bfb^k + \lambda^k\bfd^k
\end{equation}
with $\lambda^k$ chosen as before and step directions $\bfd^k$
satisfying
\begin{equation}\label{eq:inexactGNsolve}
\left\|\nabla J_{\rm GN}(\bfb^k) + \bfH_{J_{\rm GN}}(\bfb^k)\bfd^k \right\|\leq \eta \left\|\nabla
J_{\rm GN}(\bfb^k)\right\|,
\end{equation}
where $0<\eta<1$ is also known as a forcing parameter and controls the accuracy of the PCG method. In our numerical experiments, we choose $\eta=10^{-1}$.
\begin{theorem}
  For any feasible starting guess $\bfb^0$ with $J_{\rm GN}(\bfb^0)<\infty$ and sufficiently small $\eta$, the iterative scheme~\cref{eq:inexactGNstep} converges to a stationary point of $J_{\rm GN}(\bfb)$, i.e.,  $\left\|\nabla J_{\rm GN}(\bfb^k)\right\|\rightarrow 0$ as $k\rightarrow\infty$.
\end{theorem}
The proof of the above result is divided into two Lemmas. First, we
verify that the assumptions of Zoutendijk's result are satisfied for our
problem; see, e.g., \cite[Thm. 3.2]{NocedalWright2006}.
\begin{lemma} \label{lemma1}
	Let $\bfb^0$ be a starting guess with $J_{\rm GN}(\bfb^0) < \infty$ and
	let $S_0 \subseteq \Sigma$ be the sub-level set
	\begin{equation*}
		S_0 = \left\lbrace\,\bfb\in\Sigma\,:\,J_{\rm GN}(\bfb) \leq J_{\rm GN}(\bfb^0) \right\rbrace.
	\end{equation*} 
	Then, the objective function in~\cref{eq:JGN} satisfies
\end{lemma}
\begin{enumerate}
	\item $J_{\rm GN}(\bfb) \geq 0$  for all $\bfb \in \Sigma$.
	\item $J_{\rm GN}$ is continuously differentiable on a
	  neighborhood $S_1$ of $S_0$.
	\item $\nabla J_{\rm GN}$ is Lipschitz continuous on $S_1$.
\end{enumerate}
\begin{proof}
	The above properties follow immediately, since $S_0 \subseteq
	\Sigma$ is compact and $J_{\rm GN}$ is twice continuously differentiable
	on any open set $S$ with $S \subseteq
	\lbrace\,\bfb\in\R^n\,:\,J_{\rm GN}(\bfb)<\infty\,\rbrace$.
	Choosing open neighborhoods $S_1$ and $S_2$ such that
	$S_0\subseteq S_1\subseteq \overline{S_1}\subseteq S_2\subseteq
	\lbrace\,\bfb\in\R^n\,:\,J_{\rm GN}(\bfb)<\infty\,\rbrace$ where $S_1$
	is bounded and $\overline{S_1}$ denotes its closure proves the
	claim.
      \end{proof}
Due to the above lemma and the fact that $\lambda^k$ and $\bfd^k$
satisfy the Wolfe conditions, Zoutendijk's result holds and in particular
\begin{equation*}
	\sum_{k\geq 0} \cos^2 \theta_k \left\| \nabla J_{\rm GN}(\bfb^k)\right\|^2 < \infty,
\end{equation*}
where $\theta_k$ denotes the angle between $-\nabla J_{\rm GN}(\bfb^k)$ and $\bfd^k$. To show that $\nabla J_{\rm GN}(\bfb^k) \to 0$ as $k \to \infty$, we show in the following that $\cos\theta_k$ is uniformly bounded away from zero.
\begin{lemma}
	For all $k$, $\bfd^k$ is a descent direction and there exists $M > 0$ such that 
	\begin{equation*}
		\cos\theta_k \geq M > 0.
	\end{equation*}
\end{lemma}
\begin{proof}
	First, we show that there exist universal constants $0<M_1\leq M_2<\infty$ such that
	\begin{equation}\label{eq:Hbound}
	   M_1 \leq \bfx^\top\bfH_{J_{\rm GN}}(\bfb)\bfx\leq M_2, \quad \text{ for all} \quad \bfx \in \CB_{\|\cdot\|} \quad \text{ and } \quad  \bfb\in S_0,
	\end{equation}
	where $S_0$ is defined in Lemma~\ref{lemma1} and
	\begin{equation*}
		\CB_{\|\cdot\|} = \left\lbrace\,\bfu \in \R^n\,:\,  \| \bfu\| = 1
		\right\rbrace.
	\end{equation*}
	Since $S_0$ and $ \CB_{\|\cdot\|}$ are compact and $ \bfb
	\mapsto \bfH_{J_{\rm GN}}(\bfb)$ is continuous on $S_0$, the mapping
	\begin{equation*}
		(\bfb,\bfx) \mapsto \bfx^\top \bfH_{J_{\rm GN}}(\bfb) \bfx
	\end{equation*}
	attains a global maximum, $M_2 < \infty$, and a global minimum,
	$M_1 \geq \gamma > 0$, on $S_0 \times \CB_{\|\cdot\|}$. The
	lower bound on $M_1$ follows from the fact that the first three
	terms of $\bfH_{J_{\rm GN}}$ are positive semi-definite and
	$\gamma>0$; see~\cref{eq:dJGN}.
	Combining \cref{eq:Hbound,eq:inexactGNsolve} we obtain
	\begin{equation*}
	  \cos\theta_k = -\frac{\nabla J(\bfb^k)^\top\bfd^k}{\left\|\nabla J(\bfb^k)\right\|\left\|\bfd^k\right\|} \geq \frac{M_1-\eta M_2}{(1+\eta)M_2} = M>0
	\end{equation*}
	for any $0<\eta<\frac{M_1}{M_2}$.
	Finally, as in \cite{NocedalWright2006}, the lower bound on $\cos\theta_k$ implies that $\bfd^k$
	is a descent direction.
\end{proof}
\begin{proof}[Proof of Theorem 1]
	The previous Lemmas verify the assumptions for Zoutendijk's
	result and show that the angle between the negative gradient and the search direction is bounded away from 90 degrees. The first one implies
	\begin{equation*}
		\sum_{k\geq 0} \cos^2 \theta_k \left\| \nabla J_{\rm GN}(\bfb^k)\right\|^2 < \infty.
	\end{equation*}
	Since $\cos\theta_k \geq M >0$ we see that $\nabla {J_{\rm GN}}(\bfb^k) \to 0$ as $k \to \infty$.
\end{proof}
As the problem size is quite large in 3D applications we use a
Preconditioned Conjugate Gradient (PCG)
method for approximately solving~\cref{eq:GNstep}; see~\cite{HestenesStiefel1952} for the original work on CG and, e.g., ~\cite[Ch.9]{Saad2003} for an introduction to preconditioning. The performance of
the method depends on the spectral properties of $\bfH_{J_{\rm
GN}}$ and in particular on the clustering of its eigenvalues. 
It is well known that penalty and barrier methods can lead to severe ill-conditioning of the associated Hessians; see, e.g.,~\cite[Ch. 17]{NocedalWright2006}.
In our application, the penalty function and its derivatives~\cref{eq:dphi} grow to infinity as the magnitude of the first partial derivative of the inhomogeneity approaches one by design. Thus, effective preconditioners are required when correcting highly distorted data.
\par
We show in this paper that effective preconditioning can be achieved by
exploiting the sparsity structure of $\bfH_{J_{\rm GN}}(\bfb^k)$. 
To be precise, we propose using the block-Jacobi preconditioner
\begin{equation}\label{eq:Pblock}
  \bfP_{\textrm{block}} (\bfb^k)=
  \bfH_D(\bfb^k) +\alpha\left(\bfD_1^\top\bfD_1+\bfM\right)+ \nabla^2 P(\bfb^k),
\end{equation}
where $\bfM$ is the matrix containing only the diagonal of the last two terms in the Hessian of the regularizer, i.e., 
$\bfD_2^\top\bfD_2+\bfD_3^\top\bfD_3$. In the following section, we
compare the performance of this preconditioner to the Jacobi-preconditioner $\bfP_{\textrm{Jac}}(\bfb^k)$ containing only
the diagonal of $\bfH_{J_{\rm GN}}(\bfb^k)$, and the symmetric Gauss-Seidel
preconditioner
\begin{equation}\label{eq:Psgs}
  \bfP_{\textrm{SGS}}(\bfb) =
  (\bfL^k+\bfP_{\textrm{Jac}}(\bfb^k))\bfP_{\textrm{Jac}}^{-1}(\bfb^k)(\bfU^k+\bfP_{\textrm{Jac}}(\bfb^k)),
\end{equation}
where $\bfL^k$ and $\bfU^k$ are the strictly lower and upper triangular part of
$\bfH_{J_{\rm GN}}(\bfb^k)$, respectively; see~\cite[Ch. 10]{Saad2003} for more details.
\par
Solving a linear system with matrix $\bfP_{\textrm{block}}(\bfb^k)$ can be broken down into 
solving $m^2$ linear systems of size $m+1$, which can be done in parallel. 
Furthermore, solving each $m+1 \times m+1$ system is of linear complexity due to its tridiagonal structure.
Therefore, the computational complexity of the preconditioner grows linearly with
respect to the number of $x_1$-faces, i.e.,  $\mathcal{O}\left(n\right)$. 
Thus, the asymptotic complexity of $\bfP_{\textrm{block}}(\bfb^k)$ is the same as the one for Jacobi and Gauss-Seidel preconditioners applied to the sparse Hessian.
However, the effective runtimes may vary considerably between the preconditioners, due to the degree of parallelism.
While the Jacobi preconditioner requires only component-wise division and is thus the cheapest and fastest to compute, 
the Gauss-Seidel preconditioner is not straightforward to parallelize
and therefore, in our experiments, the most costly per iteration;
see~\cref{sec:experiments}.  
\par
We use the criteria suggested in~\cite{Modersitzki2009} for stopping the Gauss-Newton method. 
The iteration is stopped when prescribed tolerances for the relative
changes in the objective function value, the
iterate, and the norm of the gradient are reached, i.e., if all of the following three conditions hold
\begin{align}
 \left\lvert J_{\rm GN}(\bfb^{k+1})-J_{\rm GN}(\bfb^k)\right\rvert &\leq \epsilon_{\rm obj} \left(1+|J_{\rm GN}(\bfzero)|\right) \label{eq:stopping1}\\
 \left\|\bfb^{k+1}-\bfb^k\right\| &\leq \epsilon_{\rm iter} \left(1+\left\|\bfb^k\right\|\right) \\
 \left\| \nabla J_{\rm GN}(\bfb^{k+1})\right\| &\leq \epsilon_{\rm grad} \left(1+|J_{\rm GN}(\bfzero)|\right),
 \end{align}
or a maximum number of iterations is
exceeded. In our later experiments the maximum number of iterations is set to 10, which is the default setting in~\cite{RuthottoEtAl2013hysco}.
As common in inverse problems, we solve~\eqref{eq:JGN} only to a relatively low accuracy. In our numerical experiments we choose $\epsilon_{\rm obj}=10^{-3}$, $\epsilon_{\rm iter} = \epsilon_{\rm grad} = 10^{-2}$.

\subsection{ADMM} 
\label{sub:admm}
The second approach we consider is based on the Alternating Direction
Method of Multipliers~(ADMM).
Originally developed in the mid 1970s, ADMM recently received a lot of attention in many data science and imaging applications; see, e.g.,~\cite{GabayMercier1976} and the recent surveys~\cite{BoydEtAl2011,Eckstein2012,Glowinski2014}.
The key idea in our case is to split the terms of the objective function~\cref{eq:discprob} that are separable with respect to columns of the images from those terms that couple across image columns. 
To this end, let $\bfz$ be a new artificial variable and split the objective function $J$ into
\begin{equation}\label{eq:fg}
f(\bfb) = D(\bfb) + \frac{\alpha h^3}{2}\left\|\bfD_1\bfb\right\|^2\quad\text{and}\quad g(\bfz) = \frac{\alpha h^3}{2}\left(\left\|\bfD_2\bfz\right\|^2+\left\|\bfD_3\bfz\right\|^2\right).
\end{equation}
Then we consider
\begin{equation}\label{eq:admmprob} 
	  \min_{\bfb,\bfz \in \R^n}  \iota_C(\bfb) + f(\bfb) + g(\bfz)
	  \quad \textrm{subject to} \quad\bfb=\bfz,
\end{equation}
where we encode the linear inequality constraints into the objective function, using the indicator function $\iota_C(\bfb)$ of the compact and convex set $C = \lbrace\,\bfb\in\Sigma\,:\,-1\leq\bfD_1\bfb\leq 1\,\rbrace$,  taking the value $0$, whenever $\bfb\in C$, and
$\infty$ otherwise. Thus,~\cref{eq:admmprob} is equivalent to~\cref{eq:discprob}.
For an augmentation parameter $\rho>0$, whose choice is discussed below, we aim at finding a stationary point of the augmented Lagrangian of~\cref{eq:admmprob} 
\begin{equation}\label{eq:Lp}
  \mathcal{L}_\rho(\bfb,\bfz,\bfy) = \iota_C(\bfb)+f(\bfb)+g(\bfz) +
  \bfy^\top(\bfb-\bfz)+\frac{\rho h^3}{2}\left\|\bfb-\bfz\right\|^2,
\end{equation}
where $\bfy \in \R^n$ is the Lagrange multiplier associated with the equality
constraint in~\cref{eq:admmprob}. 
 The idea of ADMM is to update $\bfb$, $\bfz$, and $\bfy$ in an alternating fashion
\begin{align}
  \bfb^{k+1} &= \mathop{\argmin}\limits_{\bfb\in\R^n} \mathcal{L}_\rho(\bfb,\bfz^k,\bfy^k), \label{eq:UadmmStep1}\\
  \bfz^{k+1} &= \mathop{\argmin}\limits_{\bfz\in\R^n} \mathcal{L}_\rho(\bfb^{k+1},\bfz,\bfy^k), \label{eq:UadmmStep2}\\
  \bfy^{k+1} &= \bfy^k+\rho h^3(\bfb^{k+1}-\bfz^{k+1}). \label{eq:UadmmStep3}
\end{align}
Introducing the scaled dual variable $\bfu = \bfy/(\rho h^3)$, we obtain the following iteration
\begin{align}
  \bfb^{k+1} &= \mathop{\argmin}\limits_{\bfb\in\mathbb{R}^n}
  \left\{ \iota_C(\bfb)+f(\bfb) +\frac{\rho h^3}{2}\left\|\bfb-\bfz^k+\bfu^k\right\|^2\right\},  \label{eq:admmStep1}\\
  \bfz^{k+1} &= \mathop{\argmin}\limits_{\bfz\in\mathbb{R}^n}
  \left\{  g(\bfz)+\frac{\rho
  h^3}{2}\left\|\bfb^{k+1}-\bfz+\bfu^k\right\|^2\right\}, \label{eq:admmStep2}\\
  \bfu^{k+1} &= \bfu^k + \bfb^{k+1}-\bfz^{k+1}. \label{eq:admmStep3}
\end{align}
The first subproblem, updating $\bfb$, is a non-convex constrained
optimization problem and approximately solved using sequential quadratic programming (SQP) as described below. For the discretization derived in the previous section, this problem is separable with respect to the columns in the image and thus can be further broken down into $m^2$ separate steps.
In each iteration of the SQP method, we form a quadratic approximation of the objective function and solve the resulting quadratic program (QP) using the active set method with Schur complement solver described in~\cite[Ch. 16]{NocedalWright2006}. The general form of the QP is
\begin{equation}\label{eq:qp}
  \min\limits_\bfx
  \frac{1}{2}\bfx^\top\bfG\bfx+\bfc^\top\bfx \quad\textrm{subject to}\quad
  \bfA\bfx\geq \bfd,
\end{equation}
where, in our case, $\bfG$ is the respective Hessian approximation, $\bfc$ is the gradient,
$\bfD_1$ and $-\bfD_1$ are stacked into $\bfA$, and $\bfd$ is a vector
of all negative ones.
Let $\bfx^k$  be the current iterate in this QP and $I$ be a subset of
component-indices describing the active set. Further, let $\bfA_I$ be 
the matrix containing the rows of $\bfA$ associated with active
constraints and let $\bfd_I$ be the corresponding right-hand-side. Then we obtain an update direction $\bfp$ for $\bfx^k$ and the Lagrange multiplier $\boldsymbol{\lambda}$ of the active constraints by solving
\begin{equation}\label{eq:schurasqp}
\begin{pmatrix} \bfG & \bfA_I^\top\\\bfA_I &
\mathbf{0}\end{pmatrix}\begin{pmatrix}\bfp \\
-\boldsymbol{\lambda}\end{pmatrix} =
\begin{pmatrix}-\bfc-\bfG\bfx^k\\\bfd_I-\bfA_I\bfx^k\end{pmatrix} =:
\begin{pmatrix} \bfg \\ \bfh\end{pmatrix}
\end{equation}
using the Schur complement, i.e., by setting
\begin{equation*}
  \boldsymbol{\lambda} =
  -\left(\bfA_I\bfG^{-1}\bfA_I^\top\right)^{-1}\left(\bfA_I\bfG^{-1}\bfg-\bfh\right)\quad \text{ and } \quad
  \bfp= \bfG^{-1}\left(\bfA_I^\top\boldsymbol{\lambda}+\bfg\right).
\end{equation*}
Now three cases can occur: First, if $\bfp = \mathbf{0}$ and
$\boldsymbol{\lambda}\geq \mathbf{0}$ component-wise, then $\bfx^k$
is the global solution to~\cref{eq:qp}. Second, if $\bfp =\mathbf{0}$ but some components of
$\boldsymbol{\lambda}$ are negative, we remove up to one constraint with negative
Lagrange multiplier per block in $\bfA$ from the active set and repeat the above
computations. Third, if $\bfp\neq\mathbf{0}$, compute a step
length $\lambda^k$ such that $\bfx^k+\lambda^k\bfp$ is feasible but at least one
additional constraint has become active. Then all of these constraints
are added to the active set.
Linear independence of the active constraints is guaranteed since the partial derivative at an active voxel can be either $1$ or $-1$.
Note that the Schur complement involves $\bfG^{-1}$ and thus an
efficient method to solve linear systems with matrix $\bfG$ is needed.
In our application, the chosen discretization allows for fast inversions. 
\par
More precisely, for the update of $\bfb$, we
have 
\begin{equation*}
\bfc = \nabla D(\bfb^k) + \alpha h^3 \bfD_1\bfb^k + \rho h^3 (\bfb^k-\bfz^k+\bfu^k)
\end{equation*}
and
\begin{equation*}
	\bfG = \bfH_D\left(\bfb^k\right) + \alpha h^3 \bfD_1^\top\bfD_1+\rho h^3
	\bfI\left(n\right),
\end{equation*}
which has the
before mentioned block-diagonal structure with tridiagonal blocks of size $m +1$.
Therefore, linear systems involving $\bfG$ can be solved in parallel by a direct method with a linear complexity
of $\mathcal{O}\left(n\right)$.
\par
The second subproblem, updating $\bfz$, is an unconstrained strictly convex
quadratic optimization problem with a structured, symmetric positive
definite Hessian
\begin{equation*}
  \widetilde{\bfG}=\alpha h^3 \left(\bfD_2^\top\bfD_2+\bfD_3^\top\bfD_3\right) + \rho h^3
	\bfI\left(n\right)
\end{equation*}
and has the closed form solution
\begin{equation*}
  \bfz^{k+1} = \widetilde{\bfG}^{-1}\left( \rho h^3(\bfb^{k+1}+\bfu^k)\right).
\end{equation*}
The matrix $\widetilde{\bfG}$ can be reordered into a block-diagonal matrix with $m+1$ blocks, whose blocks are matrices of size $m^2 \times m^2$. 
Each block is a discretization of the negative two-dimensional Laplacian with homogeneous Neumann boundary conditions on a regular mesh and thus is a Block-Toeplitz-Toeplitz-Block~(BTTB) matrix. 
Hence, $\widetilde{\bfG}$ can be diagonalized using two-dimensional Discrete Cosine Transforms~(DCT); see~\cite[Ch. 4]{HansenNagyOLeary2006}. 
More specifically, denote by $\bfC(m)\in\R^{m\times m}$ the
one-dimensional DCT of size $m$.
Then we have
\begin{equation}\label{dctdiag}
  \til{\bfD}(m,h)^\top\til{\bfD}(m,h) = \bfC(m)^\top\boldsymbol{\Lambda}(m,h)\bfC(m)
\end{equation}
for some diagonal matrix $\boldsymbol{\Lambda}(m,h)$. Combining~\cref{eq:d2d3} and~\cref{dctdiag}, we immediately obtain
\begin{equation}
  \bfD_2^\top\bfD_2+\bfD_3^\top\bfD_3 =
  \bfC^\top\boldsymbol{\Lambda}\bfC, \text{ where } \bfC = \bfC(m)\kron\bfC(m)\kron\bfI(m+1)
\end{equation}
is the two-dimensional DCT along the $e_2$- and $e_3$-direction, and 
\begin{equation*}
  \boldsymbol{\Lambda} = \bfI(m)\kron  \boldsymbol{\Lambda}(m,h)\kron\bfI(m+1)
  + \boldsymbol{\Lambda}(m,h)\kron\bfI(m)\kron\bfI(m+1)
\end{equation*}
is a diagonal matrix that only needs to be computed once for every fixed
discretization $(m,h)$. Finally, since $\bfC$ is orthogonal, we
obtain 
\begin{equation*}
  \widetilde{\bfG} = \bfC^\top\left(\alpha h^3\boldsymbol\Lambda+\rho
h^3\bfI\left(n\right)\right)\bfC.
\end{equation*}
Thus, the second ADMM step, updating $\bfz$, requires $m+1$ two-dimensional DCTs, which is of complexity
$\mathcal{O}\left( (m+1) (m^2 \log m^2)\right)$, followed by a
diagonal solve,
and finally $m+1$ inverse two-dimensional DCTs. Again this is a direct solve and does not
require the use of iterative methods.

\bigskip

For the ADMM algorithm, we use the stopping criteria proposed
in~\cite[Sec. 3]{BoydEtAl2011}, which are also justified by our convergence analysis in~\cref{sub:admmconv}. We stop when the norm of the primal and dual residual satisfy
\begin{equation}\label{eq:admmStop}
	\left\|\bfb^{k+1}-\bfz^{k+1}\right\| \leq \epsilon_{\rm pri} \quad\text{
	and }\quad \rho h^3 \left\|\bfz^{k}-\bfz^{k+1}\right\| \leq \epsilon_{\rm
	dual},
\end{equation}
where $\epsilon_{\rm pri}$ and $\epsilon_{\rm dual}$ are computed
exactly as~(3.12) in~\cite{BoydEtAl2011} using a combination of an absolute
and a relative tolerance
\begin{align*}
  \epsilon_{\rm pri} &= \sqrt{n}\ \epsilon_{\rm
  abs}+\epsilon_{\rm
  rel}\max\left\lbrace\left\|\bfb^k\right\|,\left\|\bfz^k\right\|\right\rbrace\\
  \epsilon_{\rm dual} &= \sqrt{n}\ \epsilon_{\rm
  abs}+\epsilon_{\rm rel}\rho h^3\left\|\bfu^k\right\|.
\end{align*}
In our numerical experiments, we choose $\epsilon_{\rm abs}= \epsilon_{\rm rel} = 2\cdot 10^{-1}$.
\par

\subsection{Convergence of ADMM}
\label{sub:admmconv}
It is important to stress that the first subproblem in our ADMM algorithm,~\cref{eq:admmStep1}, is non-convex
and thus the traditional convergence results for ADMM do not hold.
However, ADMM can be considered a local optimization method and has been successfully applied to non-convex problems in other applications; see~\cite[Ch.9]{BoydEtAl2011} for some examples. 
Recently, convergence results have been established under some modest
conditions on the functions involved; see, for
example,~\cite{WangEtAl2015}. In the following, we show convergence of ADMM 
for the specific problem at hand. Using the smoothness of our problem we obtain a simplified, but less general, convergence proof as compared to~\cite{WangEtAl2015}. 
We first note that the functions $f$ and $g$ in~\cref{eq:fg} are twice continuously differentiable. Further, $\nabla f$ is Lipschitz continuous over $C$ and $\nabla g$ is Lipschitz continuous over $\R^n$. We denote the corresponding Lipschitz constants by $L_f$ and $L_g$ respectively. The relation between these Lipschitz constants and the augmentation parameter, $\rho$, is crucial in the convergence analysis. Throughout this section we use both the unscaled ADMM formulation~\cref{eq:UadmmStep1,eq:UadmmStep2,eq:UadmmStep3} as well as the equivalent scaled formulation \cref{eq:admmStep1,eq:admmStep2,eq:admmStep3}, whichever is more convenient. Recall that $\bfu^k = \bfy^k/(\rho h^3)$ is the scaled Lagrange multiplier.
The main result of this section is the following.
\begin{theorem}\label{thm:ourconvadmm}
For each $\rho > \frac{2}{h^3} \max\lbrace L_f, L_g\rbrace$, the sequence of iterates $\{\bfb^k,\bfz^k,\bfy^k\}$ generated by the ADMM algorithm~\cref{eq:UadmmStep1,eq:UadmmStep2,eq:UadmmStep3} converges subsequentially and each limit point $\{\bfb^*,\bfz^*,\bfy^*\}$ is a stationary point of the Lagrangian $\CL_\rho$ in~\cref{eq:Lp}.
\end{theorem}
As in~\cite{WangEtAl2015}, our proof is based on the following three properties.
\begin{theorem}\label{theo:principles}
	For each $\rho > \frac{2}{h^3} \max\{L_f, L_g\} $, the sequence of iterates $\lbrace \bfb^k,\bfz^k,\bfy^k \rbrace$ generated by the ADMM algorithm~\cref{eq:UadmmStep1,eq:UadmmStep2,eq:UadmmStep3} has the following properties:
	\begin{description}
		\item[(P1)] The iterates $\lbrace\bfb^k,\bfz^k,\bfy^k\rbrace$ are well-defined and bounded, and $\{\CL_\rho(\bfb^k,\bfz^k,\bfy^k)\}_k$ is bounded below.
		\item[(P2)] The value of the Lagrangian decreases sufficiently fast, meaning there exists a constant $C_1>0$ such that
		\[	\CL_\rho(\bfb^k,\bfz^k,\bfy^k)-\CL_\rho(\bfb^{k+1},\bfz^{k+1},\bfy^{k+1})\geq C_1 \left(\left\|\bfb^k-\bfb^{k+1}\right\|^2+\left\|\bfz^k-\bfz^{k+1}\right\|^2\right).\]
		\item[(P3)] There exists a constant $C_2>0$ and subgradients $\bfd^{k+1}\in\partial\CL_\rho(\bfb^{k+1},\bfz^{k+1},\bfy^{k+1})$ such that
		\[\left\|\bfd^{k+1}\right\| \leq C_2\left( \left\|\bfb^k-\bfb^{k+1}\right\|+\left\|\bfz^k-\bfz^{k+1}\right\| \right). \]
	\end{description}
\end{theorem}
We split the proof of \cref{theo:principles} into several lemmas.
\begin{lemma}\label{lemm:bzupdate}
	For $\rho > 0$ the subproblems~\cref{eq:admmStep1} and~\cref{eq:admmStep2} have at least one solution for all $k \in \N$.
\end{lemma}
\begin{proof}
	For arbitrary but fixed $k$, the first subproblem, the update $\bfb^k \to \bfb^{k+1}$, consists of minimizing the function
	\begin{equation*}
		f(\bfb) +\frac{\rho h^3}{2}\left\|\bfb-\bfz^k+\bfu^k\right\|^2
	\end{equation*}
	over the compact and convex set $C$. On this set $f$ is twice continuously differentiable, so the problem is well-defined, i.e., there exists a global solution.
	The second subproblem, the update $\bfz^k \to \bfz^{k+1}$, consists of minimizing 
	\begin{equation*}
		g(\bfz) +\frac{\rho
		  h^3}{2}\left\|\bfb^{k+1}-\bfz+\bfu^k\right\|^2
	\end{equation*}
	over $\bfz \in \R^n$. This problem is well defined, since $g$ is a convex quadratic function and $\rho>0$, which renders the overall objective function strictly convex. Thus, there exists a unique global minimizer.
\end{proof}
Next we show that the augmented Lagrangian decreases sufficiently after the first ADMM step~\eqref{eq:UadmmStep1}.
\begin{lemma}\label{lem:bdescent}
	If $\rho \geq \frac{2}{h^3} L_f$, then the update $\bfb^k \to \bfb^{k+1}$ does not increase the value of the augmented Lagrangian. More precisely, for $\rho>0$
	\begin{equation*}
		\CL_\rho(\bfb^k,\bfz^k,\bfy^k)-\CL_\rho(\bfb^{k+1},\bfz^k,\bfy^k)
		\geq	
		\left(\frac{\rho h^3}{2}-L_f\right) \left\|\bfb^k-\bfb^{k+1}\right\|^2.
	\end{equation*}
	\end{lemma}
\begin{proof}
 Denote $\bfd^{k+1} = \bfy^{k+1} + \rho h^3(\bfz^{k+1}-\bfz^k)$. From the optimality condition for~\cref{eq:UadmmStep1} we conclude that $\bfb^{k+1}$ satisfies
\begin{align}
 &\bfzero \in \nabla f(\bfb^{k+1}) + \partial\iota_C(\bfb^{k+1}) + \bfy^k+\rho h^3(\bfb^{k+1}-\bfz^k)\nonumber\\
 \Leftrightarrow \quad &\bfzero \in \nabla f(\bfb^{k+1}) + \partial\iota_C(\bfb^{k+1}) + \bfy^{k+1} + \rho h^3(\bfz^{k+1}-\bfz^k)\nonumber\\
  \Leftrightarrow \quad & -\left(\nabla f(\bfb^{k+1})+ \bfy^{k+1} + \rho h^3(\bfz^{k+1}-\bfz^k)\right)\in\partial\iota_C(\bfb^{k+1})\nonumber\\
  \Leftrightarrow \quad & \left(\nabla f(\bfb^{k+1}) + \bfd^{k+1}\right)^\top\left(\bfw-\bfb^{k+1}\right)\geq 0 \quad \forall \bfw\in C\label{eq:proxgrad}.
\end{align}
 Here, we used that the subgradient of an indicator function of a convex set $C$ is
\begin{equation*}
	\partial \iota_C(\bfb^{k+1}) = \left\{\,\bfd \in \R^n \,:\,\bfd^\top (\bfw - \bfb^{k+1}) \leq 0, \, \forall \bfw \in C\,\right\},
\end{equation*}
see also \cite[Ch. 8]{Rockafeller2009}.
 Denoting  $\CL_\rho^k=\CL_\rho(\bfb^k,\bfz^k,\bfy^k)$ and $\CL_\rho^{k+1}=\CL_\rho(\bfb^{k+1},\bfz^k,\bfy^k)$ we get
 \begin{align*}
 \CL_\rho^k-\CL_\rho^{k+1} & = f(\bfb^k)-f(\bfb^{k+1}) + {(\bfy^k)}^\top(\bfb^k-\bfb^{k+1})+\frac{\rho h^3}{2}\left( \left\|\bfb^k-\bfz^k\right\|^2-\left\|\bfb^{k+1}-\bfz^k\right\|^2\right)\\
 &= f(\bfb^k)-f(\bfb^{k+1}) + {(\bfd^{k+1})}^\top(\bfb^k-\bfb^{k+1})+\frac{\rho h^3}{2}\left\|\bfb^k-\bfb^{k+1}\right\|^2\nonumber\\
 &= f(\bfb^k)-f(\bfb^{k+1}) -\nabla f(\bfb^{k+1})^\top(\bfb^k-\bfb^{k+1}) \nonumber\\&\qquad +\left(\nabla f(\bfb^{k+1})+ {\bfd^{k+1}}\right)^\top(\bfb^k-\bfb^{k+1})+\frac{\rho h^3}{2}\left\|\bfb^k-\bfb^{k+1}\right\|^2\nonumber\\
 &\geq \left(\frac{\rho h^3}{2}-L_f\right)\left\|\bfb^k-\bfb^{k+1}\right\|^2\nonumber,
 \end{align*}
where we used~\cref{eq:proxgrad} and the Lipschitz continuity of $\nabla f$ in the last step.
\end{proof}
A similar result also holds for the second ADMM step~\eqref{eq:UadmmStep2}.
\begin{lemma}\label{lem:zdescent}
	For any $\rho>0$, the update $\bfz^k \to \bfz^{k+1}$ does not increase the value of the augmented Lagrangian. More precisely,
	\begin{equation*}
	\CL_\rho(\bfb^{k+1},\bfz^k,\bfy^k)-\CL_\rho(\bfb^{k+1},\bfz^{k+1},\bfy^k)\geq \frac{\rho h^3}{2}\left\|\bfz^k-\bfz^{k+1}\right\|^2.
	\end{equation*}
\end{lemma}
\begin{proof}
	From the optimality condition for~\cref{eq:UadmmStep2} we conclude that $\bfz^{k+1}$ satisfies
	\begin{equation*}
		0 = \nabla g(\bfz^{k+1}) - \bfy^k - \rho h^3 (\bfb^{k+1} - \bfz^{k+1}) = \nabla g(\bfz^{k+1}) - \rho h^3 \bfy^{k+1}
	\end{equation*}
	and therefore $\bfy^{k+1} = \nabla g(\bfz^{k+1})$.
	Similar to before, denoting $\CL_\rho^k=\CL_\rho(\bfb^{k+1},\bfz^k,\bfy^k)$ and $\CL_\rho^{k+1}=\CL_\rho(\bfb^{k+1},\bfz^{k+1},\bfy^k)$ we get
	\begin{align*}
	\CL_\rho^k-\CL_\rho^{k+1} & = g(\bfz^k)-g(\bfz^{k+1}) - {(\bfy^k)}^\top(\bfz^k-\bfz^{k+1}) + \frac{\rho h^3}{2}\left(\left\|\bfb^{k+1}-\bfz^k\right\|^2-\left\|\bfb^{k+1}-\bfz^{k+1}\right\|^2\right)\\
	& = g(\bfz^k)-g(\bfz^{k+1}) -{(\bfy^{k+1})}^\top(\bfz^k-\bfz^{k+1})+\frac{\rho h^3}{2}\left\|\bfz^k-\bfz^{k+1}\right\|^2\nonumber\\
	& = g(\bfz^k)-g(\bfz^{k+1}) -\nabla g(\bfz^{k+1})^\top (\bfz^k-\bfz^{k+1}) + \frac{\rho h^3}{2}\left\|\bfz^k-\bfz^{k+1}\right\|^2\nonumber\\
	&\geq \frac{\rho h^3}{2}\left\|\bfz^k-\bfz^{k+1}\right\|^2
	\end{align*}
	where we used the convexity of $g$ in the last step.
\end{proof}
We now show an analogous result for the third ADMM step~\eqref{eq:UadmmStep3}.
\begin{lemma}\label{lem:ydescent}
	For any $\rho>0$, the update $\bfy^k \to \bfy^{k+1}$ does not increase the value of the augmented Lagrangian, i.e.,
	\begin{equation*}
	\CL_\rho(\bfb^{k+1},\bfz^{k+1},\bfy^k)-\CL_\rho(\bfb^{k+1},\bfz^{k+1},\bfy^{k+1})\geq 0.
	\end{equation*}
\end{lemma}
\begin{proof}
	Denoting $\CL_\rho^k=\CL_\rho(\bfb^{k+1},\bfz^{k+1},\bfy^k)$ and $\CL_\rho^{k+1}=\CL_\rho(\bfb^{k+1},\bfz^{k+1},\bfy^{k+1})$ we get
	\begin{align*}
	\CL_\rho^k-\CL_\rho^{k+1} & = (\bfy^k-\bfy^{k+1})^\top(\bfb^{k+1}-\bfz^{k+1})
	 = \rho h^3\left\|\bfb^{k+1}-\bfz^{k+1}\right\|^2
	 \geq 0.
	\end{align*}
\end{proof}
Having established the above results, we can now verify that our problem satisfies the three properties in~\cref{theo:principles}.
\begin{proof}[Proof of \cref{theo:principles}]
	We note that (P2), the sufficient decrease, follows immediately from \cref{lem:bdescent,lem:zdescent,lem:ydescent} with
	$C_1 = \frac{\rho h^3}{2}-L_f > 0$.
	\begin{description}
		\item[(P1)] The well-definedness of the updates $\bfb^k \to \bfb^{k+1}$ and $\bfz^k \to \bfz^{k+1}$ and thus also the update $\bfy^k \to \bfy^{k+1}$ follows from~\cref{lemm:bzupdate}. We next show that the sequence $\{ \CL(\bfb^k,\bfz^k, \bfy^k) \}_k$ is bounded below. First, note that for each $k$
		\begin{equation*}
			f(\bfb^k) + g(\bfb^k) \geq 0.
		\end{equation*}
		Using that $\bfy^k = \nabla g(\bfz^k)$ for all $k$ and
		the Lipschitz continuity of $\nabla g$ we note that
		\begin{equation}\label{eq:ybound}
			\left\| \bfy^{k+1} - \bfy^k\right\| \leq L_g \left\| \bfz^{k+1} - \bfz^k\right\|.
		\end{equation}
		Combining this with $\rho h^3 \geq 2 L_g$ we obtain
		\begin{align}\label{eq:Lbound}
		\CL_\rho(\bfb^k,\bfz^k,\bfy^k) &= f(\bfb^k)+g(\bfz^k)+{(\bfy^k)}^\top(\bfb^k-\bfz^k)+\frac{\rho h^3}{2}\left\|\bfb^k-\bfz^k\right\|^2\\
		&= f(\bfb^k)+g(\bfb^k) + g(\bfz^k) - g(\bfb^k) - \nabla g(\bfz^k)^\top(\bfz^k-\bfb^k)+\frac{\rho h^3}{2}\left\|\bfb^k-\bfz^k\right\|^2\nonumber\\
		&\geq f(\bfb^k)+g(\bfb^k) + \left(\frac{\rho h^3}{2}-L_g\right)\left\|\bfb^k-\bfz^k\right\|^2\nonumber\\
		&\geq 0,\nonumber
		\end{align}
		which shows that $\left\{\CL_\rho(\bfb^k,\bfz^k,\bfy^k)\right\}_k$ is lower bounded by zero.		
		Finally, we need to show the boundedness of
		$\lbrace\bfb^k,\bfz^k,\bfy^k\rbrace$. For $\bfb^k$ this
		directly follows from the boundedness of $C$. From (P2)
		we know that $\CL_\rho(\bfb^k,\bfz^k,\bfy^k)$ is
		monotonically decreasing and therefore bounded above by
		$\CL_\rho(\bfb^0,\bfz^0,\bfy^0)$. Using this, the
		boundedness of $\lbrace\bfb^k\rbrace$, \cref{eq:Lbound},
		and the lower boundedness of $f(\bfb^k)+g(\bfb^k)$, we
		conclude the boundedness of $\lbrace\bfz^k\rbrace$.
		Finally, again using $\bfy^k = \nabla g(\bfz^k)$ and the
		Lipschitz continuity of $\nabla g$, we have
		\begin{align*}
		\textstyle \left\|\bfy^k\right\| &= \left\|\nabla g(\bfz^k)\right\| 
		\leq \left\|\nabla g(\bfz^k)-\nabla g(\bfzero)\right\| + \left\|\nabla g(\bfzero)\right\| \leq L_g \left\|\bfz^k\right\| + \left\|\nabla g(\bfzero)\right\|,
		\end{align*}
		which shows the boundedness of $\lbrace\bfy^k\rbrace$.%
		\item[(P3)] We need to bound the derivatives of $\CL_{\rho}$. First, we note that
		\begin{align*}
			\partial_\bfb \CL_\rho(\bfb^{k+1},\bfz^{k+1},\bfy^{k+1})  
			&= \nabla f(\bfb^{k+1}) + \partial\iota_C(\bfb^{k+1}) + \bfy^{k+1} + \rho h^3(\bfb^{k+1} - \bfz^{k+1})\\
			&= \nabla f(\bfb^{k+1}) + \partial\iota_C(\bfb^{k+1}) + \bfy^k + \rho h^3(\bfb^{k+1} - \bfz^k) \\&\qquad + \bfy^{k+1} - \bfy^k + \rho h^3(\bfz^k-\bfz^{k+1}).
		\end{align*}
		The optimality condition of~\cref{eq:UadmmStep1} implies that $\bfzero \in f(\bfb^{k+1}) + \partial\iota_C(\bfb^{k+1}) + \bfy^k + \rho h^3(\bfb^{k+1} - \bfz^k)$ and thus
		\begin{equation*}
		\bfy^{k+1} - \bfy^k + \rho h^3(\bfz^k-\bfz^{k+1}) \in \partial_\bfb \CL_\rho(\bfb^{k+1},\bfz^{k+1},\bfy^{k+1}), 
		\end{equation*}
		which is bounded by $(L_g+\rho h^3)\left\|\bfz^k-\bfz^{k+1}\right\|$ due to~\cref{eq:ybound}.
		Second, we note that
		\begin{align*}
		\left\|\partial_\bfz \CL_\rho(\bfb^{k+1},\bfz^{k+1},\bfy^{k+1})\right\| &= \left\|\nabla g(\bfz^{k+1}) -\bfy^{k+1} - \rho h^3(\bfb^{k+1}-\bfz^{k+1})\right\| \\
		& = \left\|\bfy^k-\bfy^{k+1}\right\| \leq L_g\left\|\bfz^k-\bfz^{k+1}\right\|.
		\end{align*}
		Finally, we note that
		\begin{align*}
		\left\|\partial_\bfy \CL_\rho(\bfb^{k+1},\bfz^{k+1},\bfy^{k+1}) \right\| &= \left\|\bfb^{k+1}-\bfz^{k+1}\right\|\\
		& = \left\|\frac{1}{\rho h^3}(\bfy^{k+1}-\bfy^k)\right\| \leq \frac{L_g}{\rho h^3}\left\|\bfz^k-\bfz^{k+1}\right\|.
		\end{align*}
		Therefore, setting $C_2 = \max\lbrace\frac{1}{2},3L_g\rbrace$ there exists  $\bfd^{k+1}\in\partial \CL_\rho(\bfb^{k+1},\bfz^{k+1},\bfy^{k+1})$ as claimed.
	\end{description}
\end{proof}
Finally, we conclude by proving the main result, which is done exactly as in \cite{WangEtAl2015} by using the properties  (P1)--(P3).
 \begin{proof}[Proof of \cref{thm:ourconvadmm}]
	In \cref{theo:principles} we have established that the
	properties (P1)--(P3) hold for the iterates generated
	by~\cref{eq:UadmmStep1,eq:UadmmStep2,eq:UadmmStep3}, provided
	that $\rho > \frac{2}{h^3} \max\lbrace L_f, L_g\rbrace$. From
	(P1) we know that the set of iterates
	$\lbrace\bfb^k,\bfz^k,\bfy^k\rbrace$ is bounded, so it has a
	convergent subsequence. We denote a limit point by
	$(\bfb^*,\bfz^*,\bfy^*)$. Also from (P1) we know that the
	sequence $\left\{\CL_\rho(\bfb^k,\bfz^k,\bfy^k)\right\}_k$ is bounded
	below. By~(P2) it is also monotonically and sufficiently
	decreasing and this implies
	$\left\|\bfb^k-\bfb^{k+1}\right\|\rightarrow 0$ and
	$\left\|\bfz^k-\bfz^{k+1}\right\|\rightarrow 0$. Finally, by (P3), we get
	that there exists a subgradient $\bfd^k\in\partial\CL_\rho(\bfb^k,\bfz^k,\bfy^k)$ with $\left\|\bfd^k\right\|\rightarrow 0$, which shows that $\bfzero\in\partial\CL_\rho(\bfb^*,\bfz^*,\bfy^*)$, and thus $(\bfb^*,\bfz^*,\bfy^*)$ is a stationary point.
\end{proof}
%
The lower bound for $\rho$ depends on the Lipschitz constants $L_f$ and $L_g$ of the distance and regularization function, respectively, as well as the voxel size $h$. The Lipschitz constants are commonly not available in practice. In our numerical experiments in \cref{sec:experiments} we use a modified version of the
adaptive augmentation parameter choice described in~\cite{BoydEtAl2011},
which ensures that the augmentation parameter $\rho$ remains larger than an
experimentally defined lower bound $\rho_{\rm min}$, which we chose equal for all steps in the multilevel optimization. We compare this
parameter choice method with a constant choice of the augmentation parameter
and the unmodified adaptive scheme proposed for convex problems in~\cite{BoydEtAl2011}.

\subsection{Multilevel Strategy} 
\label{sub:multilevel}
As common in image registration and also suggested in~\cite{RuthottoEtAlPMB2012}, we employ a multilevel approach for solving~\cref{eq:varProb}; see~\cite{Modersitzki2009} for details. In a nutshell, we start by solving a discrete version of~\cref{eq:varProb} on a relatively coarse grid.
 Then, we prolongate the estimated field inhomogeneity to a finer grid to serve as a starting guess 
for the next discrete optimization problem. On each level, the resolution of the image data is increased as well and the procedure is repeated until the desired resolution is achieved.
Apart from reducing computational costs on a coarse grid and obtaining
excellent starting guesses, multilevel approaches have been observed to
be more robust against local minima, which are less likely to occur in
the coarse grid discretization; see ~\cref{fig:multilevel} for an example.  
\begin{figure}[t]
	\scriptsize
	\renewcommand{\rottext}[1]{\rotatebox{90}{\hbox to 18 mm{\hss #1\hss}}}
	\newcommand{\image}[1]{\includegraphics[width=23mm,trim=62 140 62 20, clip=true]{./img/multilevel/#1}}
  \begin{center}
    \begin{tabular}{@{}c@{ }c@{ }cc@{ }cc@{ }c@{}}
		& \multicolumn{2}{c}{$32\times32$}
		& \multicolumn{2}{c}{$64\times64$}
		& \multicolumn{2}{c}{$128\times128$}\\
	\rottext{initial data}&
    \image{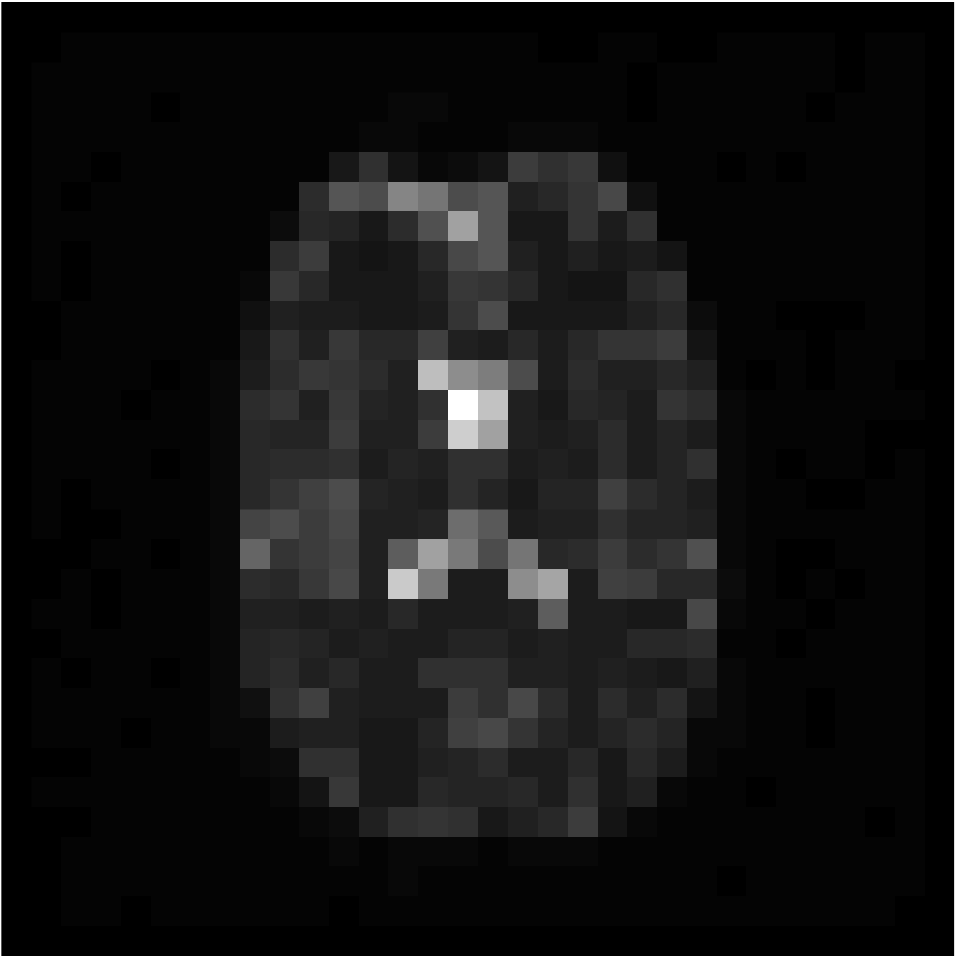} &
    \image{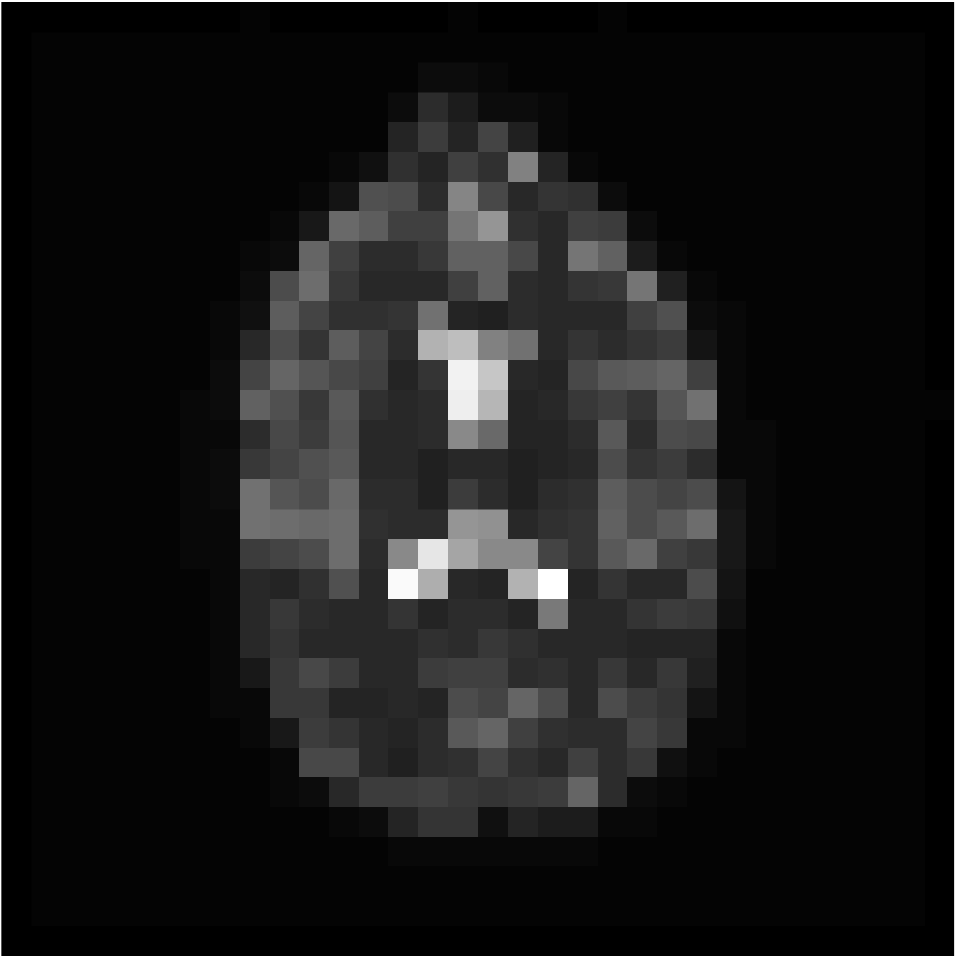} &
    \image{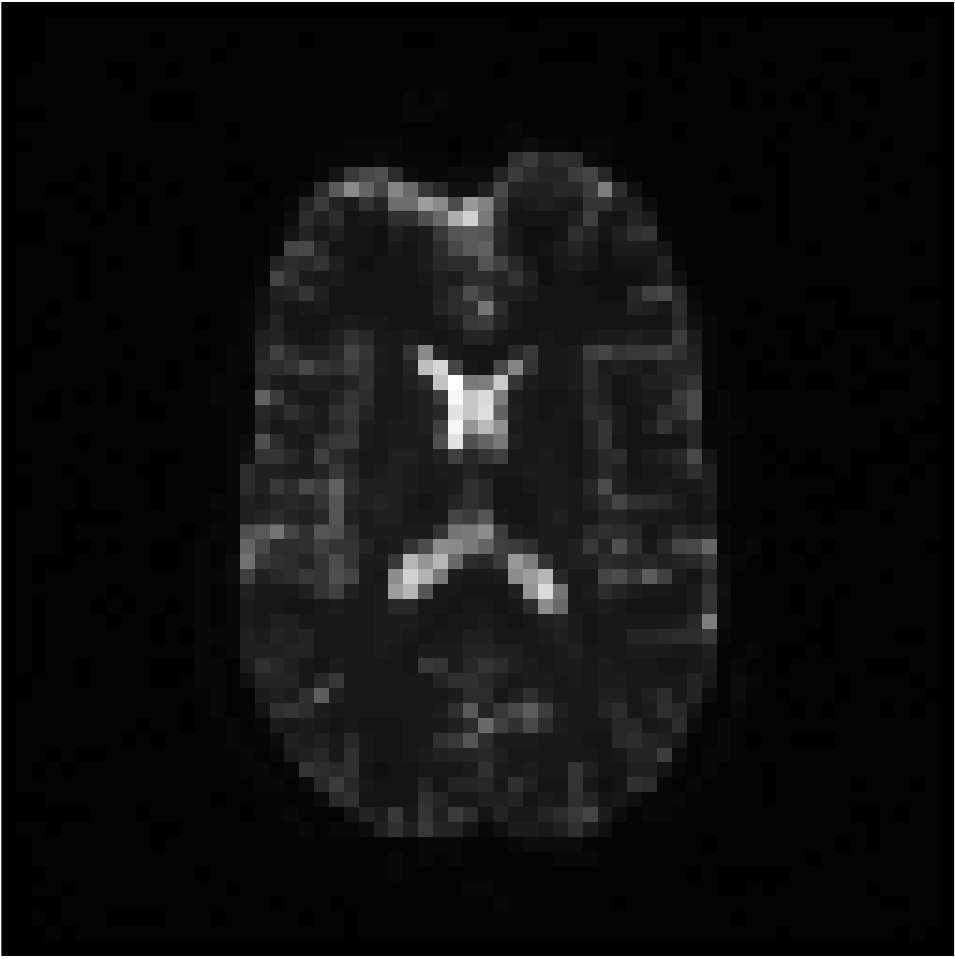} &
    \image{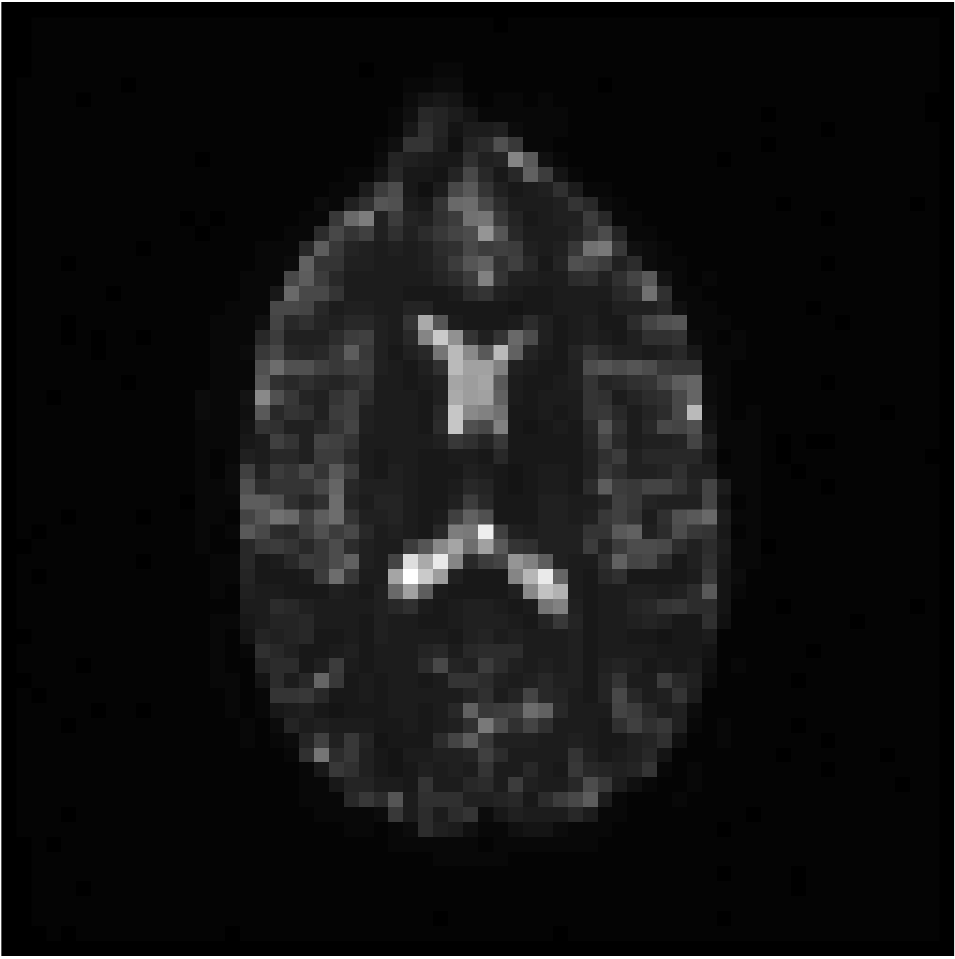} &
    \image{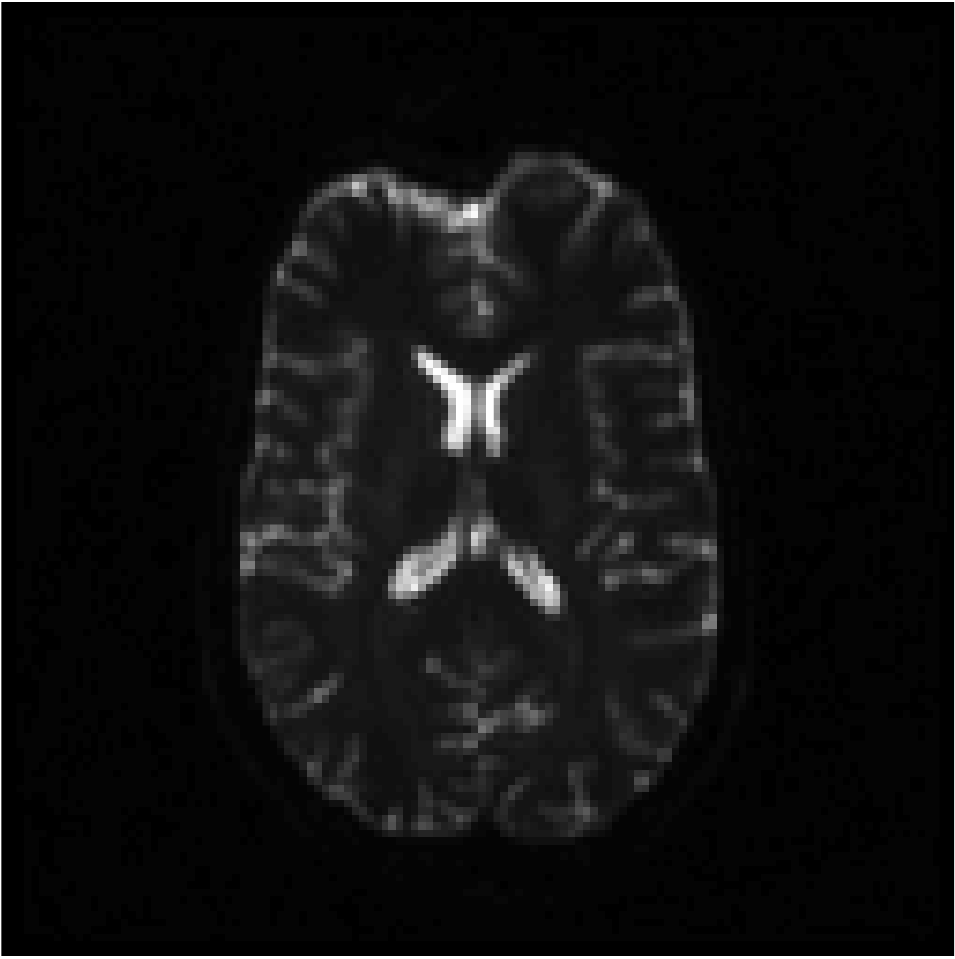} &
    \image{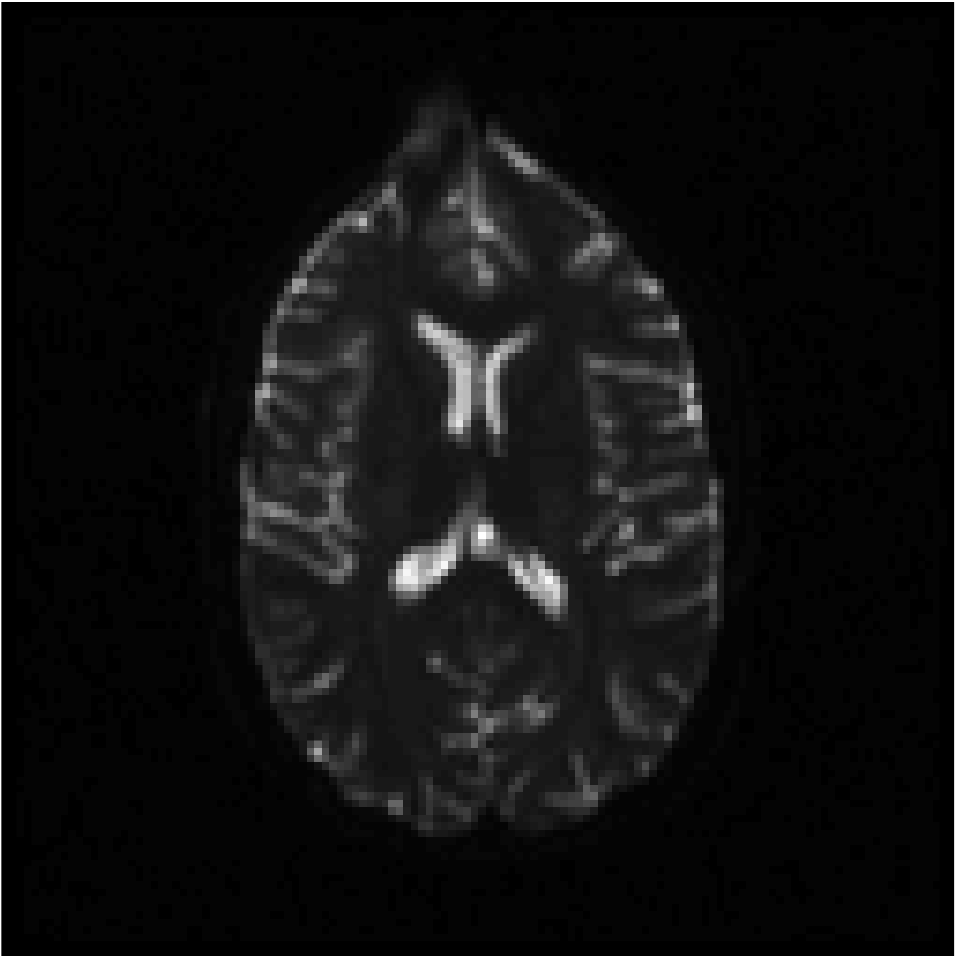} \\
	\rottext{inhomogeneity}&
    \image{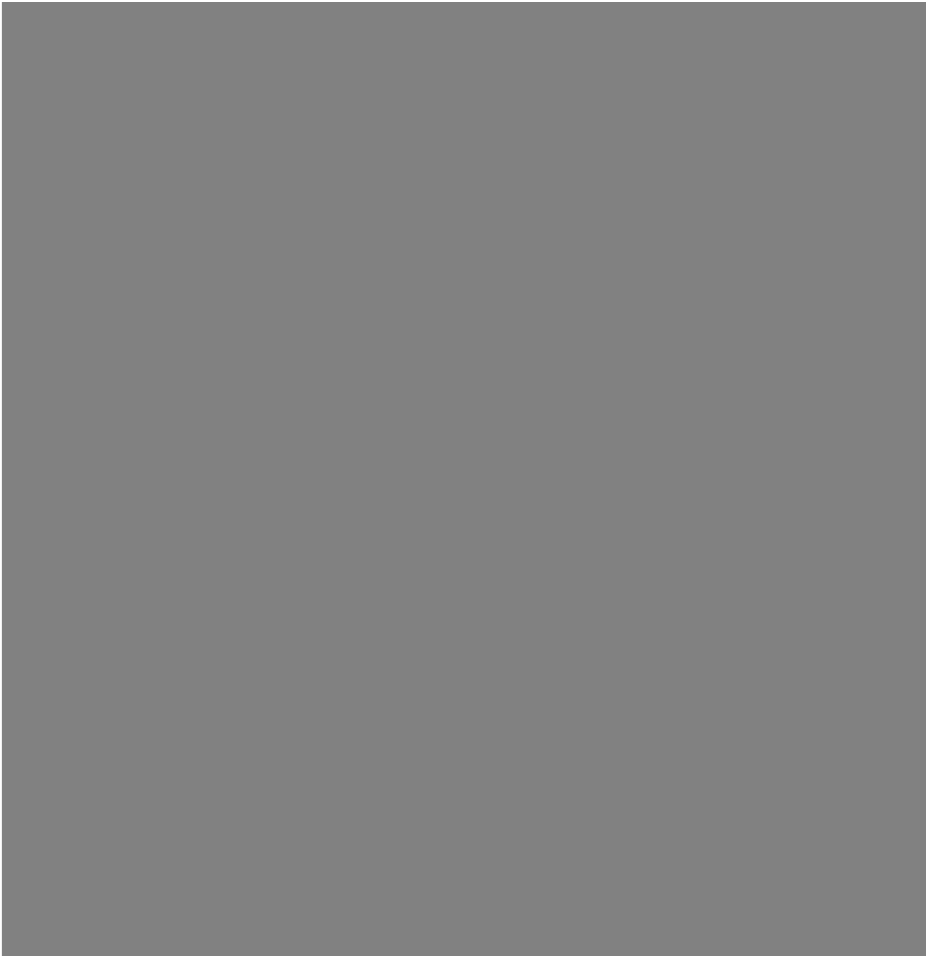} &
    \image{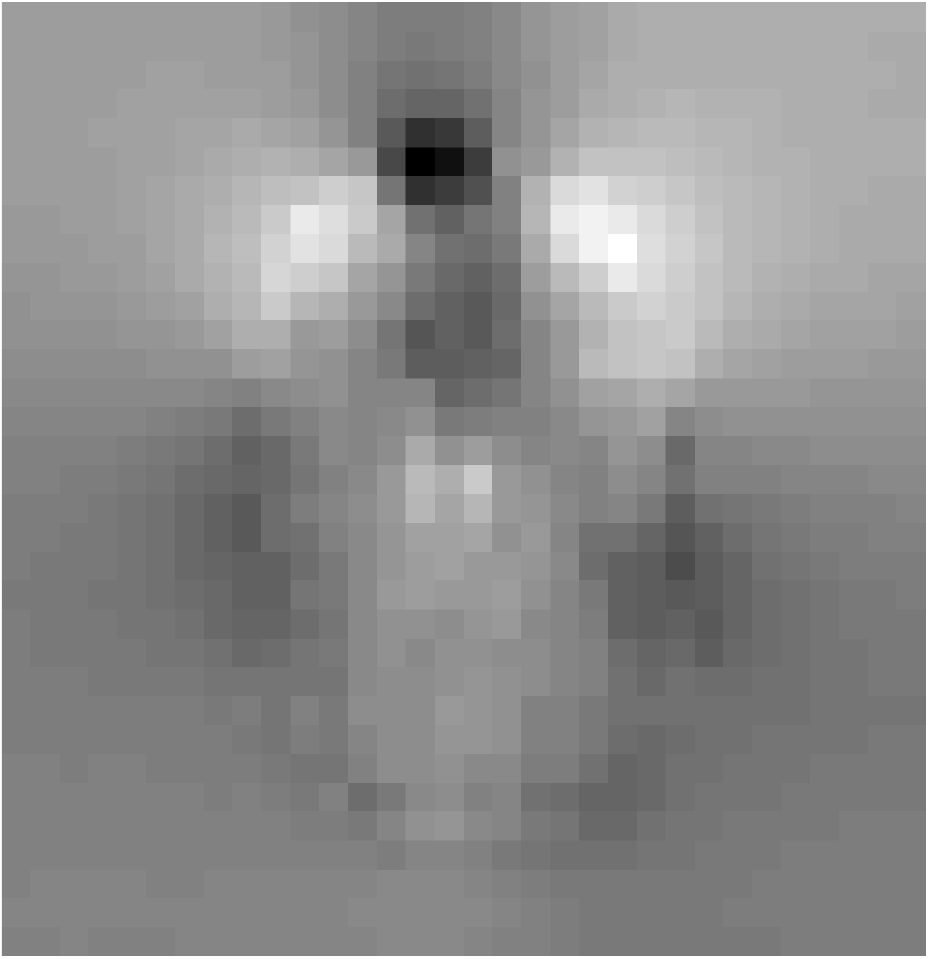} &
    \image{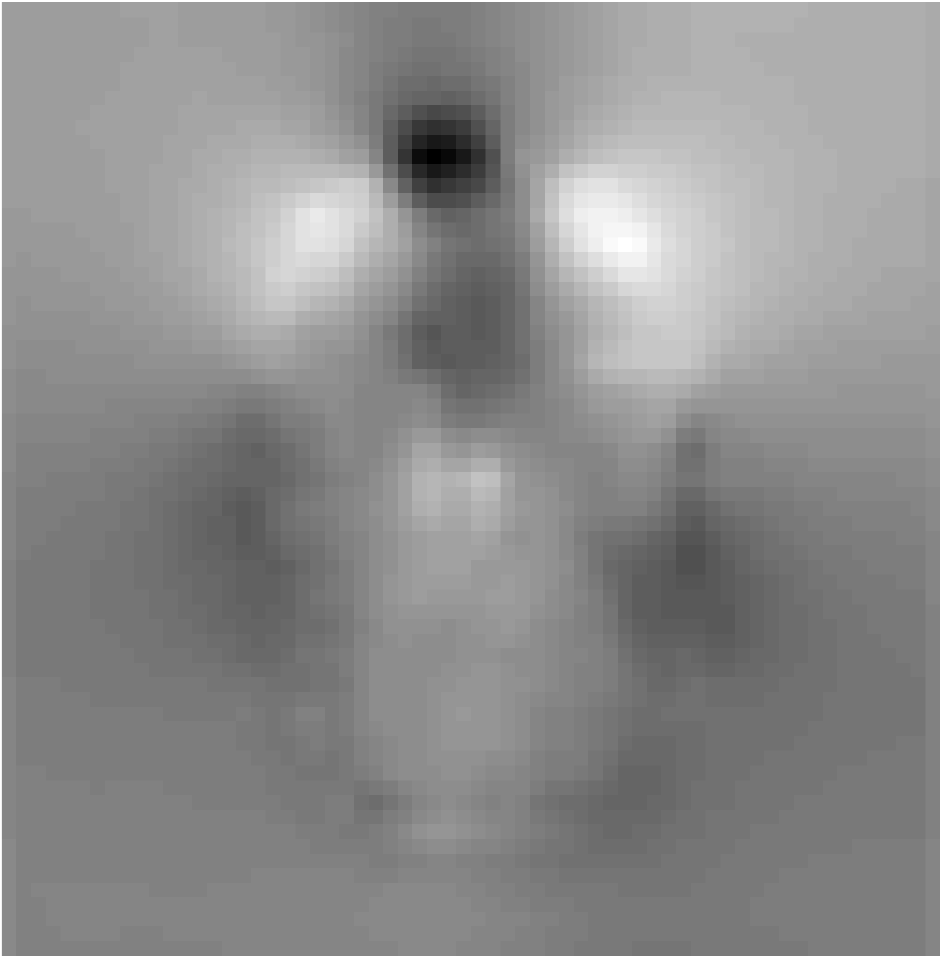} &
    \image{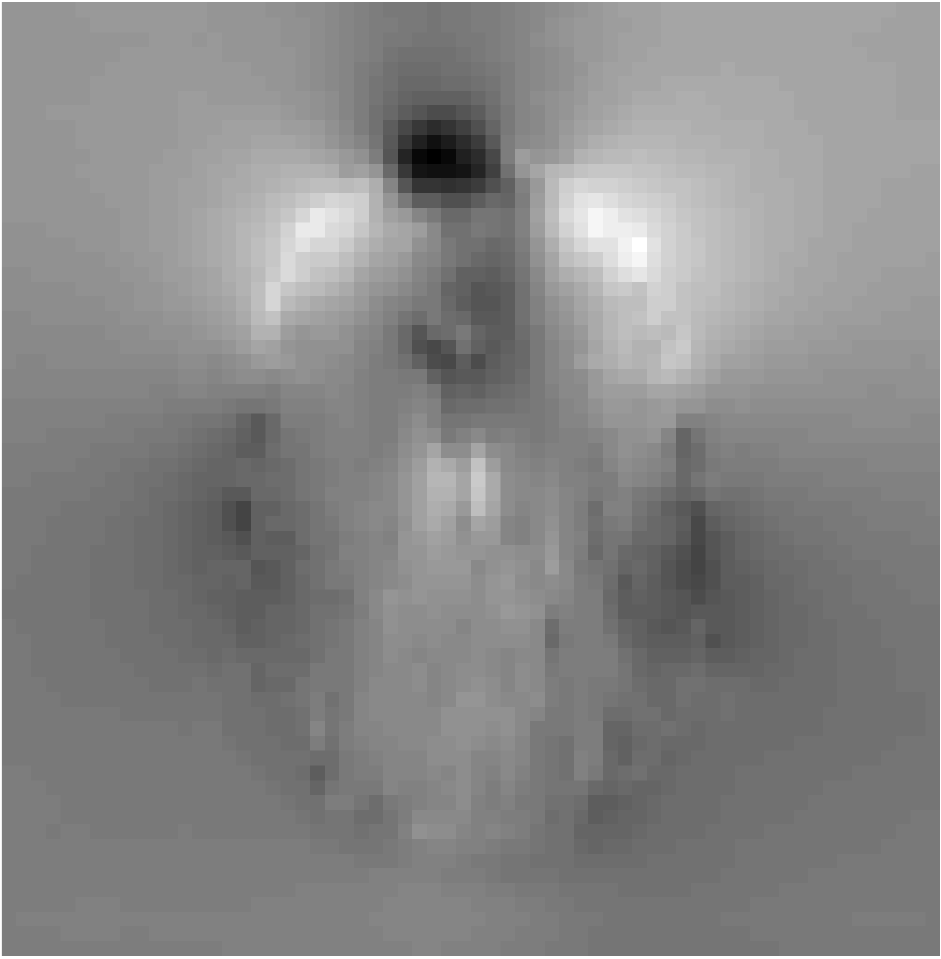} &
    \image{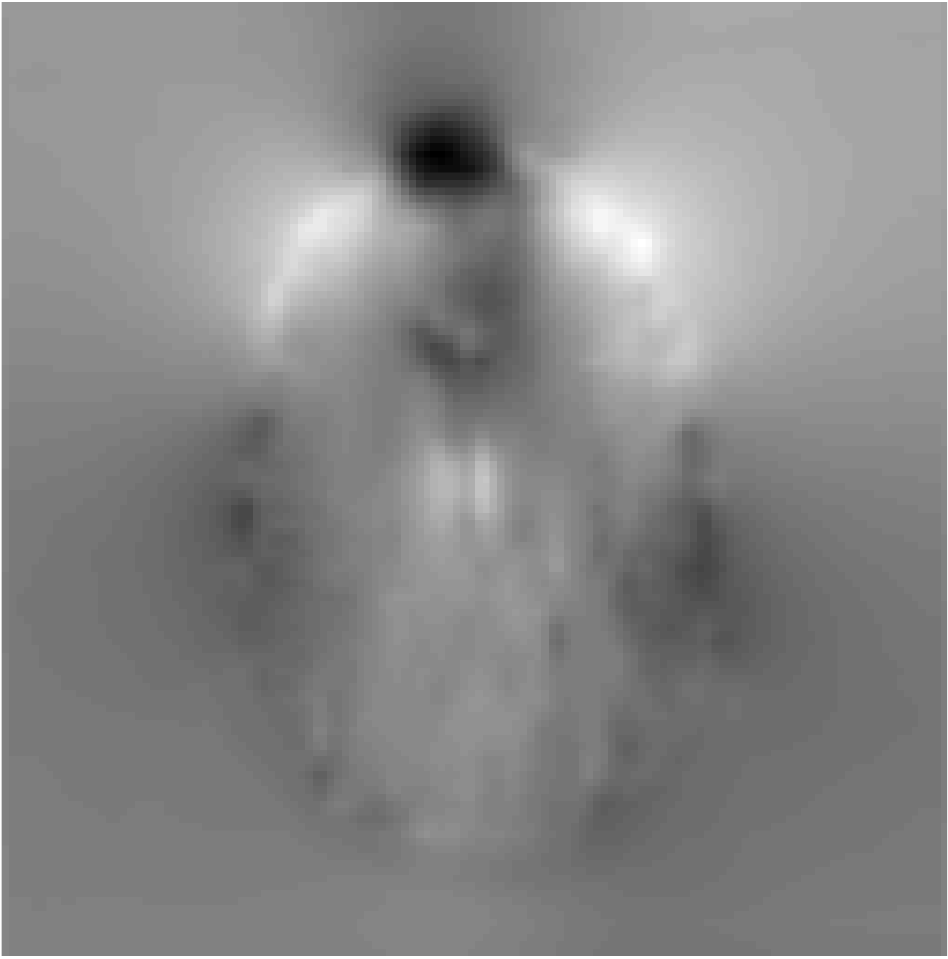} &
    \image{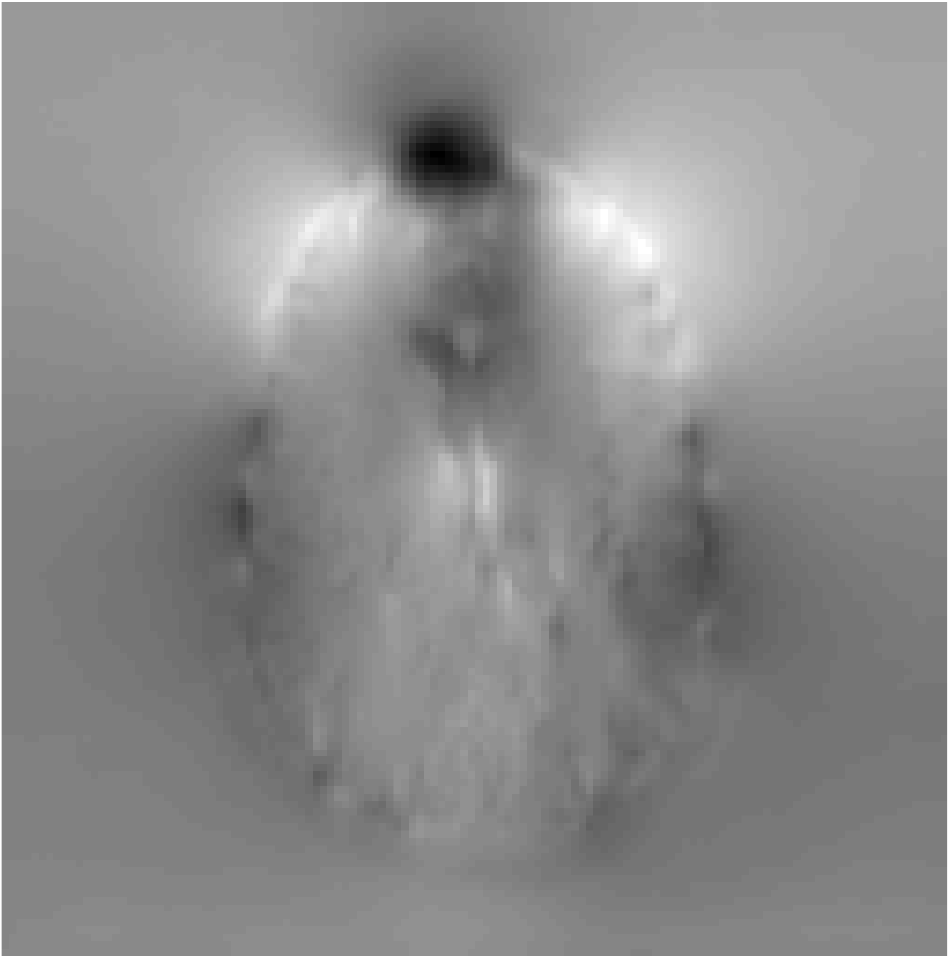} \\
	\rottext{corrected}&
    \image{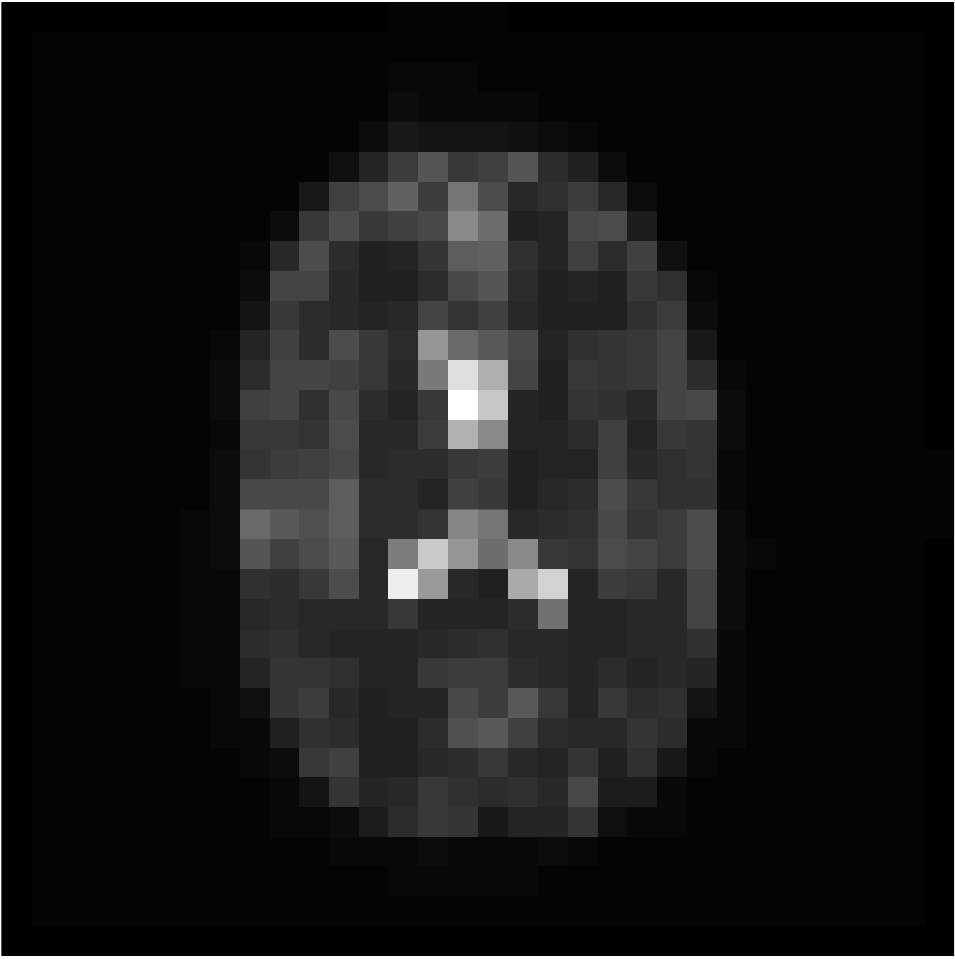} &
    \image{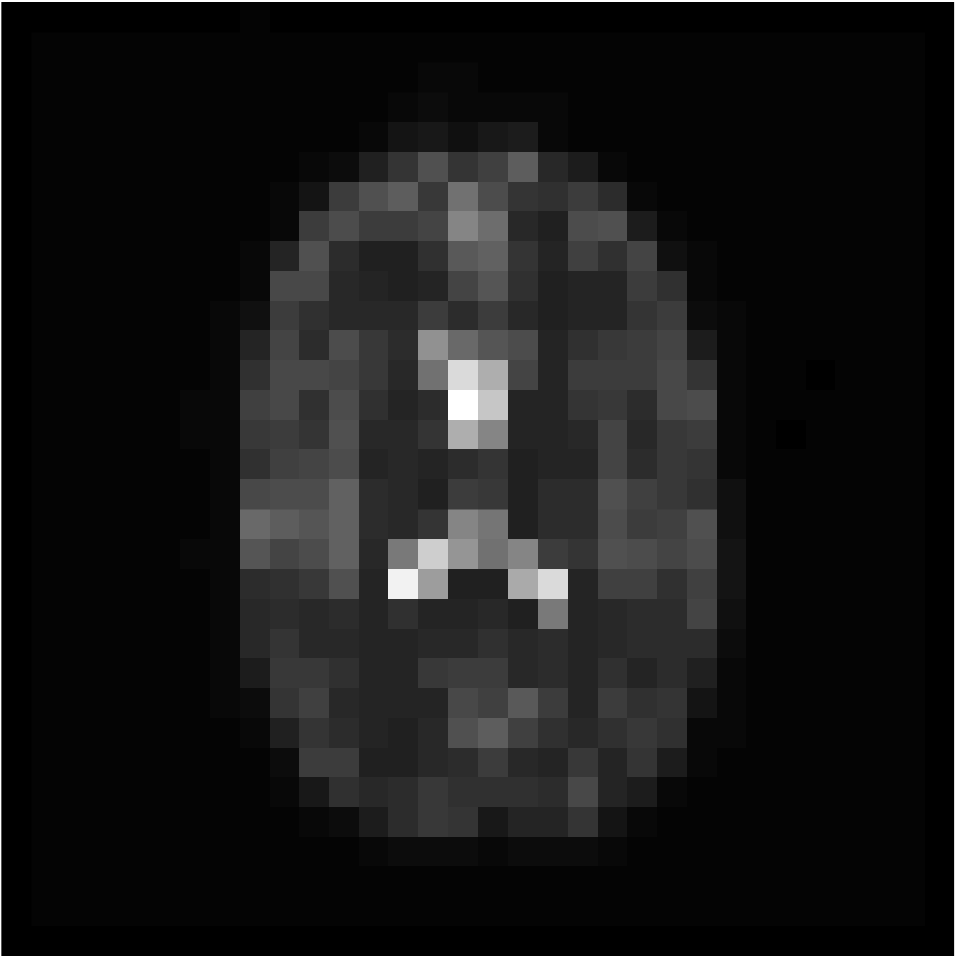} &
    \image{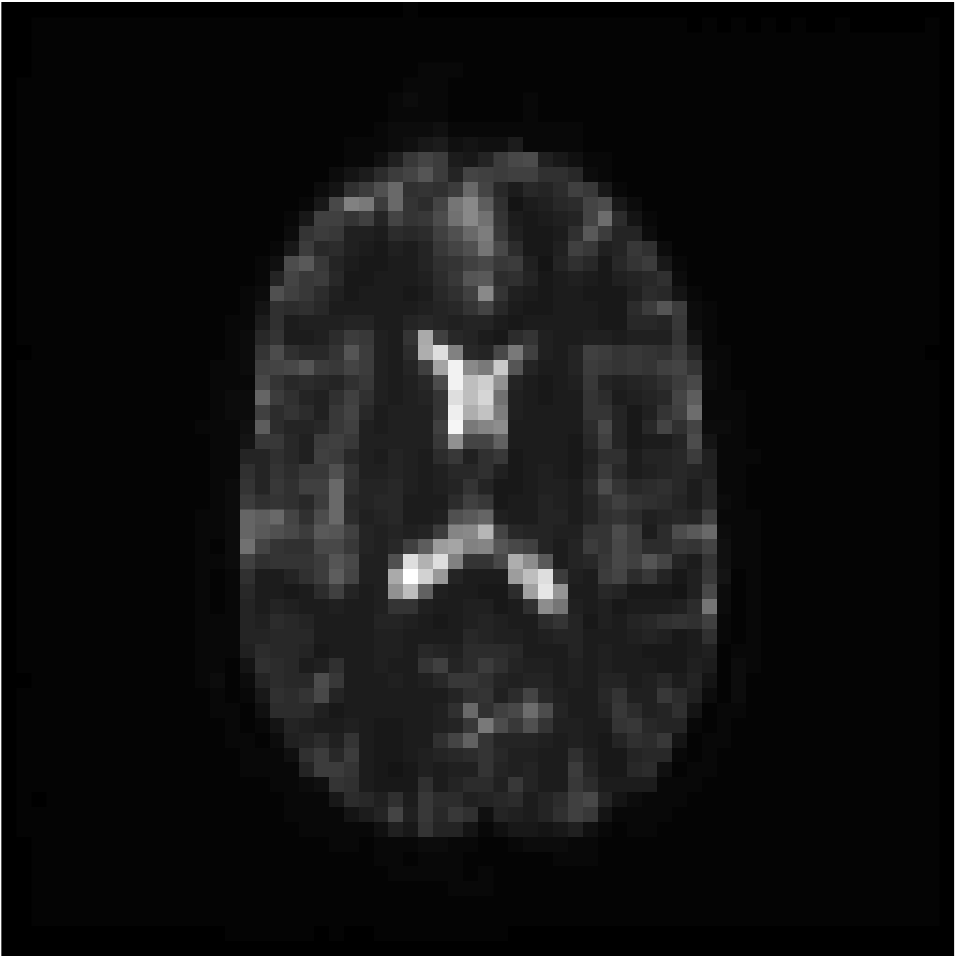} &
    \image{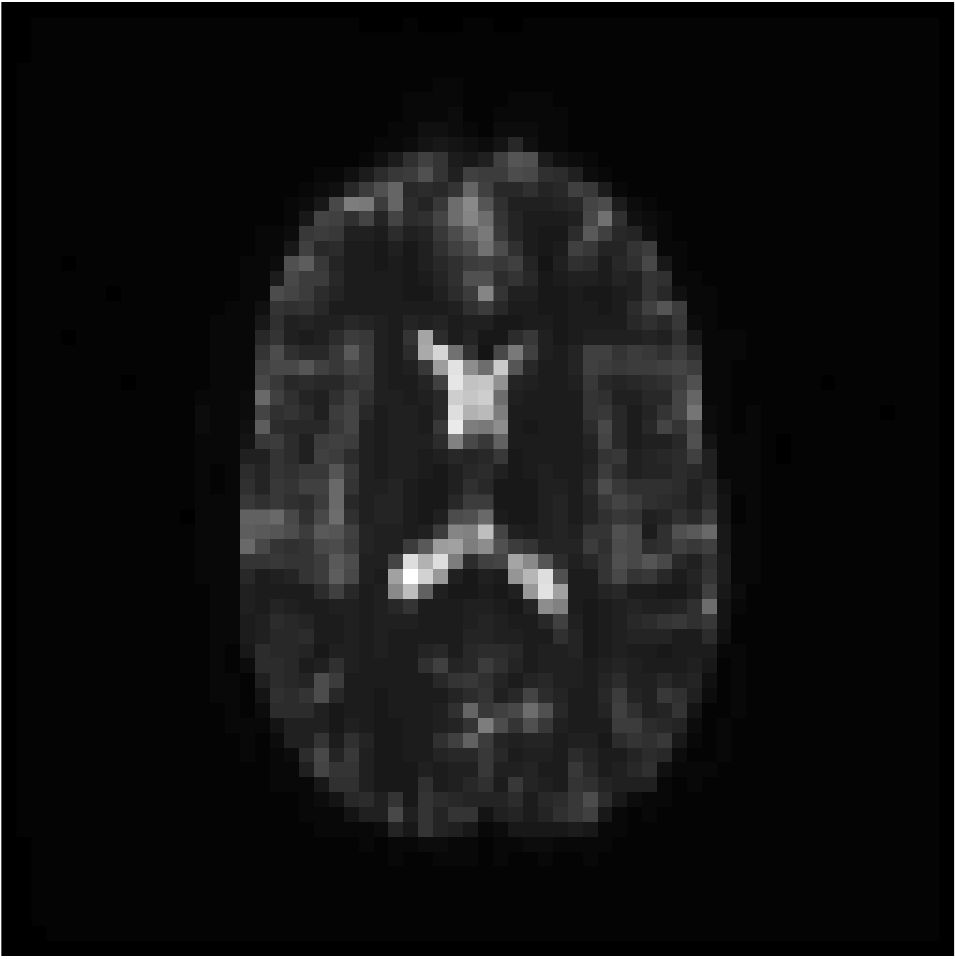} &
    \image{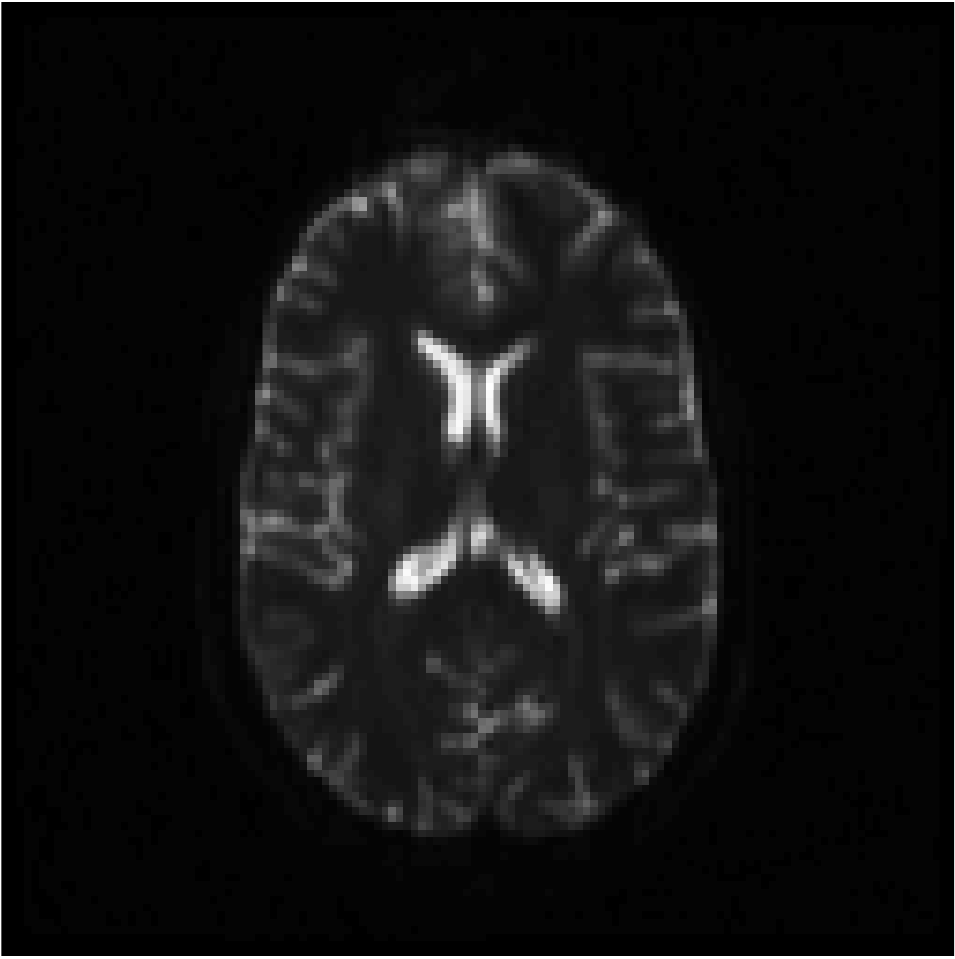} &
    \image{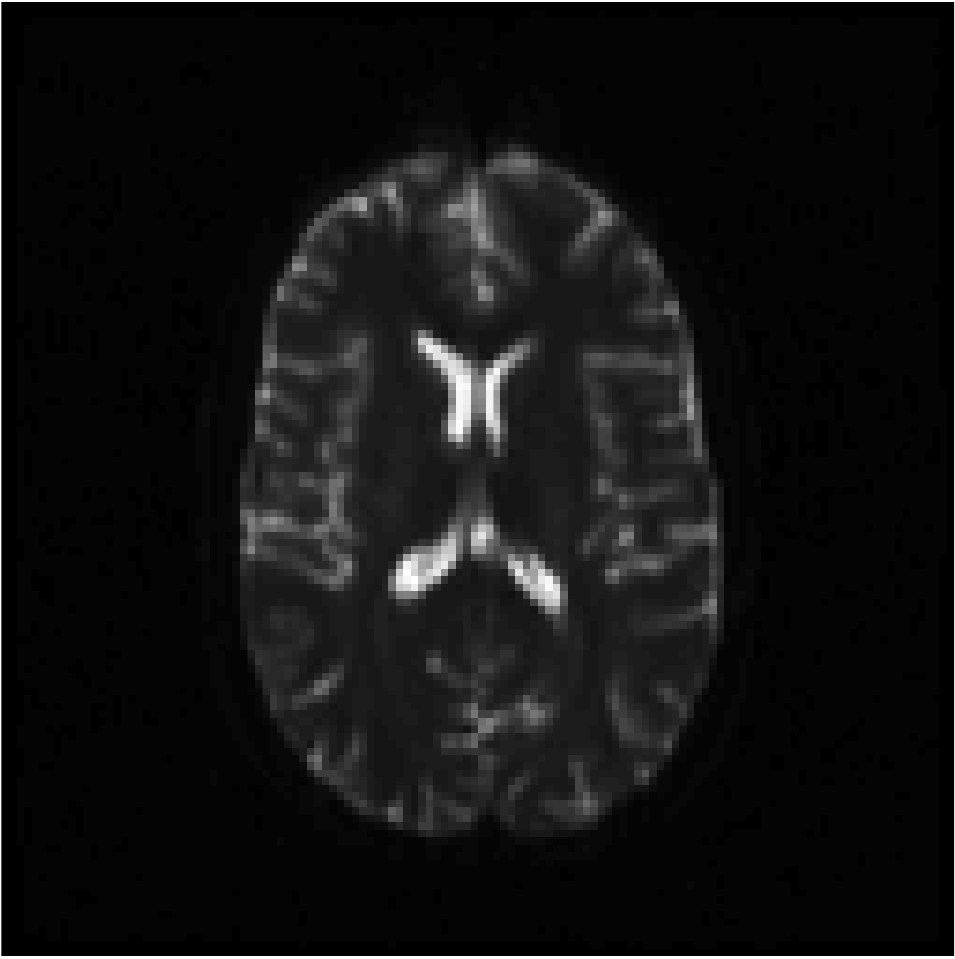} 
  \end{tabular}
  \end{center}
  \caption{Multilevel example. Deformed images $\CI_v$ and $\CI_{-v}$
  (top row), estimated field inhomogeneity (middle row) and corrected
  images (bottom row) are visualized for three different discretization
  levels (coarse to fine from left to right). The image data is shown in
  pairs corresponding to the different phase encoding directions. The
  left plots for the field inhomogeneity visualize the starting
  guesses and the right plots the solutions on each level. Data is courtesy of
  Harald Kugel, University Hospital M{\"u}nster, Germany,
  cf.~\cref{sec:experiments}.}
  \label{fig:multilevel}
\end{figure}
\par
The incorporation of the Gauss-Newton method into this multilevel
framework is straight-forward, as the prolongation from a coarser grid
can just be used as a starting guess for the Gauss-Newton iteration on a
finer grid, as stated above. 
\par
For ADMM, however, there are several options how to initialize the optimization
on a finer grid discretization, given the results on the coarser
grid. We have tested and compared three
strategies:
\begin{enumerate}
	\item Prolongate all three coarse mesh vectors $\bfb_c^{k+1}$, $\bfz_c^{k+1}$, and $\bfu_c^{k+1}$ and use the resulting fine mesh vectors $\bfb_f$, $\bfz_f$, and $\bfu_f$ as initial guesses for the next level.
	\item  Restart ADMM using the fine mesh variable $\bfb_f$ for
	both $\bfz^0$ and $\bfb^0$ and set the dual variable to zero, i.e., $\bfu^0=\mathbf{0}$.
	\item Restart ADMM on the fine level using $\bfb^0 = \bfz^0 = \hf( \bfb_f + \bfz_f)$ and $\bfu^0 = \mathbf{0}$ as initial guesses.
\end{enumerate}
In our examples, we obtained comparable results for all three strategies, however, 
the third strategy performed best and is used in the subsequent experiments.



\section{Numerical Experiments} 
\label{sec:experiments}
In this section, we perform numerical experiments to compare the
effectiveness of the preconditioning techniques and the performance of
the proposed ADMM method using real 2D and 3D data. We conclude the
section by comparing the proposed methods to an existing state-of-the-art
method for susceptibility artifact correction. A MacBook Pro laptop with
2.8~GHz Intel~Core~i7 processor and 16~GB 1600~MHz DD3 memory running
MATLAB 2015A is used for all numerical experiments.

\paragraph{Test Data} 
\label{par:test_data}
The 2D data are courtesy of Harald Kugel, Department of Clinical Radiology,  University Hospital  M{\"u}nster, Germany. 
A healthy subject was measured on a 3T scanner (Gyroscan Intera/Achieva 3.0T, System Release 2.5 (Philips, Best, NL)) using a standard clinical acquisition protocol.
For this example we extracted one slice of a series of 2D spin echo EPI
measurements that are performed using opposite phase encoding directions
along the anterior posterior direction. The field of view is 240~mm~$\times$~240~mm with a slice thickness of 3.6 mm. The acquisition matrix
is $128\times 128$, resulting in a pixel size of 1.875~mm~$\times$~1.875~mm and interpolated by zero filling to
0.9375~mm~$\times$~0.9375~mm. Contrast parameters were TR~=~9473~ms and
TE~=~95~ms.
Here, we use a three-level multilevel strategy with grid sizes of $32\times 32$, $64\times 64$, and $128\times 128$.
\par
The 3D data are provided by the Human Connectome Project~\cite{VanEssen2012}. 
We use a pair of unprocessed $b=0$ weighted images of a female subject aged 31-35 (subject id 111312) acquired using reversed phase encoding direction on a 7T scanner.
The field of view is 210~mm~$\times$~210~mm~$\times$~138.5~mm and the
voxel size is 1.05~mm in all spatial directions.
As in the 2D case, we use a three-level multilevel strategy using grid
sizes of $50\times50\times33$, $100\times100\times66$, and
$200\times200\times132$.

\paragraph{Performance Of Preconditioners In 2D} 
\label{par:2d_example}
We consider the Gauss-Newton-PCG method and investigate the performance
of the preconditioners described in \cref{sub:gauss_newton_pcg}.
To allow for a direct comparison within reasonable time we consider the 2D test data described above.
The main computational burden in GN-PCG is iteratively solving a linear
system involving the approximate Hessian~\cref{eq:dJGN} in each Gauss-Newton iteration. Using PCG, the number of iterations required depends on the spectral properties of the preconditioned Hessian and in particular on the clustering of its eigenvalues. 
\par
The spectra of the approximated Hessian with and without preconditioning
are illustrated in \cref{fig:spectrum} for the final Gauss-Newton iteration on the coarse level and parameters $\alpha=200$, $\beta=10$.
It can be seen that all three preconditioners condense the spectrum around the eigenvalue $1$. 
The tightest clustering is obtained using the symmetric Gauss-Seidel
preconditioner, $\bfP_{\rm SGS}$, that is, however, the most expensive
and not easy to parallelize. The least effective preconditioner is the
Jacobi preconditioner, $\bfP_{\rm Jac}$, which is also the cheapest. A good
trade-off between clustering of the eigenvalues and efficiency is observed for the proposed
block-Jacobi preconditioner $\bfP_{\textrm{block}}$.
\begin{figure}[t]
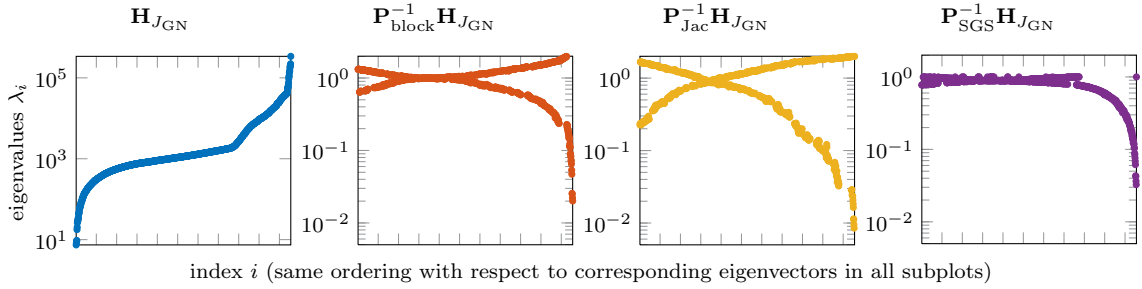

  \scriptsize
  \centering
	\setlength{\iwidth}{30mm}
	\setlength{\iheight}{25mm}
	\begin{tabular}{c@{}c@{}c@{}c@{}c}
	  & $\bfH_{J_{\rm GN}}$ & $\bfP_{\rm block}^{-1} \bfH_{J_{\rm
	  GN}}$ & $\bfP_{\rm
	  Jac}^{-1} \bfH_{J_{\rm GN}}$ & $\bfP_{\rm SGS}^{-1}
	  \bfH_{J_{\rm GN}}$\\[3mm]
		\rotatebox[origin=c]{90}{eigenvalues $\lambda_i$} &
      	\raisebox{-.5\height}{\input{./img/2Dpcg/2Dbrain_pcg_level5_it20_spectrum_H.tex}} &
        \raisebox{-.5\height}{\input{./img/2Dpcg/2Dbrain_pcg_level5_it20_spectrum_block.tex}} &
        \raisebox{-.5\height}{\input{./img/2Dpcg/2Dbrain_pcg_level5_it20_spectrum_Jac.tex}} &
        \raisebox{-.5\height}{\input{./img/2Dpcg/2Dbrain_pcg_level5_it20_spectrum_SGS.tex}} \\
        & \multicolumn{4}{c}{index $i$ (same ordering with respect to corresponding eigenvectors in all subplots)}
      \end{tabular} 
\caption{Spectra of the approximated Hessian~\cref{eq:dJGN} before
and after applying the block-Jacobi, Jacobi, and symmetric Gauss-Seidel preconditioners described in \cref{sub:gauss_newton_pcg}. As
test data, we consider the final iteration on the coarse level of the 2D
example also shown in \cref{fig:multilevel} with $\alpha=200$ and $\beta=10$.}
\label{fig:spectrum}
\end{figure}
\par
We compare the performance of the PCG solver using the different
preconditioners for the data from the coarse level of the 2D example
in \cref{fig:preconditioners}.
To illustrate the effect of the current estimate of $\bfb$ on the convergence, we compare the performance of PCG in the first and final Gauss-Newton iteration.
To demonstrate the convergence behavior, we approximately solve~\cref{eq:GNstep} to a relative residual tolerance of $10^{-6}$. Note that during the Gauss-Newton algorithm, we use a relatively large tolerance of $10^{-1}$.
Comparing the subplots in \cref{fig:preconditioners} it can be seen that the convergence is considerably faster in the first iteration, where $\bfb\equiv 0$, than in the final iteration. This effect is most pronounced for the unpreconditioned scheme but also notable for the preconditioned schemes.
While the symmetric Gauss-Seidel preconditioner uses the smallest number
of iterations overall, the proposed block-Jacobi preconditioner shows the best
performance during the first few PCG-iterations. 
Therefore, and due to the fact that this solver is parallelizable and of
low complexity, the block-Jacobi preconditioner is attractive for large-scale applications. Furthermore the symmetric Gauss-Seidel scheme has
the highest computational cost per PCG-iteration (in both the first and
final Gauss-Newton iteration), taking about 3.2~ms per PCG-iteration on average, whereas the block-Jacobi scheme only takes
about 0.3~ms per PCG-iteration on average. The Jacobi scheme is the cheapest
of the three preconditioners, taking only about 0.2~ms per
PCG-iteration on average.
\begin{figure}[t]
  \scriptsize
  \centering
      \iwidth=40mm
      \iheight=25mm
    \begin{tabular}{c@{\,}ccc}
    & \multicolumn{1}{c}{$32\times32$, first GN-iteration} &
    \multicolumn{1}{c}{$32\times32$, final GN-iteration} \\
      \rotatebox[origin=c]{90}{relative residual} &
      \raisebox{-.5\height}{
%
%
\definecolor{mycolor1}{rgb}{0.00000,0.44700,0.74100}%
\definecolor{mycolor2}{rgb}{0.85000,0.32500,0.09800}%
\definecolor{mycolor3}{rgb}{0.92900,0.69400,0.12500}%
\definecolor{mycolor4}{rgb}{0.49400,0.18400,0.55600}%
\begin{tikzpicture}[font=\scriptsize]

\begin{axis}[%
width=0.951\iwidth,
height=\iheight,
at={(0\iwidth,0\iheight)},
scale only axis,
xmin=0,
xmax=100,
ymode=log,
ymin=5e-07,
ymax=1e1,
yminorticks=true,
axis background/.style={fill=white},
]
\addplot [color=mycolor1,solid,thick]
  table[row sep=crcr]{%
1	1\\
2	0.773483460352376\\
3	0.552414652737569\\
4	0.420164667109479\\
5	0.267321057019219\\
6	0.25568945762772\\
7	0.213183718226248\\
8	0.151218523176581\\
9	0.152529228447799\\
10	0.117006008412681\\
11	0.0997508223733809\\
12	0.0930487559081483\\
13	0.0671039306603788\\
14	0.0883778271997561\\
15	0.0747673206018275\\
16	0.0699223670770564\\
17	0.0538647647342674\\
18	0.0491656651363554\\
19	0.0441320342630557\\
20	0.0335661665378155\\
21	0.0275307768135954\\
22	0.024230781950028\\
23	0.038601897086525\\
24	0.0337852955192759\\
25	0.022601742529795\\
26	0.0259324969016983\\
27	0.0198471029519746\\
28	0.0160453772580731\\
29	0.017347039758371\\
30	0.0297745779253281\\
31	0.0142137914446069\\
32	0.013166589728052\\
33	0.0116132684296274\\
34	0.0101627436134487\\
35	0.0200975072536941\\
36	0.00973444498833182\\
37	0.0180002129986361\\
38	0.00932334613163465\\
39	0.0182750484340548\\
40	0.00810556769872901\\
41	0.0203509663453196\\
42	0.0069696080747398\\
43	0.00634516615651116\\
44	0.0121693597246019\\
45	0.00624407078238137\\
46	0.00786877400989476\\
47	0.0074112483931743\\
48	0.00522459303219179\\
49	0.00512417560627453\\
50	0.00472227806223659\\
51	0.0051280089072384\\
52	0.00624924028284581\\
53	0.00362638464755816\\
54	0.00441552724376779\\
55	0.00579752535313297\\
56	0.00513435141019345\\
57	0.00565372361555741\\
58	0.0053318525207084\\
59	0.00359121027749688\\
60	0.00343135352158448\\
61	0.00787799071686592\\
62	0.00359711867373386\\
63	0.00653785853818918\\
64	0.00486103176620457\\
65	0.00405978177057356\\
66	0.00371578201936201\\
67	0.00354927398528895\\
68	0.00269873686016991\\
69	0.00316913152053913\\
70	0.00420124780063999\\
71	0.00326041557020286\\
72	0.00360357607906809\\
73	0.00259115588575275\\
74	0.0060173922506413\\
75	0.0041367398265859\\
76	0.00340878854477991\\
77	0.00449743527888229\\
78	0.00304405402852774\\
79	0.00272823772583977\\
80	0.00264394812760827\\
81	0.00412064230882172\\
82	0.00244614343080014\\
83	0.0023811527173143\\
84	0.00261632986221021\\
85	0.00219993850105575\\
86	0.00201174840866635\\
87	0.00437460375971209\\
88	0.00206656268197094\\
89	0.00344173899639461\\
90	0.00182875752978886\\
91	0.00178534603567804\\
92	0.00416733866503954\\
93	0.00177412068324892\\
94	0.00311485058093711\\
95	0.00165971141503981\\
96	0.0016642252166337\\
97	0.00402916341941609\\
98	0.00153596788622894\\
99	0.00142303663493266\\
100	0.00220129140588062\\
101	0.0021068000814507\\
};

\addplot [color=mycolor2,dashed,mark=*,mark options={solid},mark size=0.4mm,thick,mark repeat=5]
  table[row sep=crcr]{%
1	1\\
2	0.0254638016352529\\
3	0.00736606178734242\\
4	0.00526262040759343\\
5	0.00262208041693671\\
6	0.00215330211598149\\
7	0.00158072900087025\\
8	0.00142376596209219\\
9	0.000961185983713467\\
10	0.00111573555075833\\
11	0.00082937335489474\\
12	0.000843644897172589\\
13	0.000543565445980488\\
14	0.000538243375169715\\
15	0.000461595130894214\\
16	0.000474844289527205\\
17	0.000304044438925699\\
18	0.000240563299652893\\
19	0.000170562136892782\\
20	0.000177681153016906\\
21	0.000121312179754805\\
22	0.00012143516803979\\
23	0.000100679880722299\\
24	9.19950502213665e-05\\
25	6.16234427835497e-05\\
26	6.04279047073353e-05\\
27	4.27724738414217e-05\\
28	3.74321272120149e-05\\
29	2.95812185575121e-05\\
30	2.91015893577113e-05\\
31	2.06248084473086e-05\\
32	2.49191064198453e-05\\
33	2.4062866854825e-05\\
34	2.29714849955797e-05\\
35	1.58729641324354e-05\\
36	1.16985007885748e-05\\
37	7.81926874416214e-06\\
38	6.01470804840687e-06\\
39	4.08189142817824e-06\\
40	3.32985512042124e-06\\
41	1.96882583957461e-06\\
42	1.46365001896687e-06\\
43	1.09406830637674e-06\\
44	9.57030828409568e-07\\
};

\addplot [color=mycolor3,dashed,mark=triangle*,mark options={solid},mark size=0.5mm,thick,mark repeat=5]
  table[row sep=crcr]{%
1	1\\
2	0.326499150108709\\
3	0.148335294873891\\
4	0.0954584929650848\\
5	0.08209293104841\\
6	0.0642914754668549\\
7	0.0494707923734004\\
8	0.0451922228264767\\
9	0.0424882740727033\\
10	0.0212789907937147\\
11	0.0178129029242109\\
12	0.0162613204770536\\
13	0.0105558914718277\\
14	0.0112829726397847\\
15	0.00856920591715021\\
16	0.00579586272585582\\
17	0.00526030821461556\\
18	0.0047966691367877\\
19	0.00366929551902173\\
20	0.00280207791365204\\
21	0.00277594635870879\\
22	0.00294698802796444\\
23	0.00322721015871544\\
24	0.00257881748103167\\
25	0.00191463770852511\\
26	0.00168379061468903\\
27	0.0018219746398801\\
28	0.00186215158679792\\
29	0.00135538662615863\\
30	0.00109646079778978\\
31	0.00110321500489015\\
32	0.000748753250564328\\
33	0.000689199558602155\\
34	0.000644324002648899\\
35	0.000661904732655952\\
36	0.00049828972592934\\
37	0.000562971889444797\\
38	0.000583620898043583\\
39	0.000557597982098396\\
40	0.0004114061341713\\
41	0.000410816236189356\\
42	0.000299203046302266\\
43	0.000288136887866728\\
44	0.000297732016586157\\
45	0.000246929399282897\\
46	0.000224913057955671\\
47	0.000218590168054499\\
48	0.00016832166063332\\
49	0.000136938632820325\\
50	0.000119397814799407\\
51	0.000100803237820593\\
52	0.0001141414082414\\
53	0.000106411912794032\\
54	0.000109795437372559\\
55	8.55585900833093e-05\\
56	7.94545042418465e-05\\
57	7.09075666865291e-05\\
58	6.22282271067692e-05\\
59	5.18871678545973e-05\\
60	4.61504161913369e-05\\
61	4.11245103706496e-05\\
62	3.37465696729351e-05\\
63	2.81642727949685e-05\\
64	2.58426837720377e-05\\
65	1.990695056493e-05\\
66	1.72043646006711e-05\\
67	1.6077713600598e-05\\
68	1.37869878028124e-05\\
69	1.19269704135736e-05\\
70	8.27799194438075e-06\\
71	6.74071195548067e-06\\
72	5.52040933276794e-06\\
73	5.35855805188508e-06\\
74	4.76463136673819e-06\\
75	4.05350457647075e-06\\
76	3.98387232448259e-06\\
77	3.64275238656279e-06\\
78	3.07011742281906e-06\\
79	2.53165960446996e-06\\
80	2.3452649125556e-06\\
81	2.31174878684437e-06\\
82	1.95650010943032e-06\\
83	1.67026562404535e-06\\
84	1.40154773523245e-06\\
85	1.26348721243547e-06\\
86	1.08996006772531e-06\\
87	9.42120790138933e-07\\
};

\addplot [color=mycolor4,dashed,mark=square*,mark options={solid},mark size=0.4mm,thick,mark repeat=5]
  table[row sep=crcr]{%
1	1\\
2	0.124798422725632\\
3	0.101418278358452\\
4	0.0538302685194774\\
5	0.0237217445141084\\
6	0.0155800913447028\\
7	0.0087135372894044\\
8	0.00402826033675784\\
9	0.00326465163081689\\
10	0.00313724708679477\\
11	0.00184260514770541\\
12	0.00140822150379217\\
13	0.000960354096961501\\
14	0.000604818953252854\\
15	0.000628628420787565\\
16	0.000522855373746037\\
17	0.000545326451445467\\
18	0.000205273529771326\\
19	0.000163767450010063\\
20	0.000142248612213545\\
21	0.000122688920907017\\
22	7.10896705556709e-05\\
23	5.27216453967675e-05\\
24	3.38700780840638e-05\\
25	2.09058183677398e-05\\
26	1.49555222859407e-05\\
27	7.79485948498771e-06\\
28	6.02992484894262e-06\\
29	3.65026591979795e-06\\
30	2.62980236350988e-06\\
31	2.75807629903683e-06\\
32	1.51700890170284e-06\\
33	1.26660853717839e-06\\
34	1.02935830776426e-06\\
35	7.99115883766204e-07\\
};

\end{axis}
\end{tikzpicture}%
} &
      \raisebox{-.5\height}{
%
%
\definecolor{mycolor1}{rgb}{0.00000,0.44700,0.74100}%
\definecolor{mycolor2}{rgb}{0.85000,0.32500,0.09800}%
\definecolor{mycolor3}{rgb}{0.92900,0.69400,0.12500}%
\definecolor{mycolor4}{rgb}{0.49400,0.18400,0.55600}%
\begin{tikzpicture}[font=\scriptsize]

\begin{axis}[%
width=0.951\iwidth,
height=\iheight,
at={(0\iwidth,0\iheight)},
scale only axis,
xmin=0,
xmax=100,
ymode=log,
ymin=5e-07,
ymax=1e1,
yminorticks=true,
axis background/.style={fill=white},
]
\addplot [color=mycolor1,solid,thick]
  table[row sep=crcr]{%
1	1\\
2	1.47523280216187\\
3	2.10302047747563\\
4	1.97822031437013\\
5	1.24102713516894\\
6	0.90216618347651\\
7	1.34662332959807\\
8	1.04589163993352\\
9	0.955308897839724\\
10	0.633943306147675\\
11	0.665739154768581\\
12	0.875438171121919\\
13	0.729690707240543\\
14	0.574151689203049\\
15	0.586448138365804\\
16	0.498927824476833\\
17	0.457135258705028\\
18	0.536953988676805\\
19	0.403168006593748\\
20	0.402247318211792\\
21	0.369317492205928\\
22	0.3855430387017\\
23	0.357804288482376\\
24	0.266880214206933\\
25	0.619720269687064\\
26	0.244374773107902\\
27	0.194976671252617\\
28	0.229001872852775\\
29	0.24048511223859\\
30	0.232085144657041\\
31	0.233389443882535\\
32	0.198759351629024\\
33	0.268527959023326\\
34	0.213487234495242\\
35	0.20022315830416\\
36	0.401657496949406\\
37	0.251428507010365\\
38	0.185607375225014\\
39	0.350266575637682\\
40	0.239771646701332\\
41	0.456923898913075\\
42	0.212377093087603\\
43	0.315656832259173\\
44	0.267381419398389\\
45	0.205327392816184\\
46	0.303587139190361\\
47	0.239766424614404\\
48	0.198685793791366\\
49	0.391708624430857\\
50	0.213631241248531\\
51	0.253419392753049\\
52	0.356503902970883\\
53	0.210563231981988\\
54	0.251809479704953\\
55	0.251888977147634\\
56	0.194438630160291\\
57	0.458025459743163\\
58	0.2885171463688\\
59	0.229132877107692\\
60	0.257477480173379\\
61	0.227815862636282\\
62	0.253218836538762\\
63	0.245004654560221\\
64	0.179381038921349\\
65	0.30539636619868\\
66	0.160973064544655\\
67	0.174480576629173\\
68	0.143278252094646\\
69	0.308331295947518\\
70	0.182610338799483\\
71	0.211160815024378\\
72	0.14728195182551\\
73	0.309412143011949\\
74	0.133797701972373\\
75	0.22436147845376\\
76	0.126420741249883\\
77	0.158995549177534\\
78	0.161048243349086\\
79	0.203182014699494\\
80	0.183634856219017\\
81	0.133577730408953\\
82	0.133808481721134\\
83	0.206000628139345\\
84	0.13293612155962\\
85	0.216240486515226\\
86	0.146852712337272\\
87	0.201043973459691\\
88	0.143607181082557\\
89	0.146096426859026\\
90	0.251646247722585\\
91	0.200374901938527\\
92	0.173222008565281\\
93	0.130825633143417\\
94	0.122166179153158\\
95	0.200080659159642\\
96	0.13054290392232\\
97	0.143065890002016\\
98	0.222332201777883\\
99	0.183472198164582\\
100	0.177256720134511\\
101	0.155403514097008\\
};

\addplot [color=mycolor2,dashed,mark=*,mark options={solid},mark size=0.4mm,thick,mark repeat=5]
  table[row sep=crcr]{%
1	1\\
2	0.191339495326108\\
3	0.128768712303506\\
4	0.0961472003656994\\
5	0.106743837492255\\
6	0.101423274298228\\
7	0.0874388499282105\\
8	0.0689494289819565\\
9	0.0589159546740037\\
10	0.0634221792482374\\
11	0.0528928429866193\\
12	0.0392070627616446\\
13	0.0298545310196614\\
14	0.0214496932557434\\
15	0.0177564530780891\\
16	0.0114323247145368\\
17	0.00725888471936178\\
18	0.00591333474812438\\
19	0.00420770142327341\\
20	0.00362405264260163\\
21	0.00298808155410897\\
22	0.00303204667546862\\
23	0.00254962927657799\\
24	0.00284726384518058\\
25	0.00198861830680341\\
26	0.00176180653412741\\
27	0.00121603238787776\\
28	0.00108118022410545\\
29	0.000785468359389021\\
30	0.000666889275303922\\
31	0.000461425768192349\\
32	0.000380094325430877\\
33	0.000250929498133945\\
34	0.000207100983232622\\
35	0.000151606731900346\\
36	0.00014028518953479\\
37	0.00011117130899209\\
38	0.000101769123290632\\
39	6.83036763540633e-05\\
40	5.82726701959939e-05\\
41	3.84366719516986e-05\\
42	3.804870113614e-05\\
43	3.15157922864946e-05\\
44	3.30798797980977e-05\\
45	3.01582351662926e-05\\
46	2.7036032229103e-05\\
47	1.94432379865463e-05\\
48	1.39598440939192e-05\\
49	1.03176654838879e-05\\
50	7.25518332541337e-06\\
51	5.74631337605063e-06\\
52	4.68259327702163e-06\\
53	3.41779990278703e-06\\
54	2.21084673559372e-06\\
55	1.54291009397311e-06\\
56	1.38429255505116e-06\\
57	1.02291167357577e-06\\
58	8.16340581650909e-07\\
};

\addplot [color=mycolor3,dashed,mark=triangle*,mark options={solid},mark size=0.5mm,thick,mark repeat=5]
  table[row sep=crcr]{%
1	1\\
2	0.97706466143587\\
3	0.826624673062755\\
4	0.605514621330423\\
5	0.532776927478158\\
6	0.476841822924213\\
7	0.43191866467668\\
8	0.408298511854172\\
9	0.409135717347235\\
10	0.354455882287185\\
11	0.280688672712863\\
12	0.280168354271917\\
13	0.206135776318813\\
14	0.242606937609645\\
15	0.195761947359179\\
16	0.222248885238954\\
17	0.166832358211848\\
18	0.147289826202068\\
19	0.142446350102175\\
20	0.131369281899855\\
21	0.114206353409786\\
22	0.106758006450317\\
23	0.0788226301883782\\
24	0.0578127118215838\\
25	0.0511436938168118\\
26	0.0509152047528189\\
27	0.0376740556365538\\
28	0.0376649734571982\\
29	0.0274338135111563\\
30	0.0254842088410485\\
31	0.0237193525046561\\
32	0.0182602421739917\\
33	0.0141267578793552\\
34	0.0144346371086535\\
35	0.0126692341119327\\
36	0.0107424751610062\\
37	0.0083289681436138\\
38	0.00681371428288828\\
39	0.00602536664021669\\
40	0.00584169837833975\\
41	0.00470398978370181\\
42	0.00424522368511576\\
43	0.00375359207629544\\
44	0.00320535594705516\\
45	0.00372117978486788\\
46	0.00311867314212678\\
47	0.00256717429411609\\
48	0.00219723215870154\\
49	0.00189124977013808\\
50	0.00154113563676753\\
51	0.00125123857847535\\
52	0.00121473096099899\\
53	0.00115193160990784\\
54	0.00104009250267803\\
55	0.000881960031774239\\
56	0.000729762439275971\\
57	0.000697024148774252\\
58	0.000511727312934954\\
59	0.000439749498836669\\
60	0.000388767601381571\\
61	0.000334295983112147\\
62	0.00032131300706608\\
63	0.000253790638966987\\
64	0.000197743023243073\\
65	0.000148655568415788\\
66	0.000131475003539696\\
67	0.000112046662401163\\
68	0.000100129724420571\\
69	9.31919586215996e-05\\
70	8.61479534654535e-05\\
71	8.24602265615441e-05\\
72	6.73392222895757e-05\\
73	5.4855211229425e-05\\
74	5.39271003740548e-05\\
75	4.79804104654809e-05\\
76	4.81688472145402e-05\\
77	4.00484285931404e-05\\
78	3.65726504621073e-05\\
79	3.46282051722489e-05\\
80	3.15536443005534e-05\\
81	2.80895789319996e-05\\
82	2.58980460282149e-05\\
83	2.08910913869541e-05\\
84	1.51567996082308e-05\\
85	1.27593597252488e-05\\
86	1.02095129646668e-05\\
87	8.87521861796649e-06\\
88	6.57241650503327e-06\\
89	5.59816361604803e-06\\
90	4.52236525980662e-06\\
91	4.27183259637948e-06\\
92	3.42805744419671e-06\\
93	3.14546531513615e-06\\
94	2.62318579376521e-06\\
95	2.25800558990205e-06\\
96	1.77414415665198e-06\\
97	1.42883324318008e-06\\
98	1.30524395474198e-06\\
99	1.10263375008943e-06\\
100	9.92255378061859e-07\\
};

\addplot [color=mycolor4,dashed,mark=square*,mark options={solid},mark size=0.4mm,thick,mark repeat=5]
  table[row sep=crcr]{%
1	1\\
2	0.548041303931757\\
3	0.339190469439589\\
4	0.274028183864361\\
5	0.209302308863463\\
6	0.17675434319699\\
7	0.195339474610032\\
8	0.122106756883853\\
9	0.0888218862083532\\
10	0.0588124895106901\\
11	0.0287478297925016\\
12	0.0252662973608361\\
13	0.0167964662387283\\
14	0.00804115238975938\\
15	0.00746590410381806\\
16	0.00518715106120291\\
17	0.00391058155807117\\
18	0.00349610161103396\\
19	0.00157703974098417\\
20	0.00107391738186312\\
21	0.000757097079944607\\
22	0.000508929815682919\\
23	0.000320062617942934\\
24	0.00020806060618036\\
25	0.000164347547785967\\
26	7.8404572643067e-05\\
27	8.07873709076011e-05\\
28	6.11121154177645e-05\\
29	3.92154393412175e-05\\
30	2.41651467034569e-05\\
31	2.27877817890705e-05\\
32	1.10988917988269e-05\\
33	5.62412579524639e-06\\
34	3.86739770870281e-06\\
35	2.9750467309139e-06\\
36	1.38368164558192e-06\\
37	9.43204339145793e-07\\
};

\end{axis}
\end{tikzpicture}%
} & 
      \raisebox{-.5\height}{
%
%
\definecolor{mycolor1}{rgb}{0.00000,0.44700,0.74100}%
\definecolor{mycolor2}{rgb}{0.85000,0.32500,0.09800}%
\definecolor{mycolor3}{rgb}{0.92900,0.69400,0.12500}%
\definecolor{mycolor4}{rgb}{0.49400,0.18400,0.55600}%
\begin{tikzpicture}

\begin{axis}[%
width=40mm,
height=35mm,
hide axis,
xmin = 0,
xmax = 1,
ymin = 0,
ymax = 1,
enlargelimits=false,
legend style={legend cell align=left,align=left,draw=white!15!black}
]
\addlegendimage{mycolor1,thick};
\addlegendentry{no preconditioner};
\addlegendimage{mycolor2,dashed,mark=*,mark options={solid},thick};
\addlegendentry{$\bfP_{\textrm{block}}$};
\addlegendimage{mycolor3,dashed,mark=triangle*,mark options={solid},thick};
\addlegendentry{$\bfP_{\textrm{Jac}}$};
\addlegendimage{mycolor4,dashed,mark=square*,mark options={solid},thick};
\addlegendentry{$\bfP_{\textrm{SGS}}$};

\end{axis}
\end{tikzpicture}%
} \\
      & \multicolumn{2}{c}{PCG-iteration} &
    \end{tabular}
   \caption{PCG performance for different preconditioning techniques in
    the first and final iteration of the coarse level of the 2D example also
    shown in \cref{fig:multilevel}, where $\alpha=200$, $\beta=10$.}
    \label{fig:preconditioners}
\end{figure}
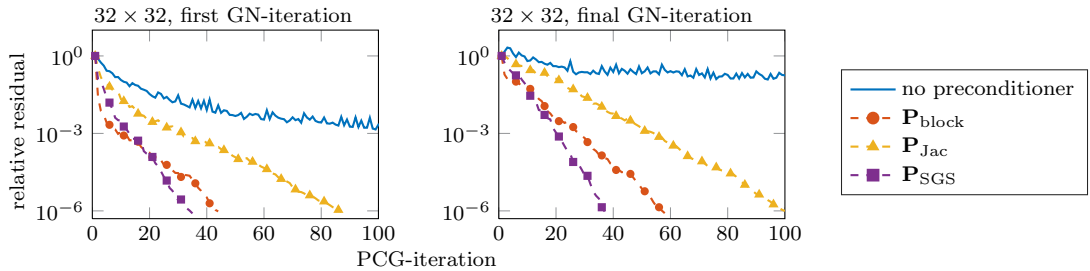
\par
Finally, \cref{fig:2Dbrain_pcg} shows the condition number of the
preconditioned approximate Hessian and the number of PCG-iterations
required  to achieve a relative residual tolerance of $10^{-1}$ in each Gauss-Newton iteration. 
Results for each level of the multilevel scheme with grid sizes of $32\times 32$, $64\times 64$, and $128\times 128$ are shown. 
The penalty parameter is fixed at $\beta=10$ but we consider two
settings of the smoothness regularizer:  $\alpha=200$ and $\alpha=2$.
It can be seen that the block-Jacobi preconditioner outperforms the Jacobi
preconditioner, while having only slightly higher computational cost per
iteration.
The total runtimes for the
three-level optimization for the example with $\alpha=200$ were 1.0~s without preconditioner,
0.7~s with the block-Jacobi, and 0.9~s with both
the Jacobi and symmetric Gauss-Seidel preconditioners. In the case of
$\alpha=2$ the runtimes remain almost unchanged for all schemes except for the Jacobi preconditioner, where the runtime
increases by a factor of about $1.8$. This increase is due to the larger number of outer iterations.
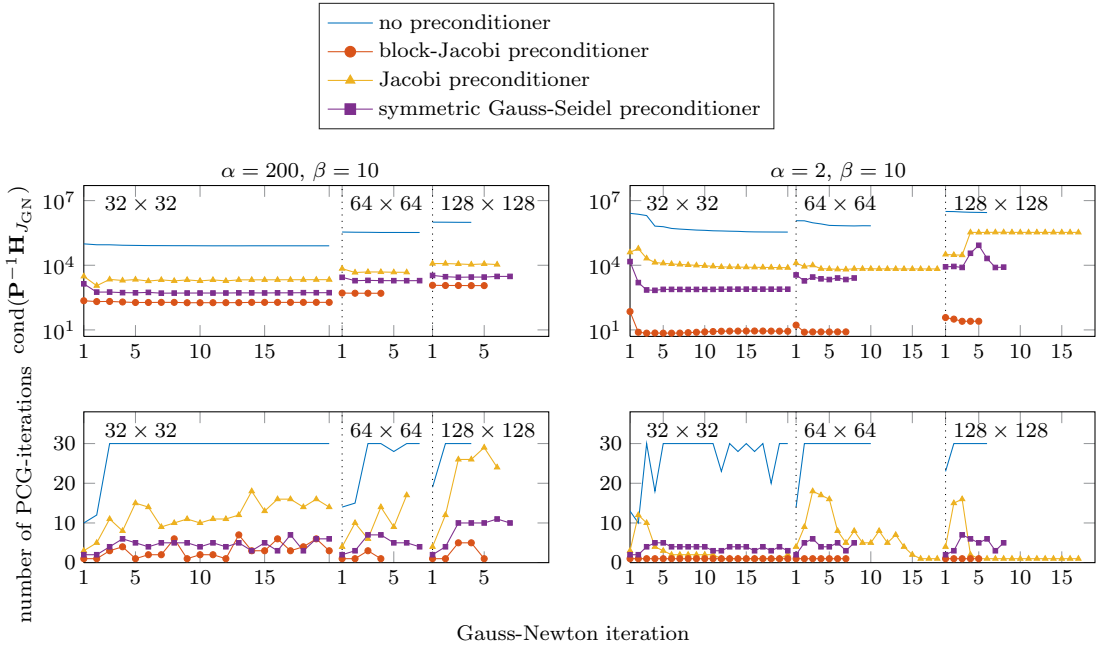
\begin{figure}[t]
	\scriptsize
  \setlength\iwidth{65mm}
  \setlength\iheight{20mm}
  \centering
%
%
\definecolor{mycolor1}{rgb}{0.00000,0.44700,0.74100}%
\definecolor{mycolor2}{rgb}{0.85000,0.32500,0.09800}%
\definecolor{mycolor3}{rgb}{0.92900,0.69400,0.12500}%
\definecolor{mycolor4}{rgb}{0.49400,0.18400,0.55600}%
\begin{tikzpicture}

\begin{axis}[%
width=70mm,
height=35mm,
hide axis,
xmin = 0,
xmax = 1,
ymin = 0,
ymax = 1,
enlargelimits=false,
legend style={legend cell align=left,align=left,draw=white!15!black}
]
\addlegendimage{mycolor1};
\addlegendentry{no preconditioner};
\addlegendimage{mycolor2,mark=*};
\addlegendentry{block-Jacobi preconditioner};
\addlegendimage{mycolor3,mark=triangle*};
\addlegendentry{Jacobi preconditioner};
\addlegendimage{mycolor4,mark=square*};
\addlegendentry{symmetric Gauss-Seidel preconditioner};

\end{axis}
\end{tikzpicture}%
  \vspace{0.5em}\\
  \begin{tabular}{r@{\,}rr}
    & \multicolumn{1}{c}{$\alpha=200$, $\beta=10$} &
    \multicolumn{1}{c}{$\alpha=2$, $\beta=10$} \\
      \rotatebox[origin=c]{90}{${\rm cond}(\bfP^{-1} \bfH_{J_{\rm
      GN}})$} &
      \raisebox{-.5\height}{
%
%
\definecolor{mycolor1}{rgb}{0.00000,0.44700,0.74100}%
\definecolor{mycolor2}{rgb}{0.85000,0.32500,0.09800}%
\definecolor{mycolor3}{rgb}{0.92900,0.69400,0.12500}%
\definecolor{mycolor4}{rgb}{0.49400,0.18400,0.55600}%
\begin{tikzpicture}

\begin{axis}[%
width=0.951\iwidth,
height=\iheight,
at={(0\iwidth,0\iheight)},
scale only axis,
xmin=1,
xmax=37,
xtick={1,5,10,15,20,21,25,28,32},
xticklabels={{1},{5},{10},{15},{},{1},{5},{1},{5}},
ymode=log,
ymin=5,
ymax=5e7,
yminorticks=true,
axis background/.style={fill=white}
]
\addplot [color=mycolor1,solid,forget plot]
  table[row sep=crcr]{%
1	98340.6312977753\\
2	89365.5852655373\\
3	89153.1743718882\\
4	84853.2677118175\\
5	83244.9497479039\\
6	81795.0751394697\\
7	81512.7388156971\\
8	80955.4454279223\\
9	80897.7610708792\\
10	80293.2345500504\\
11	79845.3838435881\\
12	79947.8199518265\\
13	79745.0679393343\\
14	80142.7495211517\\
15	80099.3961010324\\
16	80037.4254415979\\
17	80016.4044585685\\
18	79951.9874532111\\
19	79929.9953274759\\
20	79930.0819002271\\
};
\addplot [color=mycolor1,solid,forget plot]
  table[row sep=crcr]{%
21	348359.846083293\\
22	341088.053231024\\
23	338338.356187342\\
24	332167.604362036\\
25	332445.928461435\\
26	331656.455246946\\
27	331468.240157573\\
};
\addplot [color=mycolor1,solid,forget plot]
  table[row sep=crcr]{%
28	1000835.79901672\\
29	994861.378558143\\
30	983111.622869884\\
31	980618.520489869\\
};
\addplot [color=mycolor3,solid,mark=triangle*,mark options={solid},mark size=0.4mm,forget plot]
  table[row sep=crcr]{%
1	3055.06162427761\\
2	1124.24192533111\\
3	2205.49587689738\\
4	1969.42138764731\\
5	2174.27022109739\\
6	1847.8772537224\\
7	2092.92353926836\\
8	1883.7707635634\\
9	2093.82726838412\\
10	1880.71147300238\\
11	2091.59319891351\\
12	1881.46773715546\\
13	2099.43762203758\\
14	2102.35435124011\\
15	2110.20052750564\\
16	2117.17220017486\\
17	2125.03305831215\\
18	2132.96942911883\\
19	2137.94693783161\\
20	2141.83122948983\\
};
\addplot [color=mycolor3,solid,mark=triangle*,mark options={solid},mark size=0.4mm,forget plot]
  table[row sep=crcr]{%
21	6972.99284426866\\
22	4608.62023406197\\
23	4891.47460642333\\
24	4864.5703228495\\
25	4732.62245762231\\
26	4720.15528170529\\
};
\addplot [color=mycolor3,solid,mark=triangle*,mark options={solid},mark size=0.4mm,forget plot]
  table[row sep=crcr]{%
28	11966.3355237299\\
29	11971.5334046367\\
30	11503.4934562598\\
31	10816.3453757107\\
32	11563.8650519377\\
33	10796.7034056671\\
};
\addplot [color=mycolor2,solid,mark=*,mark options={solid},mark size=0.4mm,forget plot]
  table[row sep=crcr]{%
1	225.350914554379\\
2	206.635817749913\\
3	209.259822942341\\
4	197.465768164618\\
5	187.044638581395\\
6	187.518348268317\\
7	189.674498726518\\
8	189.527076594838\\
9	182.439501628308\\
10	182.65248984951\\
11	183.316204829323\\
12	183.921644848724\\
13	184.149267264644\\
14	189.434668892707\\
15	189.4485728703\\
16	189.423196595548\\
17	189.553684577425\\
18	189.640387633472\\
19	189.78760845512\\
20	190.058601357495\\
};
\addplot [color=mycolor2,solid,mark=*,mark options={solid},mark size=0.4mm,forget plot]
  table[row sep=crcr]{%
21	513.184150102365\\
22	503.24659633501\\
23	502.729850020382\\
24	500.781290404802\\
};
\addplot [color=mycolor2,solid,mark=*,mark options={solid},mark size=0.4mm,forget plot]
  table[row sep=crcr]{%
28	1176.17471219998\\
29	1151.89956836628\\
30	1150.66640402836\\
31	1137.04075065594\\
32	1136.57478197295\\
};
\addplot [color=mycolor4,solid,mark=square*,mark options={solid},mark size=0.3mm,forget plot]
  table[row sep=crcr]{%
1	1375.22464811484\\
2	567.577646412147\\
3	570.719860494603\\
4	536.165213830373\\
5	526.884112459779\\
6	550.969691751955\\
7	507.239256924583\\
8	511.14992041131\\
9	514.822278318472\\
10	514.712142964384\\
11	514.675399698064\\
12	515.586094139998\\
13	523.538749811762\\
14	523.372737072976\\
15	518.268229002964\\
16	524.785529017338\\
17	525.074288731577\\
18	534.111482040709\\
19	531.47492009616\\
20	533.395090870851\\
};
\addplot [color=mycolor4,solid,mark=square*,mark options={solid},mark size=0.3mm,forget plot]
  table[row sep=crcr]{%
21	2760.79557353572\\
22	1922.97380236406\\
23	1988.65256483238\\
24	1938.39724903281\\
25	1926.88069418609\\
26	1920.5458070595\\
27	1933.20131175864\\
};
\addplot [color=mycolor4,solid,mark=square*,mark options={solid},mark size=0.3mm,forget plot]
  table[row sep=crcr]{%
28	3366.17051008522\\
29	2935.9969543814\\
30	2775.68458913924\\
31	2831.9205697375\\
32	2823.78887066076\\
33	3034.40726647612\\
34	3041.04962605712\\
};
\addplot [color=black,dotted,forget plot]
  table[row sep=crcr]{%
21	5\\
21	5e7\\
};
\addplot [color=black,dotted,forget plot]
  table[row sep=crcr]{%
28	5\\
28	5e7\\
};
\node[right, align=left, text=black]
at (axis cs:2,7e6) {$32\times32$};
\node[right, align=left, text=black]
at (axis cs:21,7e6) {$64\times64$};
\node[right, align=left, text=black]
at (axis cs:28,7e6) {$128\times128$};
\end{axis}
\end{tikzpicture}%
} &
      \raisebox{-.5\height}{
%
%
\definecolor{mycolor1}{rgb}{0.00000,0.44700,0.74100}%
\definecolor{mycolor2}{rgb}{0.85000,0.32500,0.09800}%
\definecolor{mycolor3}{rgb}{0.92900,0.69400,0.12500}%
\definecolor{mycolor4}{rgb}{0.49400,0.18400,0.55600}%
\begin{tikzpicture}

\begin{axis}[%
width=0.951\iwidth,
height=\iheight,
at={(0\iwidth,0\iheight)},
scale only axis,
xmin=1,
xmax=57,
xtick={1,5,10,15,20,21,25,30,35,39,43,48,53},
xticklabels={{1},{5},{10},{15},{},{1},{5},{10},{15},{1},{5},{10},{15}},
ymode=log,
ymin=5,
ymax=5e7,
yminorticks=true,
axis background/.style={fill=white}
]
\addplot [color=mycolor1,solid,forget plot]
  table[row sep=crcr]{%
1	2563225.26052441\\
2	2331163.06786878\\
3	2020608.83678517\\
4	654263.783205068\\
5	616074.090753123\\
6	511497.920326548\\
7	477235.918922358\\
8	453146.481813983\\
9	429531.529879368\\
10	417226.764663647\\
11	397911.159538466\\
12	391044.733630783\\
13	381129.170268131\\
14	374532.243120336\\
15	360127.750735019\\
16	353275.336161657\\
17	350672.028172136\\
18	348881.046820435\\
19	347503.143339383\\
20	346038.231177164\\
};
\addplot [color=mycolor1,solid,forget plot]
  table[row sep=crcr]{%
21	1167478.13840527\\
22	1155420.6705293\\
23	935831.733523898\\
24	828024.307452449\\
25	708257.230130019\\
26	691118.358998402\\
27	678055.434753485\\
28	667871.362609575\\
29	682666.542743814\\
30	678805.58098166\\
};
\addplot [color=mycolor1,solid,forget plot]
  table[row sep=crcr]{%
39	3131907.10948529\\
40	3095870.47370291\\
41	2958744.34286634\\
42	2870741.79654976\\
43	2834765.14473525\\
44	2789713.55407592\\
};
\addplot [color=mycolor3,solid,mark=triangle*,mark options={solid},mark size=0.4mm,forget plot]
  table[row sep=crcr]{%
1	40082.5344739563\\
2	59139.9152566619\\
3	20840.8813405072\\
4	13395.8861567695\\
5	12328.3025164955\\
6	11476.7330944684\\
7	10916.8657244348\\
8	10374.5857582548\\
9	9925.04332330195\\
10	9492.74297676161\\
11	8907.96942825014\\
12	8349.58069036379\\
13	8274.40986655516\\
14	8199.07319482932\\
15	8121.98065826125\\
16	8042.65888005439\\
17	7961.29214390958\\
18	7876.76103017554\\
19	7790.35407592655\\
20	7701.31026640422\\
};
\addplot [color=mycolor3,solid,mark=triangle*,mark options={solid},mark size=0.4mm,forget plot]
  table[row sep=crcr]{%
21	12630.6183659624\\
22	8745.61286040032\\
23	10008.2052941471\\
24	6889.62441507316\\
25	6789.22168271299\\
26	6503.61784653331\\
27	6393.8810763142\\
28	6906.25449626804\\
29	6880.48381534515\\
30	6887.01978704456\\
31	6796.2131605416\\
32	6779.51376933856\\
33	6799.02510462655\\
34	6755.77563699451\\
35	6759.85435277472\\
36	6743.13250035088\\
37	6747.98891945422\\
38	6750.06614519422\\
};
\addplot [color=mycolor3,solid,mark=triangle*,mark options={solid},mark size=0.4mm,forget plot]
  table[row sep=crcr]{%
39	31783.4913826325\\
40	30074.1004904267\\
41	29510.0547144222\\
42	335847.078307298\\
43	332921.114343866\\
44	335761.903836043\\
45	337870.156135608\\
46	339251.679675562\\
47	335393.760456969\\
48	335449.279735347\\
49	335214.531550748\\
50	335110.306038186\\
51	334987.84385943\\
52	334954.405377922\\
53	334937.23089071\\
54	334935.084218299\\
55	334932.899021243\\
};
\addplot [color=mycolor2,solid,mark=*,mark options={solid},mark size=0.4mm,forget plot]
  table[row sep=crcr]{%
1	70.9167991685705\\
2	7.78180223299617\\
3	7.07777017953901\\
4	7.19192575653548\\
5	7.06676796703401\\
6	7.07474067481712\\
7	7.12726872482015\\
8	7.45870415087379\\
9	7.77120668613082\\
10	8.07154472838908\\
11	8.40219007787537\\
12	8.66438403151116\\
13	8.77324332908073\\
14	8.82826968270558\\
15	8.87140177043888\\
16	8.90595726101649\\
17	8.87186446111567\\
18	8.78404270373374\\
19	8.72584941926861\\
20	8.65997101260326\\
};
\addplot [color=mycolor2,solid,mark=*,mark options={solid},mark size=0.4mm,forget plot]
  table[row sep=crcr]{%
21	16.6611914604651\\
22	7.84673313337444\\
23	8.05460309286397\\
24	8.06591603444064\\
25	8.07427982082185\\
26	8.09065989152665\\
27	8.09695022182909\\
};
\addplot [color=mycolor2,solid,mark=*,mark options={solid},mark size=0.4mm,forget plot]
  table[row sep=crcr]{%
39	37.415595435646\\
40	31.6741402761548\\
41	25.1110909892754\\
42	25.1891085862564\\
43	25.1902702721988\\
};
\addplot [color=mycolor4,solid,mark=square*,mark options={solid},mark size=0.3mm,forget plot]
  table[row sep=crcr]{%
1	14781.7079217532\\
2	1568.01008637856\\
3	714.769253031661\\
4	685.483782684131\\
5	760.100107158291\\
6	760.555500127673\\
7	758.279321407105\\
8	757.550893686061\\
9	757.825008336682\\
10	758.125515142309\\
11	758.916458311068\\
12	772.533872045642\\
13	774.55262868826\\
14	776.260794991341\\
15	777.096167010449\\
16	777.723771510108\\
17	777.89979201249\\
18	775.497472502975\\
19	774.806236083949\\
20	773.887295930177\\
};
\addplot [color=mycolor4,solid,mark=square*,mark options={solid},mark size=0.3mm,forget plot]
  table[row sep=crcr]{%
21	3558.48381087243\\
22	1889.12125347022\\
23	2837.11088135001\\
24	2348.82453591456\\
25	2173.66861001808\\
26	2558.04299173205\\
27	2180.52205027992\\
28	2564.74407903732\\
};
\addplot [color=mycolor4,solid,mark=square*,mark options={solid},mark size=0.3mm,forget plot]
  table[row sep=crcr]{%
39	8485.46026186219\\
40	8796.52573948418\\
41	7926.69953957378\\
42	36041.9415622757\\
43	84448.9765020785\\
44	20703.40385303\\
45	7888.73922199219\\
46	8182.73287346472\\
};
\addplot [color=black,dotted,forget plot]
  table[row sep=crcr]{%
21	5\\
21	5e7\\
};
\addplot [color=black,dotted,forget plot]
  table[row sep=crcr]{%
39	5\\
39	5e7\\
};
\node[right, align=left, text=black]
at (axis cs:2,7e6) {$32\times32$};
\node[right, align=left, text=black]
at (axis cs:21,7e6) {$64\times64$};
\node[right, align=left, text=black]
at (axis cs:39,7e6) {$128\times128$};
\end{axis}
\end{tikzpicture}%
} \\
      \rotatebox[origin=c]{90}{number of PCG-iterations} &
      \raisebox{-.5\height}{
%
%
\definecolor{mycolor1}{rgb}{0.00000,0.44700,0.74100}%
\definecolor{mycolor2}{rgb}{0.85000,0.32500,0.09800}%
\definecolor{mycolor3}{rgb}{0.92900,0.69400,0.12500}%
\definecolor{mycolor4}{rgb}{0.49400,0.18400,0.55600}%
\begin{tikzpicture}

\begin{axis}[%
width=0.951\iwidth,
height=\iheight,
at={(0\iwidth,0\iheight)},
scale only axis,
xmin=1,
xmax=37,
xtick={1,5,10,15,20,21,25,28,32},
xticklabels={{1},{5},{10},{15},{},{1},{5},{1},{5}},
ymin=0,
ymax=38,
axis background/.style={fill=white}
]
\addplot [color=mycolor1,solid,forget plot]
  table[row sep=crcr]{%
1	10\\
2	12\\
3	30\\
4	30\\
5	30\\
6	30\\
7	30\\
8	30\\
9	30\\
10	30\\
11	30\\
12	30\\
13	30\\
14	30\\
15	30\\
16	30\\
17	30\\
18	30\\
19	30\\
20	30\\
};
\addplot [color=mycolor1,solid,forget plot]
  table[row sep=crcr]{%
21	14\\
22	15\\
23	30\\
24	30\\
25	28\\
26	30\\
27	30\\
};
\addplot [color=mycolor1,solid,forget plot]
  table[row sep=crcr]{%
28	19\\
29	30\\
30	30\\
31	30\\
};
\addplot [color=mycolor3,solid,mark=triangle*,mark options={solid},mark size=0.4mm,forget plot]
  table[row sep=crcr]{%
1	3\\
2	5\\
3	11\\
4	8\\
5	15\\
6	14\\
7	9\\
8	10\\
9	11\\
10	10\\
11	11\\
12	11\\
13	12\\
14	18\\
15	13\\
16	16\\
17	16\\
18	14\\
19	16\\
20	14\\
};
\addplot [color=mycolor3,solid,mark=triangle*,mark options={solid},mark size=0.4mm,forget plot]
  table[row sep=crcr]{%
21	4\\
22	10\\
23	6\\
24	14\\
25	9\\
26	17\\
};
\addplot [color=mycolor3,solid,mark=triangle*,mark options={solid},mark size=0.4mm,forget plot]
  table[row sep=crcr]{%
28	4\\
29	12\\
30	26\\
31	26\\
32	29\\
33	24\\
};
\addplot [color=mycolor2,solid,mark=*,mark options={solid},mark size=0.4mm,forget plot]
  table[row sep=crcr]{%
1	1\\
2	1\\
3	3\\
4	4\\
5	1\\
6	2\\
7	2\\
8	6\\
9	1\\
10	2\\
11	2\\
12	1\\
13	7\\
14	3\\
15	3\\
16	6\\
17	3\\
18	4\\
19	6\\
20	3\\
};
\addplot [color=mycolor2,solid,mark=*,mark options={solid},mark size=0.4mm,forget plot]
  table[row sep=crcr]{%
21	1\\
22	1\\
23	3\\
24	1\\
};
\addplot [color=mycolor2,solid,mark=*,mark options={solid},mark size=0.4mm,forget plot]
  table[row sep=crcr]{%
28	1\\
29	1\\
30	5\\
31	5\\
32	1\\
};
\addplot [color=mycolor4,solid,mark=square*,mark options={solid},mark size=0.3mm,forget plot]
  table[row sep=crcr]{%
1	2\\
2	2\\
3	4\\
4	6\\
5	5\\
6	4\\
7	5\\
8	5\\
9	5\\
10	4\\
11	5\\
12	4\\
13	5\\
14	3\\
15	5\\
16	3\\
17	7\\
18	3\\
19	6\\
20	6\\
};
\addplot [color=mycolor4,solid,mark=square*,mark options={solid},mark size=0.3mm,forget plot]
  table[row sep=crcr]{%
21	2\\
22	3\\
23	7\\
24	7\\
25	5\\
26	5\\
27	4\\
};
\addplot [color=mycolor4,solid,mark=square*,mark options={solid},mark size=0.3mm,forget plot]
  table[row sep=crcr]{%
28	2\\
29	4\\
30	10\\
31	10\\
32	10\\
33	11\\
34	10\\
};
\addplot [color=black,dotted,forget plot]
  table[row sep=crcr]{%
21	0\\
21	38\\
};
\addplot [color=black,dotted,forget plot]
  table[row sep=crcr]{%
28	0\\
28	38\\
};
\node[right, align=left, text=black]
at (axis cs:2,33) {$32\times32$};
\node[right, align=left, text=black]
at (axis cs:21,33) {$64\times64$};
\node[right, align=left, text=black]
at (axis cs:28,33) {$128\times128$};
\end{axis}
\end{tikzpicture}%
} &
      \raisebox{-.5\height}{
%
%
\definecolor{mycolor1}{rgb}{0.00000,0.44700,0.74100}%
\definecolor{mycolor2}{rgb}{0.85000,0.32500,0.09800}%
\definecolor{mycolor3}{rgb}{0.92900,0.69400,0.12500}%
\definecolor{mycolor4}{rgb}{0.49400,0.18400,0.55600}%
\begin{tikzpicture}

\begin{axis}[%
width=0.951\iwidth,
height=\iheight,
at={(0\iwidth,0\iheight)},
scale only axis,
xmin=1,
xmax=57,
xtick={1,5,10,15,20,21,25,30,35,39,43,48,53},
xticklabels={{1},{5},{10},{15},{},{1},{5},{10},{15},{1},{5},{10},{15}},
ymin=0,
ymax=38,
axis background/.style={fill=white}
]
\addplot [color=mycolor1,solid,forget plot]
  table[row sep=crcr]{%
1	13\\
2	10\\
3	30\\
4	18\\
5	30\\
6	30\\
7	30\\
8	30\\
9	30\\
10	30\\
11	30\\
12	23\\
13	30\\
14	28\\
15	30\\
16	28\\
17	30\\
18	20\\
19	30\\
20	30\\
};
\addplot [color=mycolor1,solid,forget plot]
  table[row sep=crcr]{%
21	14\\
22	30\\
23	30\\
24	30\\
25	30\\
26	30\\
27	30\\
28	30\\
29	30\\
30	30\\
};
\addplot [color=mycolor1,solid,forget plot]
  table[row sep=crcr]{%
39	23\\
40	30\\
41	30\\
42	30\\
43	30\\
44	30\\
};
\addplot [color=mycolor3,solid,mark=triangle*,mark options={solid},mark size=0.4mm,forget plot]
  table[row sep=crcr]{%
1	3\\
2	12\\
3	10\\
4	4\\
5	3\\
6	2\\
7	2\\
8	2\\
9	2\\
10	2\\
11	2\\
12	1\\
13	1\\
14	1\\
15	1\\
16	1\\
17	1\\
18	1\\
19	1\\
20	2\\
};
\addplot [color=mycolor3,solid,mark=triangle*,mark options={solid},mark size=0.4mm,forget plot]
  table[row sep=crcr]{%
21	4\\
22	9\\
23	18\\
24	17\\
25	16\\
26	8\\
27	5\\
28	8\\
29	5\\
30	5\\
31	8\\
32	5\\
33	7\\
34	4\\
35	2\\
36	1\\
37	1\\
38	1\\
};
\addplot [color=mycolor3,solid,mark=triangle*,mark options={solid},mark size=0.4mm,forget plot]
  table[row sep=crcr]{%
39	4\\
40	15\\
41	16\\
42	2\\
43	1\\
44	1\\
45	1\\
46	1\\
47	1\\
48	1\\
49	1\\
50	1\\
51	1\\
52	1\\
53	1\\
54	1\\
55	1\\
};
\addplot [color=mycolor2,solid,mark=*,mark options={solid},mark size=0.4mm,forget plot]
  table[row sep=crcr]{%
1	1\\
2	1\\
3	1\\
4	1\\
5	1\\
6	1\\
7	1\\
8	1\\
9	1\\
10	1\\
11	1\\
12	1\\
13	1\\
14	1\\
15	1\\
16	1\\
17	1\\
18	1\\
19	1\\
20	1\\
};
\addplot [color=mycolor2,solid,mark=*,mark options={solid},mark size=0.4mm,forget plot]
  table[row sep=crcr]{%
21	1\\
22	1\\
23	1\\
24	1\\
25	1\\
26	1\\
27	1\\
};
\addplot [color=mycolor2,solid,mark=*,mark options={solid},mark size=0.4mm,forget plot]
  table[row sep=crcr]{%
39	1\\
40	1\\
41	1\\
42	1\\
43	1\\
};
\addplot [color=mycolor4,solid,mark=square*,mark options={solid},mark size=0.3mm,forget plot]
  table[row sep=crcr]{%
1	2\\
2	2\\
3	4\\
4	5\\
5	5\\
6	4\\
7	4\\
8	4\\
9	4\\
10	4\\
11	3\\
12	3\\
13	4\\
14	4\\
15	4\\
16	3\\
17	4\\
18	3\\
19	4\\
20	3\\
};
\addplot [color=mycolor4,solid,mark=square*,mark options={solid},mark size=0.3mm,forget plot]
  table[row sep=crcr]{%
21	2\\
22	5\\
23	6\\
24	4\\
25	4\\
26	5\\
27	3\\
28	5\\
};
\addplot [color=mycolor4,solid,mark=square*,mark options={solid},mark size=0.3mm,forget plot]
  table[row sep=crcr]{%
39	2\\
40	3\\
41	7\\
42	6\\
43	5\\
44	6\\
45	3\\
46	5\\
};
\addplot [color=black,dotted,forget plot]
  table[row sep=crcr]{%
21	1\\
21	38\\
};
\addplot [color=black,dotted,forget plot]
  table[row sep=crcr]{%
39	1\\
39	38\\
};
\node[right, align=left, text=black]
at (axis cs:2,33) {$32\times32$};
\node[right, align=left, text=black]
at (axis cs:21,33) {$64\times64$};
\node[right, align=left, text=black]
at (axis cs:39,33) {$128\times128$};
\end{axis}
\end{tikzpicture}%
}\\
      & \multicolumn{2}{c}{Gauss-Newton iteration}
    \end{tabular}
  \caption{Condition number of preconditioned Hessian~\cref{eq:dJGN} and
  number of required PCG-iterations to achieve a relative residual
  tolerance of $10^{-1}$ for different options of preconditioners described
  in \cref{sub:gauss_newton_pcg}. Results are shown for different
  discretization levels (divided by vertical dotted lines) using the 2D
  test data also shown in \cref{fig:multilevel} with $\alpha=200$,
  $\beta=10$ and $\alpha=2$, $\beta=10$ respectively.}
  \label{fig:2Dbrain_pcg}
\end{figure}


\paragraph{Performance Of Preconditioners In 3D} 
\label{par:3d_example}
We compare the performance of the different preconditioners for the 3D
data set described above using $\alpha=50$ and $\beta=10$.
Detailed convergence results are provided in \cref{tab:3Dbrain_pcg}.  
We show the decrease of the objective function value, the norm of the
gradient and the number of PCG-iterations
needed to solve~\cref{eq:GNstep} at each Gauss-Newton iteration.
Here, the Gauss-Newton iterations $k=-1$ and
$k=0$ correspond to $\bfb=0$ and the initial guess
obtained by prolongation from the previous level, respectively.
As in the previous example, the block-Jacobi preconditioner outperforms both
the Jacobi and the symmetric Gauss-Seidel preconditioner. In many
Gauss-Newton iterations only one or two PCG-iterations are needed. Furthermore the
savings in computational cost per PCG-iteration when using the
block-Jacobi instead of the symmetric Gauss-Seidel preconditioner are
even more substantial in the 3D case. 
On the finest level, neither of the three methods satisfies the first stopping criterion~\eqref{eq:stopping1} after 10 iterations, which is the default setting in HySCO~\cite{RuthottoEtAl2013hysco}. In our experiments we found that increasing the number of iterations does not lead to considerable improvements of reconstruction quality.
In total the block-Jacobi scheme
takes 142~s, the Jacobi scheme takes 233~s, and the
symmetric Gauss-Seidel scheme takes 316~s for the three-level
optimization.
\begin{table}[t]
  \caption{Convergence results for the 3D correction problem with
  $\alpha =50$ and $\beta=10$ for the different preconditioning strategies
  described in \cref{sub:gauss_newton_pcg}. For each Gauss-Newton-iteration, we show
  the objective function value, the gradient norm,
  the number of PCG-iterations, and the relative residual of the PCG-solve. Empty rows indicate that the respective method has already converged with respect to the prescribed tolerances. On the finest level, all methods reach the maximum number of iterations without achieving the first condition in~\eqref{eq:stopping1}.}
  \label{tab:3Dbrain_pcg}
  \centering
  \scriptsize
\begin{tabular}{@{}|@{\,\,}c@{\,\,}|@{\,}c@{\,}|@{\,\,}rrrr@{\,\,}|@{\,\,}rrrr@{\,\,}|@{\,\,}rrrr@{\,\,}|@{}}
\hline
& GN
&\multicolumn{4}{c@{\,\,}|@{\,\,}}{Jacobi}
&\multicolumn{4}{c@{\,\,}|@{\,\,}}{symmetric Gauss-Seidel} 
&\multicolumn{4}{c@{\,\,}|@{}}{block-Jacobi}
\\

& iter
& \multicolumn{1}{@{\,\,}c}{$J_{\rm GN}$} & \multicolumn{1}{c}{$\|\nabla J_{\rm GN}\|$} & \multicolumn{1}{c}{iter} & \multicolumn{1}{c@{\,\,}|@{\,\,}}{rel. res.}
& \multicolumn{1}{@{\,\,}c}{$J_{\rm GN}$} & \multicolumn{1}{c}{$\|\nabla J_{\rm GN}\|$} & \multicolumn{1}{c}{iter} & \multicolumn{1}{c@{\,\,}|@{\,\,}}{rel. res.}
& \multicolumn{1}{@{\,\,}c}{$J_{\rm GN}$} & \multicolumn{1}{c}{$\|\nabla J_{\rm GN}\|$} & \multicolumn{1}{c}{iter} & \multicolumn{1}{c@{\,\,}|@{}}{rel. res.}
\\ \hline\hline

\multirow{12}{*}{\rotatebox[origin=center]{90}{$50\times50\times33$}} 
& -1 & 1.02e8 & --     & -- & --      & 1.02e8 & --     & -- & --      & 1.02e8 & --     & -- & --      \\
&  0 & 1.02e8 & 3.49e6 & -- & --      & 1.02e8 & 3.49e6 & -- & --      & 1.02e8 & 3.49e6 & -- & --      \\
&  1 & 3.48e7 & 1.59e6 &  4 & 8.53e-2 & 3.21e7 & 1.53e6 &  2 & 8.72e-2 & 3.04e7 & 1.46e6 &  1 & 3.81e-2 \\
&  2 & 1.61e7 & 7.84e5 &  7 & 7.88e-2 & 1.47e7 & 7.48e5 &  3 & 8.81e-2 & 1.58e7 & 6.55e5 &  1 & 9.32e-2 \\
&  3 & 1.00e7 & 4.08e5 &  9 & 8.68e-2 & 9.30e6 & 3.87e5 &  4 & 7.69e-2 & 1.03e7 & 3.62e5 &  2 & 5.93e-2 \\
&  4 & 7.92e6 & 2.21e5 & 10 & 9.37e-2 & 7.60e6 & 2.05e5 &  4 & 9.59e-2 & 8.32e6 & 1.96e5 &  2 & 7.76e-2 \\
&  5 & 6.96e6 & 1.24e5 & 11 & 9.36e-2 & 6.89e6 & 1.13e5 &  5 & 7.06e-2 & 7.43e6 & 2.30e5 &  3 & 6.68e-2 \\
&  6 & 6.84e6 & 8.55e4 & 10 & 9.77e-2 & 6.72e6 & 4.47e4 &  4 & 9.13e-2 & 6.98e6 & 1.01e5 &  1 & 4.05e-2 \\
&  7 & 6.67e6 & 2.28e4 &  7 & 9.82e-2 & 6.67e6 & 4.26e4 &  6 & 8.21e-2 & 6.90e6 & 4.48e4 &  1 & 6.57e-2 \\
&  8 & 6.65e6 & 2.85e4 & 19 & 9.39e-2 & 6.66e6 & 3.52e4 &  4 & 7.96e-2 & 6.73e6 & 3.82e4 &  5 & 7.27e-2 \\
&  9 & 6.64e6 & 2.17e4 & 16 & 9.73e-2 & 6.65e6 & 2.86e4 &  3 & 9.04e-2 & 6.72e6 & 3.13e4 &  1 & 8.80e-2 \\
& 10 & 6.63e6 & 1.89e4 & 15 & 9.62e-2 & 6.64e6 & 2.23e4 &  4 & 7.58e-2 & 6.72e6 & 4.33e4 &  2 & 9.84e-2 \\
\hline\hline

\multirow{12}{*}{\rotatebox[origin=center]{90}{$100\times100\times66$}} 
& -1 & 2.09e8 & --     & -- & --      & 2.09e8 & --     & -- & --      & 2.09e8 & --     & -- & --      \\
&  0 & 5.20e7 & 1.94e6 & -- & --      & 5.20e7 & 1.95e6 & -- & --      & 5.22e7 & 1.95e6 & -- & --      \\
&  1 & 2.44e7 & 9.96e5 &  4 & 9.63e-2 & 2.35e7 & 9.95e5 &  2 & 6.94e-2 & 2.31e7 & 9.57e5 &  1 & 2.24e-2 \\
&  2 & 1.65e7 & 5.08e5 &  6 & 8.68e-2 & 1.93e7 & 7.53e5 &  2 & 9.07e-2 & 1.61e7 & 4.68e5 &  1 & 3.48e-2 \\
&  3 & 1.42e7 & 2.62e5 &  8 & 8.50e-2 & 1.49e7 & 3.85e5 &  3 & 6.41e-2 & 1.42e7 & 2.25e5 &  1 & 5.82e-2 \\
&  4 & 1.36e7 & 1.43e5 & 10 & 8.92e-2 & 1.42e7 & 2.93e5 &  3 & 8.02e-2 & 1.36e7 & 4.04e5 &  1 & 9.64e-2 \\
&  5 & 1.33e7 & 8.57e4 & 12 & 9.11e-2 & 1.35e7 & 1.54e5 &  3 & 8.92e-2 & 1.35e7 & 2.57e5 &  1 & 2.31e-2 \\
&  6 & 1.32e7 & 6.76e4 & 13 & 9.74e-2 & 1.35e7 & 1.34e5 &  4 & 9.94e-2 & 1.34e7 & 1.60e5 &  1 & 7.73e-2 \\
&  7 & 1.32e7 & 6.32e4 & 13 & 9.92e-2 & 1.33e7 & 7.97e4 &  4 & 8.63e-2 & 1.33e7 & 1.27e5 &  1 & 8.72e-2 \\
&  8 & 1.32e7 & 5.68e4 & 12 & 9.73e-2 & 1.32e7 & 7.61e4 &  5 & 8.13e-2 &        &        &    &         \\
&  9 & 1.31e7 & 5.18e4 & 13 & 9.15e-2 & 1.32e7 & 7.10e4 &  5 & 7.53e-2 &        &        &    &         \\
& 10 &        &        &    &         & 1.31e7 & 4.78e4 &  4 & 6.96e-2 &        &        &    &         \\
\hline\hline

\multirow{12}{*}{\rotatebox[origin=center]{90}{$200\times200\times132$}} 
& -1 & 3.35e8 & --     & -- & --      & 3.35e8 & --     & -- & --      & 3.35e8 & --     & -- & --      \\
&  0 & 8.04e7 & 1.64e6 & -- & --      & 8.01e7 & 1.64e6 & -- & --      & 8.07e7 & 1.63e6 & -- & --      \\
&  1 & 6.86e7 & 1.45e6 &  6 & 8.93e-2 & 5.73e7 & 1.25e6 &  3 & 5.83e-2 & 6.85e7 & 1.44e6 &  1 & 2.11e-2 \\
&  2 & 5.95e7 & 1.28e6 &  6 & 9.40e-2 & 5.05e7 & 1.10e6 &  3 & 6.73e-2 & 5.91e7 & 1.27e6 &  1 & 2.24e-2 \\
&  3 & 5.24e7 & 1.13e6 &  6 & 9.98e-2 & 4.53e7 & 9.70e5 &  3 & 7.16e-2 & 5.54e7 & 1.19e6 &  1 & 2.40e-2 \\
&  4 & 4.69e7 & 9.91e5 &  7 & 8.88e-2 & 4.12e7 & 8.53e5 &  3 & 7.59e-2 & 4.90e7 & 1.05e6 &  1 & 2.48e-2 \\
&  5 & 4.26e7 & 8.71e5 &  7 & 9.36e-2 & 3.81e7 & 7.50e5 &  3 & 8.04e-2 & 4.41e7 & 9.16e5 &  1 & 2.66e-2 \\
&  6 & 3.93e7 & 7.66e5 &  7 & 1.00e-1 & 3.57e7 & 6.59e5 &  3 & 8.38e-2 & 4.04e7 & 8.02e5 &  1 & 2.87e-2 \\
&  7 & 3.66e7 & 6.74e5 &  8 & 9.19e-2 & 3.38e7 & 5.79e5 &  3 & 8.89e-2 & 3.75e7 & 7.02e5 &  1 & 3.13e-2 \\
&  8 & 3.46e7 & 5.92e5 &  8 & 9.74e-2 & 3.23e7 & 5.08e5 &  3 & 9.29e-2 & 3.53e7 & 6.13e5 &  1 & 3.39e-2 \\
&  9 & 3.30e7 & 5.20e5 &  9 & 8.94e-2 & 3.12e7 & 4.47e5 &  3 & 9.64e-2 & 3.36e7 & 5.36e5 &  1 & 3.73e-2 \\
& 10 & 3.17e7 & 4.57e5 &  9 & 9.54e-2 & 2.95e7 & 7.01e6 &  4 & 7.21e-2 & 3.22e7 & 4.68e5 &  1 & 4.09e-2 \\
\hline
\end{tabular}

\end{table}

\paragraph{Performance Of ADMM} 
\label{par:performance_of_admm}
We study the performance of the proposed ADMM scheme and its dependence on the choice of the augmentation parameter $\rho$  using the 3D data set described above.
The convergence of the primal residual, dual residual, and
$\rho$ for the adaptive method, the fixed parameter method, and the adaptive method
with lower bound, are visualized in \cref{fig:3Dbrain_admm}. 
In the fixed case we use $\rho=10^{2}$, which is also the lower bound in the
bounded adaptive method. For both adaptive schemes the initial parameter
is $\rho^0=10^6$. As discussed in \cref{sub:multilevel}, we  average $\bfb$ and $\bfz$ after prolongation to start the next
level.
Clearly, larger values for $\rho$ result in a smaller primal
residual, whereas smaller values for $\rho$ result in a smaller dual
residual. The benefit of the adaptive method with lower bound is, that
it forces the primal residual to converge fast during the first
iterations and afterwards keeps the primal residual small while the dual residual
converges as well. This way the equality constraint in~\cref{eq:admmprob}
is almost satisfied during most iterations, meaning that we indeed solve
the problem~\cref{eq:admmprob}. \par 
Note that, while $\rho$ changes dramatically on the coarsest level, it remains constant on the finer levels, where computations are more costly. 
The separability of the first -- and in our experience most expensive -- ADMM subproblem~\cref{eq:UadmmStep1} provides a way for substantial speed up by using parallel computing. In principle, the data of each image column can be processed completely in parallel. To strike a balance between communication overhead and computations, in our current MATLAB implementation we correct all image slices in parallel.
\begin{figure}[t]
	\scriptsize
  \setlength\iwidth{100mm}
  \setlength\iheight{18mm}
  \centering
  \begin{tabular}{r@{\,}r}
      & \multicolumn{1}{c}{\raisebox{-.5\height}{
%
%
\definecolor{mycolor1}{rgb}{0.00000,0.44700,0.74100}%
\definecolor{mycolor2}{rgb}{0.85000,0.32500,0.09800}%
\definecolor{mycolor3}{rgb}{0.92900,0.69400,0.12500}%
\definecolor{mycolor4}{rgb}{0.49400,0.18400,0.55600}%
\begin{tikzpicture}

\begin{axis}[%
width=70mm,
height=35mm,
hide axis,
xmin = 0,
xmax = 1,
ymin = 0,
ymax = 1,
enlargelimits=false,
legend style={legend cell align=left,align=left,draw=white!15!black}
]
\addlegendimage{mycolor2,mark=x};
\addlegendentry{adaptive with $\rho^0=10^6$ and $\rho_{\textrm{min}}=10^{2}$};
\addlegendimage{mycolor3,mark=*};
\addlegendentry{adaptive with $\rho^0 = 10^6$};
\addlegendimage{mycolor4,mark=asterisk};
\addlegendentry{fixed $\rho=10^2$};
\end{axis}
\end{tikzpicture}%
}}\\

      \rotatebox[origin=c]{90}{primal residual} &
      \raisebox{-.5\height}{
%
%
\definecolor{mycolor1}{rgb}{0.85000,0.32500,0.09800}%
\definecolor{mycolor2}{rgb}{0.92900,0.69400,0.12500}%
\definecolor{mycolor3}{rgb}{0.49400,0.18400,0.55600}%
\begin{tikzpicture}

\begin{axis}[%
width=0.951\iwidth,
height=\iheight,
at={(0\iwidth,0\iheight)},
scale only axis,
xmin=1,
xmax=102,
xtick={1,20,40,51,70,81,90},
xticklabels={{1},{20},{40},{1},{20},{1},{10}},
ymode=log,
ymin=1e-06,
ymax=1e4,
yminorticks=true,
axis background/.style={fill=white}
]
\addplot [color=mycolor2,solid,mark=*,mark options={solid},mark size=0.4mm,forget plot]
  table[row sep=crcr]{%
1	4.0270922149947e-06\\
2	7.20218845719601e-06\\
3	1.21213240094244e-05\\
4	2.03521389631773e-05\\
5	3.41084564600954e-05\\
6	5.70240857899927e-05\\
7	9.52257081089007e-05\\
8	0.000158114200744504\\
9	0.000261275340856521\\
10	0.000429196768576031\\
11	0.00069996179956985\\
12	0.00113349587485701\\
13	0.00181427170873687\\
14	0.00286898255076285\\
15	0.00448121488050755\\
16	0.00690181731144229\\
17	0.0104845931295637\\
18	0.0157234333129529\\
19	0.0232526719870589\\
20	0.033997502513106\\
21	0.0490781023765818\\
22	0.0702026375434463\\
23	0.0997028742401245\\
24	0.139833449934267\\
25	0.194924953302421\\
26	0.27157995568542\\
27	0.376297334672662\\
28	0.517964399724777\\
29	0.706988962323782\\
30	0.966010397379396\\
31	1.30409510325439\\
32	1.746972476753\\
33	2.36604172463565\\
34	3.21978227793878\\
35	4.3253567612203\\
36	5.81081972929069\\
37	7.75042612648997\\
38	10.2289580787846\\
39	14.0233244109084\\
40	18.4969675268079\\
41	25.5639378237604\\
42	31.730983660969\\
43	39.4311323919886\\
44	50.5404477423113\\
45	57.2146992926342\\
46	69.4387405023524\\
47	84.9076569827735\\
48	89.4495491689222\\
49	111.224399456795\\
50	124.427213834029\\
};
\addplot [color=mycolor2,solid,mark=*,mark options={solid},mark size=0.4mm,forget plot]
  table[row sep=crcr]{%
51	902.264438119301\\
52	651.883526271296\\
53	515.947018099874\\
54	430.704298545459\\
};
\addplot [color=mycolor2,solid,mark=*,mark options={solid},mark size=0.4mm,forget plot]
  table[row sep=crcr]{%
81	762.189775521876\\
82	755.27269522897\\
83	719.892683642454\\
84	653.850567439064\\
85	591.819730308844\\
86	542.426862967811\\
87	501.213712191434\\
};
\addplot [color=mycolor3,solid,mark=asterisk,mark options={solid},mark size=0.5mm,forget plot]
  table[row sep=crcr]{%
1	36.7162157067117\\
2	16.1832994602996\\
3	9.52029811840417\\
4	6.98709577249599\\
5	5.84663179189216\\
6	4.95512250986258\\
7	4.50111208146666\\
8	3.96918345688912\\
9	3.92535235622612\\
10	3.62737343072612\\
11	3.03154730884259\\
12	3.11898885453895\\
13	2.89358364293971\\
14	3.09029657793698\\
15	2.64462893499775\\
16	2.46692784861641\\
17	2.32937306810851\\
18	2.28595596686339\\
19	2.24934581076129\\
20	2.18575583183578\\
21	2.19184124390032\\
22	2.08666848261858\\
23	2.51431213481839\\
24	2.0941557976641\\
25	2.2155026488443\\
26	2.02548340661559\\
27	1.95669903473165\\
28	1.83237922826149\\
29	1.93005079855606\\
30	1.61627232068177\\
31	2.25817355835671\\
32	1.98838655084846\\
33	1.71046344427502\\
34	1.80393987843728\\
35	1.85457496985889\\
36	1.6860388834015\\
37	1.70205120225647\\
38	2.16961603330371\\
39	1.91801316526057\\
40	1.6200694080205\\
41	1.63202039019773\\
42	1.39885283087084\\
43	1.4328398918361\\
44	2.11533190747504\\
45	1.79001951396078\\
46	1.53593885390532\\
47	1.4551667881364\\
48	1.52586749016558\\
49	1.34242790778972\\
50	1.3361765677909\\
};
\addplot [color=mycolor3,solid,mark=asterisk,mark options={solid},mark size=0.5mm,forget plot]
  table[row sep=crcr]{%
51	83.3010537846591\\
52	32.2786722973492\\
53	18.5183165608232\\
54	13.5672678217792\\
55	11.4454121020495\\
56	10.2264782061219\\
57	10.2915917818669\\
58	9.54108486786303\\
59	9.63682266082318\\
60	8.96353209288536\\
61	9.12294477321144\\
62	9.05875698506644\\
63	8.87916977061704\\
};
\addplot [color=mycolor3,solid,mark=asterisk,mark options={solid},mark size=0.5mm,forget plot]
  table[row sep=crcr]{%
81	168.143154079044\\
82	93.9181939620331\\
83	59.3394337616316\\
};
\addplot [color=mycolor1,solid,mark=x,mark options={solid},mark size=0.5mm,forget plot]
  table[row sep=crcr]{%
1	4.0270922149947e-06\\
2	7.20218845719601e-06\\
3	1.21213240094244e-05\\
4	2.03521389631773e-05\\
5	3.41084564600954e-05\\
6	5.70240857899927e-05\\
7	9.52257081089007e-05\\
8	0.000158114200744504\\
9	0.000261275340856521\\
10	0.000429196768576031\\
11	0.00069996179956985\\
12	0.00113349587485701\\
13	0.00181427170873687\\
14	0.00286898255076285\\
15	0.00448121488050755\\
16	0.00690181731144229\\
17	0.0104845931295637\\
18	0.0157234333129529\\
19	0.0232526719870589\\
20	0.033997502513106\\
21	0.0490781023765818\\
22	0.0702026375434463\\
23	0.0997028742401245\\
24	0.139833449934267\\
25	0.194924953302421\\
26	0.27157995568542\\
27	0.376297334672662\\
28	0.517964399724777\\
29	0.706988962323782\\
30	0.966010397379396\\
31	1.30409510325439\\
32	1.746972476753\\
33	2.36604172463565\\
34	3.21978227793878\\
35	4.3253567612203\\
36	5.81081972929069\\
37	5.06047033461272\\
38	4.41137543458791\\
39	3.85301128691073\\
40	3.90505508906844\\
41	3.48978655042709\\
42	3.00523404492587\\
43	2.91796364977769\\
44	2.87729643625445\\
45	2.66630551412786\\
46	2.75486470443655\\
47	2.82312982829067\\
48	2.40219534977572\\
49	2.24796938818601\\
50	2.16174047267621\\
};
\addplot [color=mycolor1,solid,mark=x,mark options={solid},mark size=0.5mm,forget plot]
  table[row sep=crcr]{%
51	85.093059207527\\
52	32.8990733600462\\
53	19.3953212176384\\
54	15.0039942571808\\
55	12.6790903601064\\
56	11.7798614309365\\
57	11.4120287696342\\
58	10.7636223379753\\
59	10.9837509985643\\
60	10.4858575958678\\
61	10.664438196424\\
62	10.5690628517679\\
63	9.93370113687859\\
64	9.81600799588841\\
65	9.84341067583825\\
66	10.0661346904672\\
67	9.43277707920541\\
68	9.45835168385342\\
69	9.23679625933943\\
70	9.67035465782869\\
71	9.46700641756793\\
72	9.4069410054024\\
73	9.04717165034337\\
74	9.21867030752701\\
75	9.27035740306092\\
76	8.57434442395155\\
77	8.97973559955069\\
78	9.16058957044499\\
79	9.03416430006331\\
80	8.63726750684166\\
};
\addplot [color=mycolor1,solid,mark=x,mark options={solid},mark size=0.5mm,forget plot]
  table[row sep=crcr]{%
81	166.108368415754\\
82	90.9438762811845\\
83	58.1172039187506\\
};
\addplot [color=black,dotted,forget plot]
  table[row sep=crcr]{%
51	1e-06\\
51	812.965514959712\\
};
\addplot [color=black,dotted,forget plot]
  table[row sep=crcr]{%
81	1e-06\\
81	812.965514959712\\
};
\node[right, align=left, text=black, font=\tiny]
at (axis cs:1,3e2) {$50\times50\times33$};
\node[right, align=left, text=black, font=\tiny]
at (axis cs:51,5e-5) {$100\times100\times66$};
\node[right, align=left, text=black, font=\tiny]
at (axis cs:81,5e-5) {$200\times200\times132$};
\end{axis}
\end{tikzpicture}%
}\\

      \rotatebox[origin=c]{90}{dual residual} &
      \raisebox{-.5\height}{
%
%
\definecolor{mycolor1}{rgb}{0.85000,0.32500,0.09800}%
\definecolor{mycolor2}{rgb}{0.92900,0.69400,0.12500}%
\definecolor{mycolor3}{rgb}{0.49400,0.18400,0.55600}%
\begin{tikzpicture}

\begin{axis}[%
width=0.951\iwidth,
height=\iheight,
at={(0\iwidth,0\iheight)},
scale only axis,
xmin=1,
xmax=102,
xtick={1,20,40,51,70,81,90},
xticklabels={{1},{20},{40},{1},{20},{1},{10}},
ymode=log,
ymin=20,
ymax=1e8,
yminorticks=true,
axis background/.style={fill=white}
]
\addplot [color=mycolor2,solid,mark=*,mark options={solid},mark size=0.4mm,forget plot]
  table[row sep=crcr]{%
1	29951859.6172645\\
2	32013954.7210074\\
3	31900662.9104519\\
4	31716631.2125805\\
5	31480144.4646572\\
6	31178211.278359\\
7	30826633.285988\\
8	30342648.9406074\\
9	29738724.5510657\\
10	28993997.6817669\\
11	28087052.825227\\
12	27019854.9103724\\
13	25743521.7819206\\
14	24273831.2249898\\
15	22638385.1812062\\
16	20844866.3518203\\
17	18949201.91437\\
18	17018102.6862985\\
19	15084934.8457973\\
20	13221290.0774088\\
21	11456744.130306\\
22	9832968.2745432\\
23	8366215.14819595\\
24	7050936.45114824\\
25	5889072.05685458\\
26	4904180.9936516\\
27	4064269.08976907\\
28	3354122.08576543\\
29	2751690.4641207\\
30	2276157.9032665\\
31	1872562.11578192\\
32	1540417.5479263\\
33	1274561.25543853\\
34	1055459.1006495\\
35	873576.401468813\\
36	728752.516619523\\
37	605689.180570703\\
38	501669.527369197\\
39	421215.660213934\\
40	353844.703298765\\
41	297482.853418856\\
42	241206.422771023\\
43	200158.780141644\\
44	159884.549916928\\
45	125573.761483189\\
46	97530.2654826952\\
47	75215.9056192645\\
48	55612.6741935472\\
49	42377.0220292663\\
50	30279.6647708508\\
};
\addplot [color=mycolor2,solid,mark=*,mark options={solid},mark size=0.4mm,forget plot]
  table[row sep=crcr]{%
51	8680.65805741571\\
52	6393.69715487515\\
53	3896.74460783694\\
54	2858.10780882427\\
};
\addplot [color=mycolor2,solid,mark=*,mark options={solid},mark size=0.4mm,forget plot]
  table[row sep=crcr]{%
81	1569.37759495956\\
82	846.631149731424\\
83	416.827192983326\\
84	280.01835224228\\
85	212.440448346595\\
86	168.565530757699\\
87	143.584420211527\\
};
\addplot [color=mycolor3,solid,mark=asterisk,mark options={solid},mark size=0.5mm,forget plot]
  table[row sep=crcr]{%
1	3515631.80274783\\
2	1653031.58884742\\
3	1093412.44039116\\
4	834978.145030282\\
5	700843.353793853\\
6	607952.978665085\\
7	543245.887697114\\
8	497910.521445297\\
9	462030.877499428\\
10	427100.993561871\\
11	390631.82393411\\
12	367628.21075521\\
13	345816.588027836\\
14	337955.398586895\\
15	317648.228841788\\
16	300518.631922097\\
17	290391.903367302\\
18	281685.012646863\\
19	270211.643356346\\
20	267161.302846183\\
21	255964.118197237\\
22	241815.792682023\\
23	253216.175705168\\
24	233321.722058357\\
25	227536.511019245\\
26	223978.398165352\\
27	214512.844671105\\
28	208077.500275439\\
29	203600.551399777\\
30	193093.593357314\\
31	209409.001569482\\
32	197386.171590776\\
33	182277.159961184\\
34	181565.331746749\\
35	178298.3399345\\
36	172762.377483189\\
37	169240.632910693\\
38	185098.110110879\\
39	172809.33441115\\
40	155770.296161421\\
41	159048.760237224\\
42	148198.982998602\\
43	147125.562280712\\
44	172002.422322447\\
45	156445.074469559\\
46	147000.855166927\\
47	139767.411638537\\
48	142848.003871834\\
49	130767.529326004\\
50	128961.270699991\\
};
\addplot [color=mycolor3,solid,mark=asterisk,mark options={solid},mark size=0.5mm,forget plot]
  table[row sep=crcr]{%
51	194635.252663274\\
52	93582.4547645147\\
53	56172.7678520124\\
54	42726.1339623889\\
55	35990.4424240302\\
56	32231.8932358807\\
57	30616.2916532503\\
58	27994.0523671273\\
59	27527.2408687458\\
60	25733.3973973279\\
61	25471.752450302\\
62	24981.4582801897\\
63	24828.0551634691\\
};
\addplot [color=mycolor3,solid,mark=asterisk,mark options={solid},mark size=0.5mm,forget plot]
  table[row sep=crcr]{%
81	12095.6271772525\\
82	10441.9706174032\\
83	5906.5109657824\\
};
\addplot [color=mycolor1,solid,mark=x,mark options={solid},mark size=0.5mm,forget plot]
  table[row sep=crcr]{%
1	29951859.6172645\\
2	32013954.7210074\\
3	31900662.9104519\\
4	31716631.2125805\\
5	31480144.4646572\\
6	31178211.278359\\
7	30826633.285988\\
8	30342648.9406074\\
9	29738724.5510657\\
10	28993997.6817669\\
11	28087052.825227\\
12	27019854.9103724\\
13	25743521.7819206\\
14	24273831.2249898\\
15	22638385.1812062\\
16	20844866.3518203\\
17	18949201.91437\\
18	17018102.6862985\\
19	15084934.8457973\\
20	13221290.0774088\\
21	11456744.130306\\
22	9832968.2745432\\
23	8366215.14819595\\
24	7050936.45114824\\
25	5889072.05685458\\
26	4904180.9936516\\
27	4064269.08976907\\
28	3354122.08576543\\
29	2751690.4641207\\
30	2276157.9032665\\
31	1872562.11578192\\
32	1540417.5479263\\
33	1274561.25543853\\
34	1055459.1006495\\
35	873576.401468813\\
36	728752.516619523\\
37	638347.805210521\\
38	565755.628550367\\
39	508918.54395066\\
40	483237.854192351\\
41	443084.268084066\\
42	408529.973764044\\
43	382868.24810077\\
44	368518.054500784\\
45	344530.797126686\\
46	333271.379328012\\
47	319896.038919755\\
48	300493.067511109\\
49	290137.908317175\\
50	280803.477875814\\
};
\addplot [color=mycolor1,solid,mark=x,mark options={solid},mark size=0.5mm,forget plot]
  table[row sep=crcr]{%
51	215137.030416358\\
52	107161.796719342\\
53	68058.85310532\\
54	54962.0025481419\\
55	47595.4916271661\\
56	43529.4528033345\\
57	40283.2593991181\\
58	38894.5914133675\\
59	38082.9439459713\\
60	36027.4097115236\\
61	35006.5776123772\\
62	34603.3721409902\\
63	33145.5565998212\\
64	32193.4579808827\\
65	31481.0966300934\\
66	31469.6443358118\\
67	30077.5010993653\\
68	29466.6059953503\\
69	28731.0423036411\\
70	28848.2988313361\\
71	28195.5643709243\\
72	27719.8231330948\\
73	26968.0610456424\\
74	27103.3339745365\\
75	27213.0499159113\\
76	25804.1279213415\\
77	25837.8182096898\\
78	25460.3392208084\\
79	25571.9855755905\\
80	24594.4088028831\\
};
\addplot [color=mycolor1,solid,mark=x,mark options={solid},mark size=0.5mm,forget plot]
  table[row sep=crcr]{%
81	12605.9534962877\\
82	10598.609935618\\
83	6052.64903450105\\
};
\addplot [color=black,dotted,forget plot]
  table[row sep=crcr]{%
51	20\\
51	1e8\\
};
\addplot [color=black,dotted,forget plot]
  table[row sep=crcr]{%
81	20\\
81	1e8\\
};
\node[right, align=left, text=black, font=\tiny]
at (axis cs:1,4e2) {$50\times50\times33$};
\node[right, align=left, text=black, font=\tiny]
at (axis cs:51,4e6) {$100\times100\times66$};
\node[right, align=left, text=black, font=\tiny]
at (axis cs:81,4e6) {$200\times200\times132$};
\end{axis}
\end{tikzpicture}%
}\\

      \rotatebox[origin=c]{90}{parameter $\rho$} &
      \raisebox{-.5\height}{
%
%
\definecolor{mycolor1}{rgb}{0.85000,0.32500,0.09800}%
\definecolor{mycolor2}{rgb}{0.92900,0.69400,0.12500}%
\definecolor{mycolor3}{rgb}{0.49400,0.18400,0.55600}%
\begin{tikzpicture}

\begin{axis}[%
width=0.951\iwidth,
height=\iheight,
at={(0\iwidth,0\iheight)},
scale only axis,
xmin=1,
xmax=102,
xtick={1,20,40,51,70,81,90},
xticklabels={{1},{20},{40},{1},{20},{1},{10}},
ymode=log,
ymin=1,
ymax=5e6,
yminorticks=true,
axis background/.style={fill=white}
]
\addplot [color=mycolor2,solid,mark=*,mark options={solid},mark size=0.4mm,forget plot]
  table[row sep=crcr]{%
1	1000000\\
2	769230.769230769\\
3	591715.976331361\\
4	455166.135639508\\
5	350127.796645776\\
6	269329.074342904\\
7	207176.211033003\\
8	159366.316179233\\
9	122589.473984026\\
10	94299.5953723274\\
11	72538.1502864057\\
12	55798.577143389\\
13	42921.9824179915\\
14	33016.9095523012\\
15	25397.6227325394\\
16	19536.6328711841\\
17	15028.1791316801\\
18	11560.1377936001\\
19	8892.41368738467\\
20	6840.31822106513\\
21	5261.78324697318\\
22	4047.52557459475\\
23	3113.48121122673\\
24	2394.98554709749\\
25	1842.29657469037\\
26	1417.15121130029\\
27	1090.11631638484\\
28	838.55101260372\\
29	645.0392404644\\
30	496.184031126462\\
31	381.680023943432\\
32	293.600018418025\\
33	225.846168013865\\
34	173.727821549127\\
35	133.636785807021\\
36	102.797527543862\\
37	79.0750211875863\\
38	60.8269393750664\\
39	46.7899533654357\\
40	35.9922718195659\\
41	27.6863629381276\\
42	21.2972022600982\\
43	16.3824632769986\\
44	12.6018948284604\\
45	9.69376525266188\\
46	7.4567425020476\\
47	5.73595577080585\\
48	4.41227366985065\\
49	3.39405666911589\\
50	2.61081282239683\\
};
\addplot [color=mycolor2,solid,mark=*,mark options={solid},mark size=0.4mm,forget plot]
  table[row sep=crcr]{%
51	2.61081282239683\\
52	2.61081282239683\\
53	2.61081282239683\\
54	2.61081282239683\\
};
\addplot [color=mycolor2,solid,mark=*,mark options={solid},mark size=0.4mm,forget plot]
  table[row sep=crcr]{%
81	2.61081282239683\\
82	2.61081282239683\\
83	2.61081282239683\\
84	2.61081282239683\\
85	2.61081282239683\\
86	2.61081282239683\\
87	2.61081282239683\\
};
\addplot [color=mycolor3,solid,mark=asterisk,mark options={solid},mark size=0.5mm,forget plot]
  table[row sep=crcr]{%
1	100\\
2	100\\
3	100\\
4	100\\
5	100\\
6	100\\
7	100\\
8	100\\
9	100\\
10	100\\
11	100\\
12	100\\
13	100\\
14	100\\
15	100\\
16	100\\
17	100\\
18	100\\
19	100\\
20	100\\
21	100\\
22	100\\
23	100\\
24	100\\
25	100\\
26	100\\
27	100\\
28	100\\
29	100\\
30	100\\
31	100\\
32	100\\
33	100\\
34	100\\
35	100\\
36	100\\
37	100\\
38	100\\
39	100\\
40	100\\
41	100\\
42	100\\
43	100\\
44	100\\
45	100\\
46	100\\
47	100\\
48	100\\
49	100\\
50	100\\
};
\addplot [color=mycolor3,solid,mark=asterisk,mark options={solid},mark size=0.5mm,forget plot]
  table[row sep=crcr]{%
51	100\\
52	100\\
53	100\\
54	100\\
55	100\\
56	100\\
57	100\\
58	100\\
59	100\\
60	100\\
61	100\\
62	100\\
63	100\\
};
\addplot [color=mycolor3,solid,mark=asterisk,mark options={solid},mark size=0.5mm,forget plot]
  table[row sep=crcr]{%
81	100\\
82	100\\
83	100\\
};
\addplot [color=mycolor1,solid,mark=x,mark options={solid},mark size=0.5mm,forget plot]
  table[row sep=crcr]{%
1	1000000\\
2	769230.769230769\\
3	591715.976331361\\
4	455166.135639508\\
5	350127.796645776\\
6	269329.074342904\\
7	207176.211033003\\
8	159366.316179233\\
9	122589.473984026\\
10	94299.5953723274\\
11	72538.1502864057\\
12	55798.577143389\\
13	42921.9824179915\\
14	33016.9095523012\\
15	25397.6227325394\\
16	19536.6328711841\\
17	15028.1791316801\\
18	11560.1377936001\\
19	8892.41368738467\\
20	6840.31822106513\\
21	5261.78324697318\\
22	4047.52557459475\\
23	3113.48121122673\\
24	2394.98554709749\\
25	1842.29657469037\\
26	1417.15121130029\\
27	1090.11631638484\\
28	838.55101260372\\
29	645.0392404644\\
30	496.184031126462\\
31	381.680023943432\\
32	293.600018418025\\
33	225.846168013865\\
34	173.727821549127\\
35	133.636785807021\\
36	102.797527543862\\
37	102.797527543862\\
38	102.797527543862\\
39	102.797527543862\\
40	102.797527543862\\
41	102.797527543862\\
42	102.797527543862\\
43	102.797527543862\\
44	102.797527543862\\
45	102.797527543862\\
46	102.797527543862\\
47	102.797527543862\\
48	102.797527543862\\
49	102.797527543862\\
50	102.797527543862\\
};
\addplot [color=mycolor1,solid,mark=x,mark options={solid},mark size=0.5mm,forget plot]
  table[row sep=crcr]{%
51	102.797527543862\\
52	102.797527543862\\
53	102.797527543862\\
54	102.797527543862\\
55	102.797527543862\\
56	102.797527543862\\
57	102.797527543862\\
58	102.797527543862\\
59	102.797527543862\\
60	102.797527543862\\
61	102.797527543862\\
62	102.797527543862\\
63	102.797527543862\\
64	102.797527543862\\
65	102.797527543862\\
66	102.797527543862\\
67	102.797527543862\\
68	102.797527543862\\
69	102.797527543862\\
70	102.797527543862\\
71	102.797527543862\\
72	102.797527543862\\
73	102.797527543862\\
74	102.797527543862\\
75	102.797527543862\\
76	102.797527543862\\
77	102.797527543862\\
78	102.797527543862\\
79	102.797527543862\\
80	102.797527543862\\
};
\addplot [color=mycolor1,solid,mark=x,mark options={solid},mark size=0.5mm,forget plot]
  table[row sep=crcr]{%
81	102.797527543862\\
82	102.797527543862\\
83	102.797527543862\\
};
\addplot [color=black,dotted,forget plot]
  table[row sep=crcr]{%
51	1\\
51	5e6\\
};
\addplot [color=black,dotted,forget plot]
  table[row sep=crcr]{%
81	1\\
81	5e6\\
};
\node[right, align=left, text=black, font=\tiny]
at (axis cs:1,10) {$50\times50\times33$};
\node[right, align=left, text=black, font=\tiny]
at (axis cs:51,2e5) {$100\times100\times66$};
\node[right, align=left, text=black, font=\tiny]
at (axis cs:81,2e5) {$200\times200\times132$};
\end{axis}
\end{tikzpicture}%
}\\
      & \multicolumn{1}{c}{ADMM iteration}
  \end{tabular}
  \caption{Convergence of ADMM for a multilevel 3D EPI susceptibility artifact correction. Two adaptive and one fixed strategies for choosing the augmentation parameter $\rho$ in~\cref{eq:Lp} are compared. 
First row shows the norm of the primal residual (cf.~\cref{eq:admmStop}) for each iteration of ADMM and each level in the
coarse to fine hierarchy. Second row shows the norm of the dual residual (cf.~\cref{eq:admmStop})
and the bottom row shows the augmentation parameter for each iteration.
The regularization parameter for the smoothing term is $\alpha=50$.}
  \label{fig:3Dbrain_admm}
\end{figure}
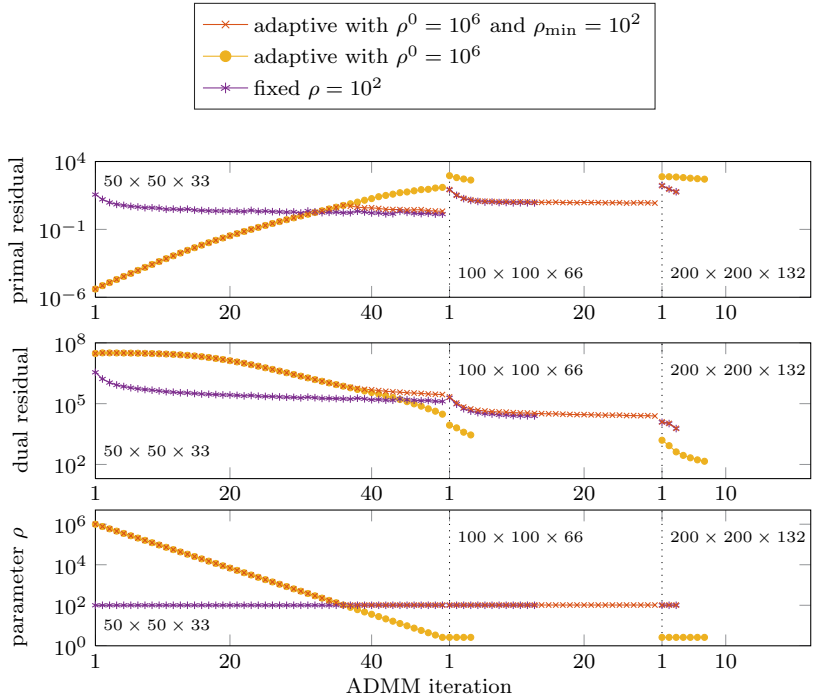

\paragraph{Comparison} 
\label{par:comparison}
We compare the quality of the proposed reconstruction methods to an
established state-of-the-art method for susceptibility artifact correction.
Exemplarily, we consider HySCO~\cite{RuthottoEtAl2013hysco} since it is based on the 
same variational formulation and uses the same three-level 
approach as our proposed method. 
In contrast to the proposed methods, HySCO uses a nodal discretization, which leads to coupling across slices and image columns (see second row in~\cref{fig:compare_sparsity}),  and provides only a simple Jacobi preconditioner.
We use the same 3D data as in the previous section, visualize the
correction results for all three different methods in
\cref{fig:compare}, and summarize quantitative results in \cref{tab:compare}. We provide results for two experiments using HySCO. First, we report results using the default settings in HySCO. By default, HySCO uses a smoothed cubic spline based approximation of the data described in~\cite{Modersitzki2009}, which aims at adding robustness. Since this comes at an additional computational costs, a second experiment using the same linear interpolation model used in the proposed methods is performed.
\begin{figure}[t]
	\renewcommand{\rottext}[1]{\rotatebox{90}{\hbox to 20mm{\hss #1\hss}}}
	\scriptsize
  	\centering
	\setlength{\iwidth}{35mm}
	\begin{tabular}{@{}|@{\,}c@{\,}|@{\,}c
	@{\,}c@{\,}c@{\,}c@{\,}|@{}}
		\hline
      	& initial data & HySCO (cubic)~\cite{RuthottoEtAl2013hysco} 
      	&  Gauss-Newton-PCG & ADMM \\ 
		\rottext{$\CI_{v}$} &
      	\includegraphics[width=\iwidth]{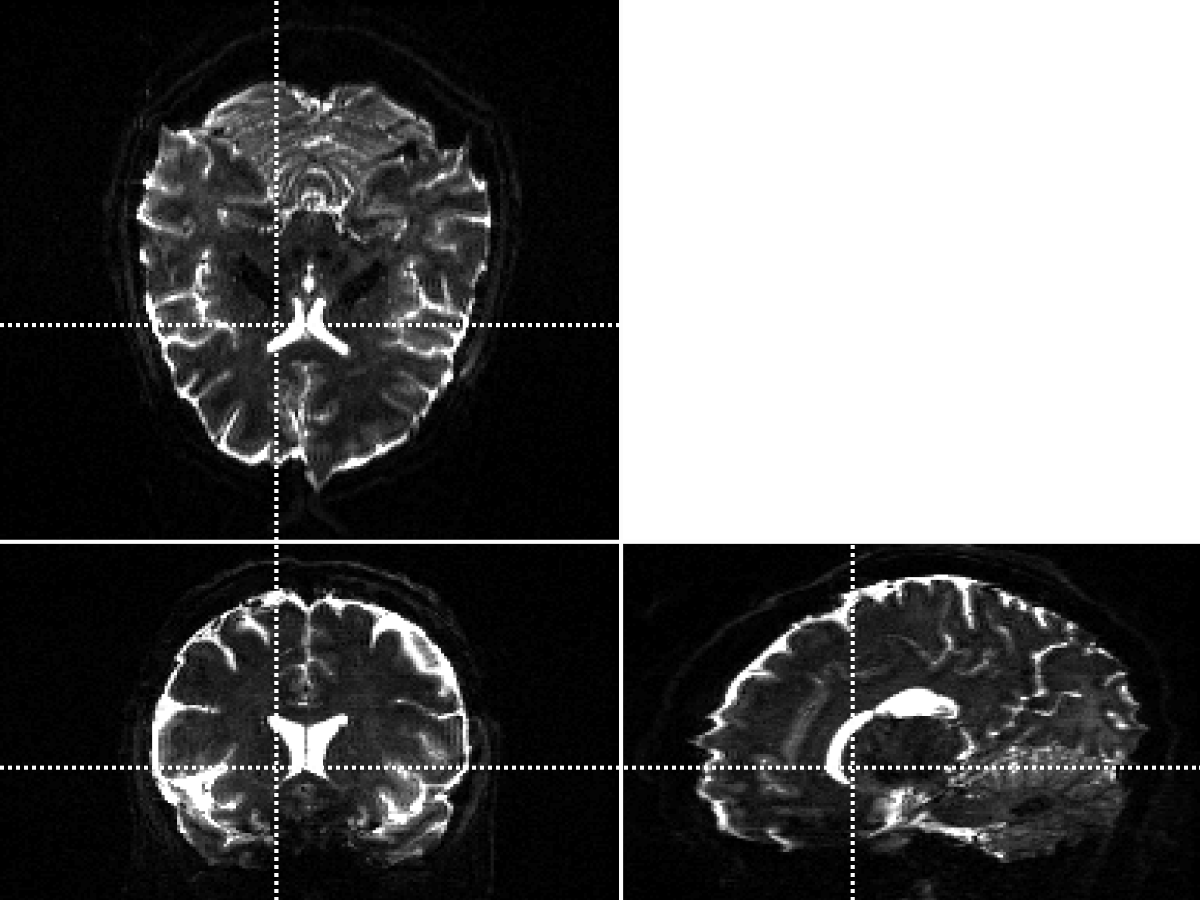} &
      	\includegraphics[width=\iwidth]{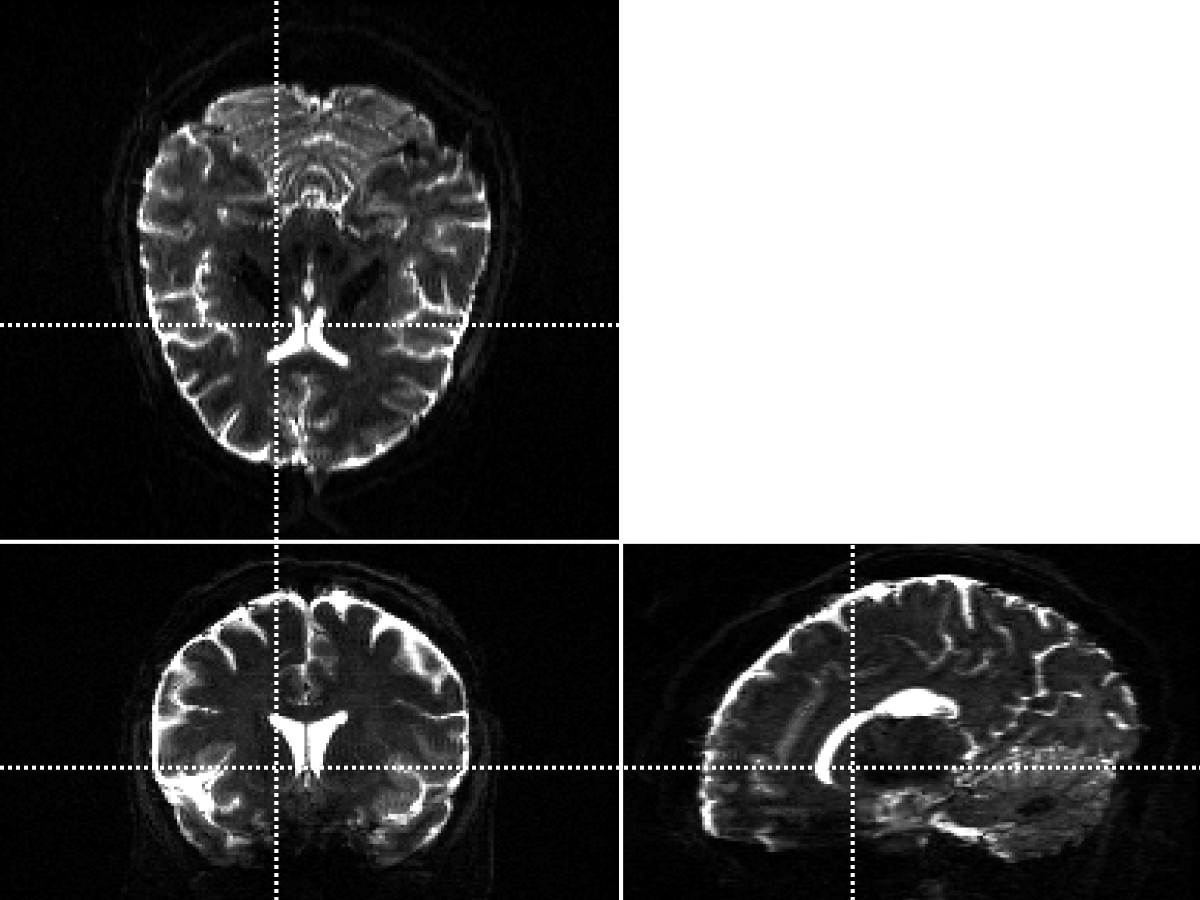} &
      	\includegraphics[width=\iwidth]{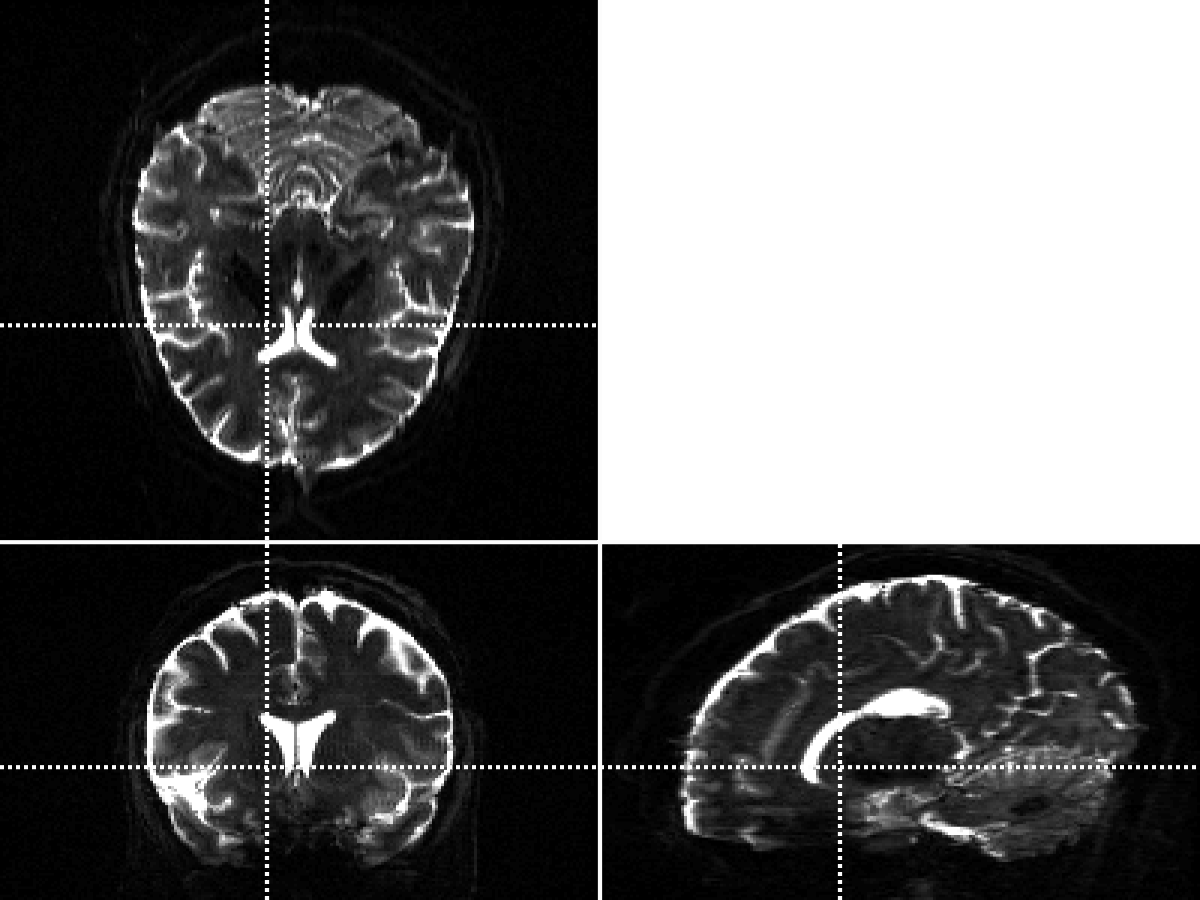} &
      	\includegraphics[width=\iwidth]{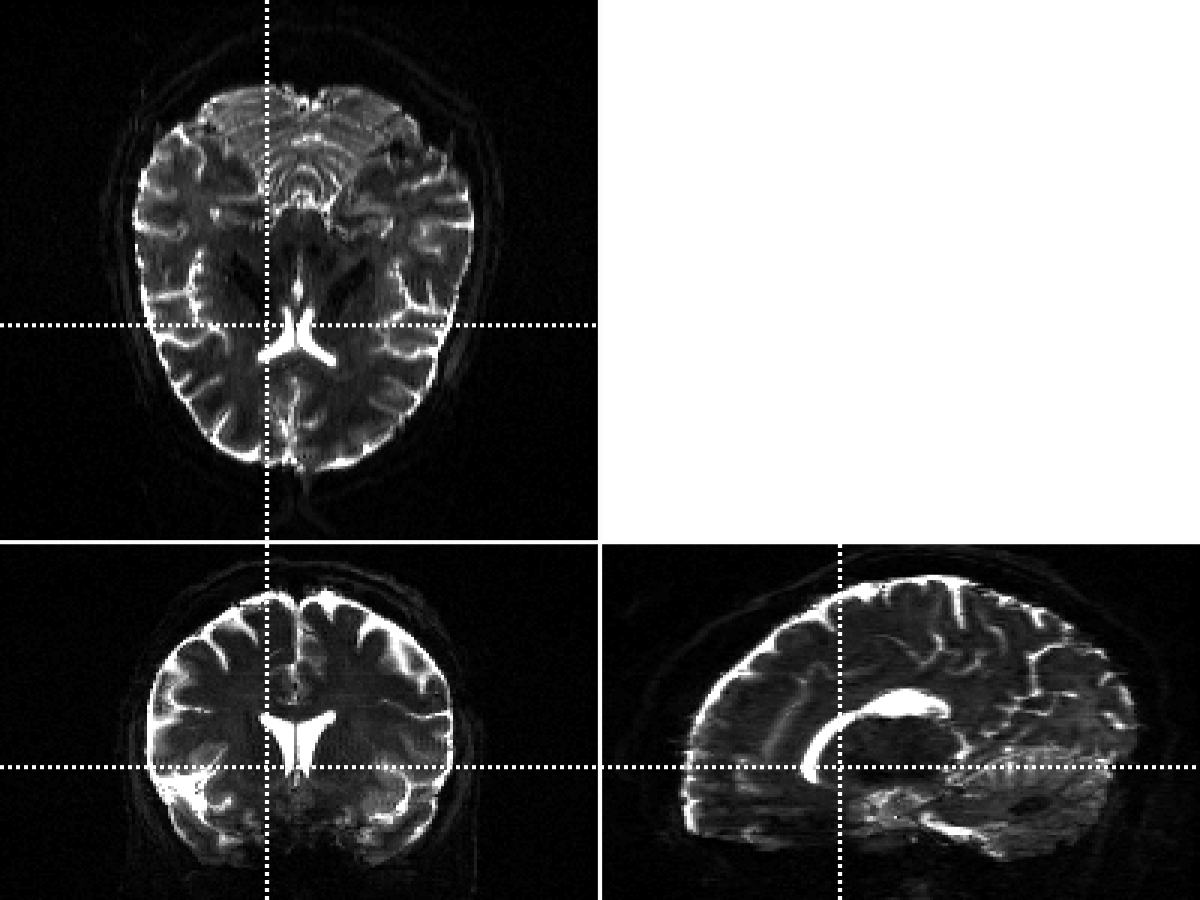} 
      	\\
		\rottext{$\CI_{-v}$} &
      	\includegraphics[width=\iwidth]{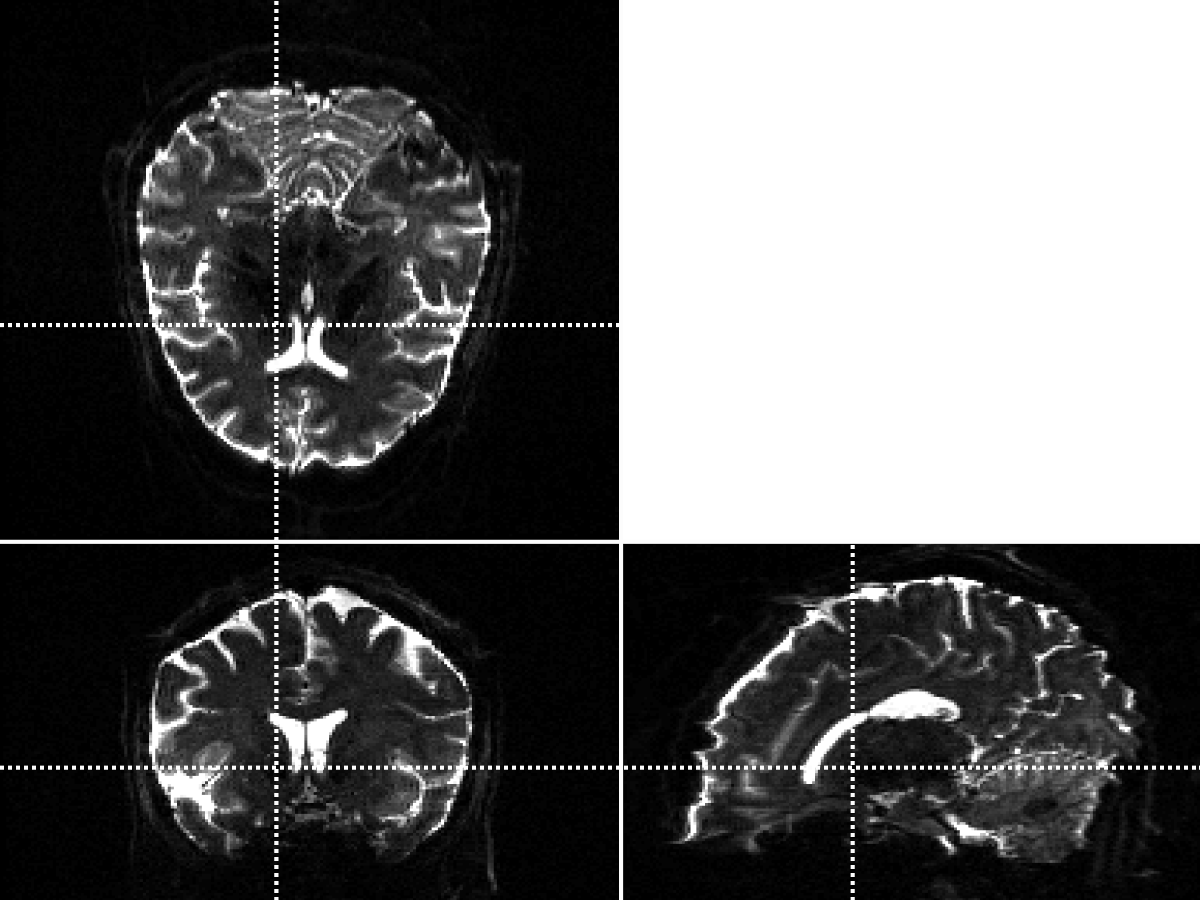} &
      	\includegraphics[width=\iwidth]{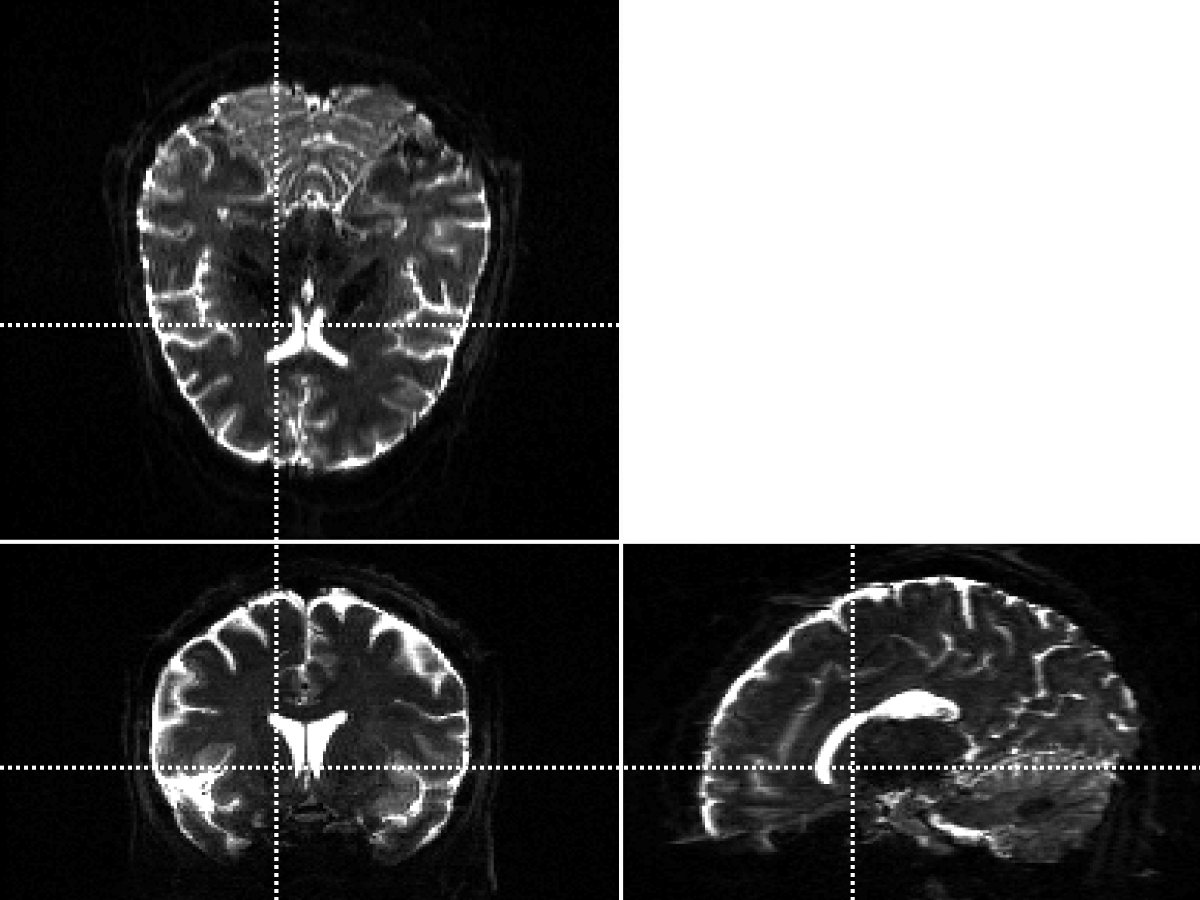} &
      	\includegraphics[width=\iwidth]{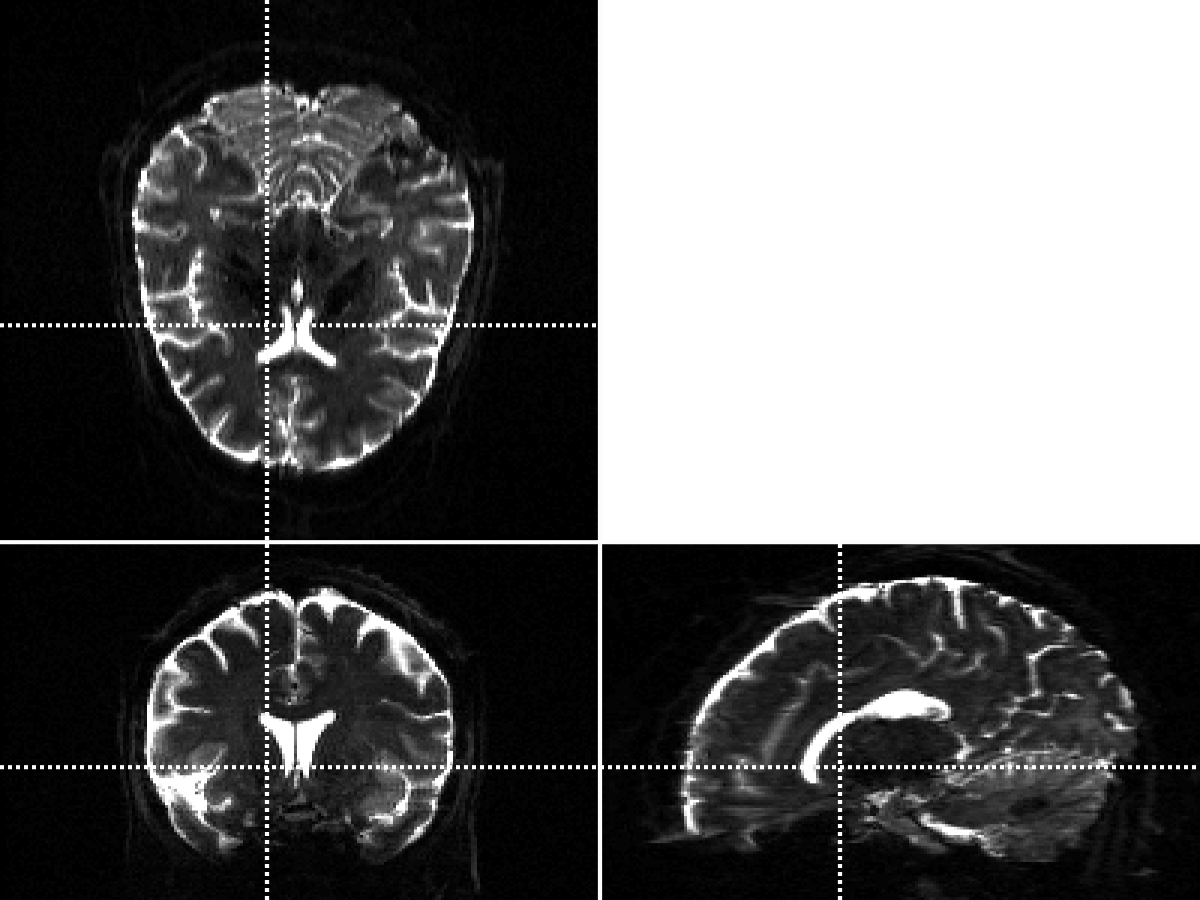} &
      	\includegraphics[width=\iwidth]{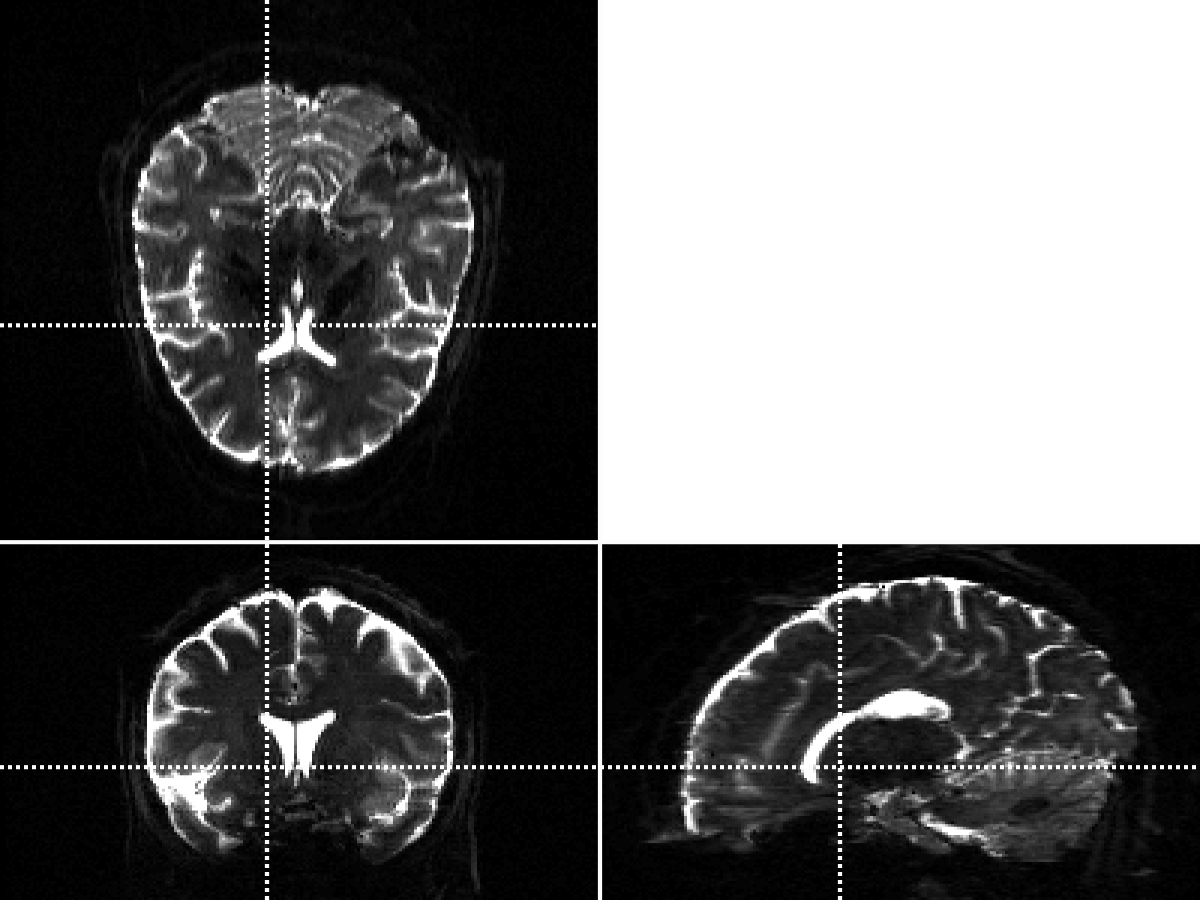} 
  	\\
	\rottext{$| \CI_{v} - \CI_{-v}|$} &
      	\includegraphics[width=\iwidth]{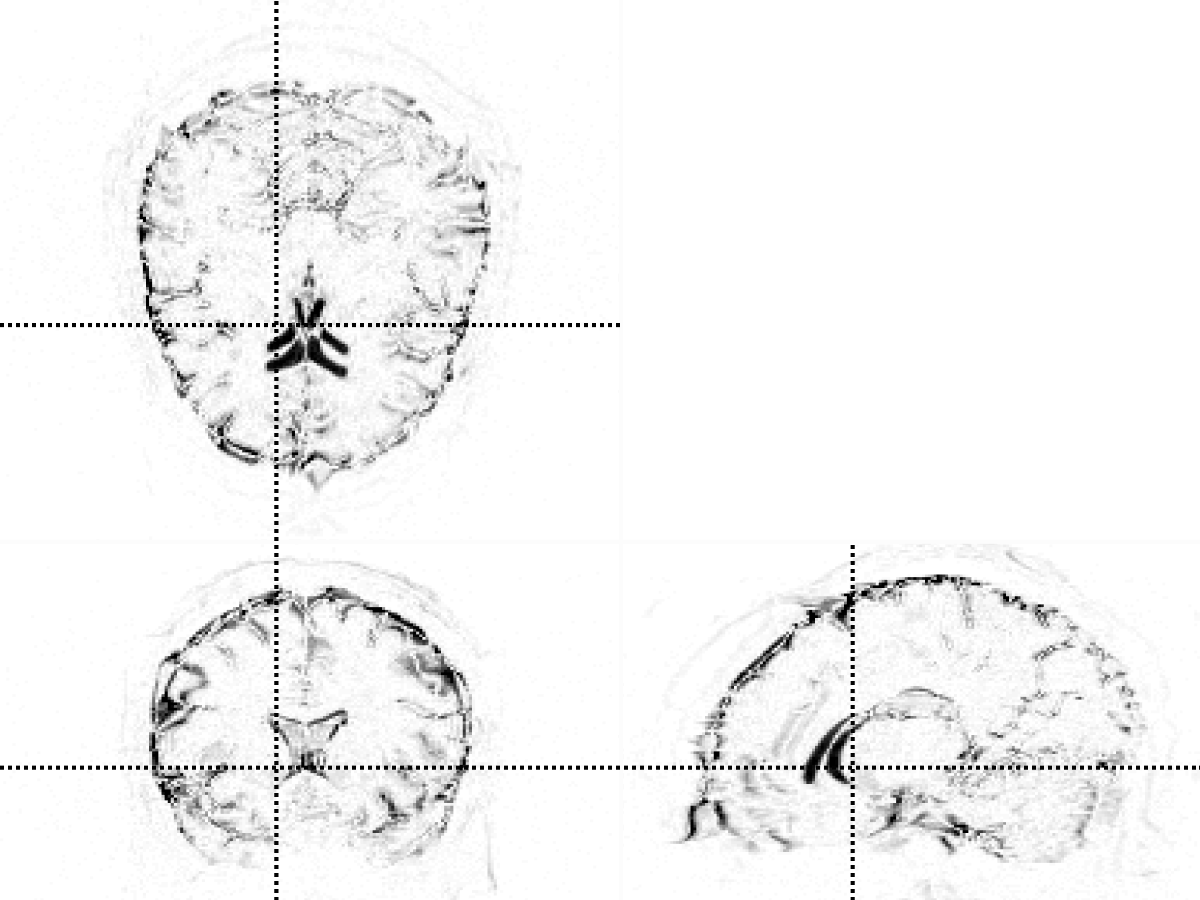} &
      	\includegraphics[width=\iwidth]{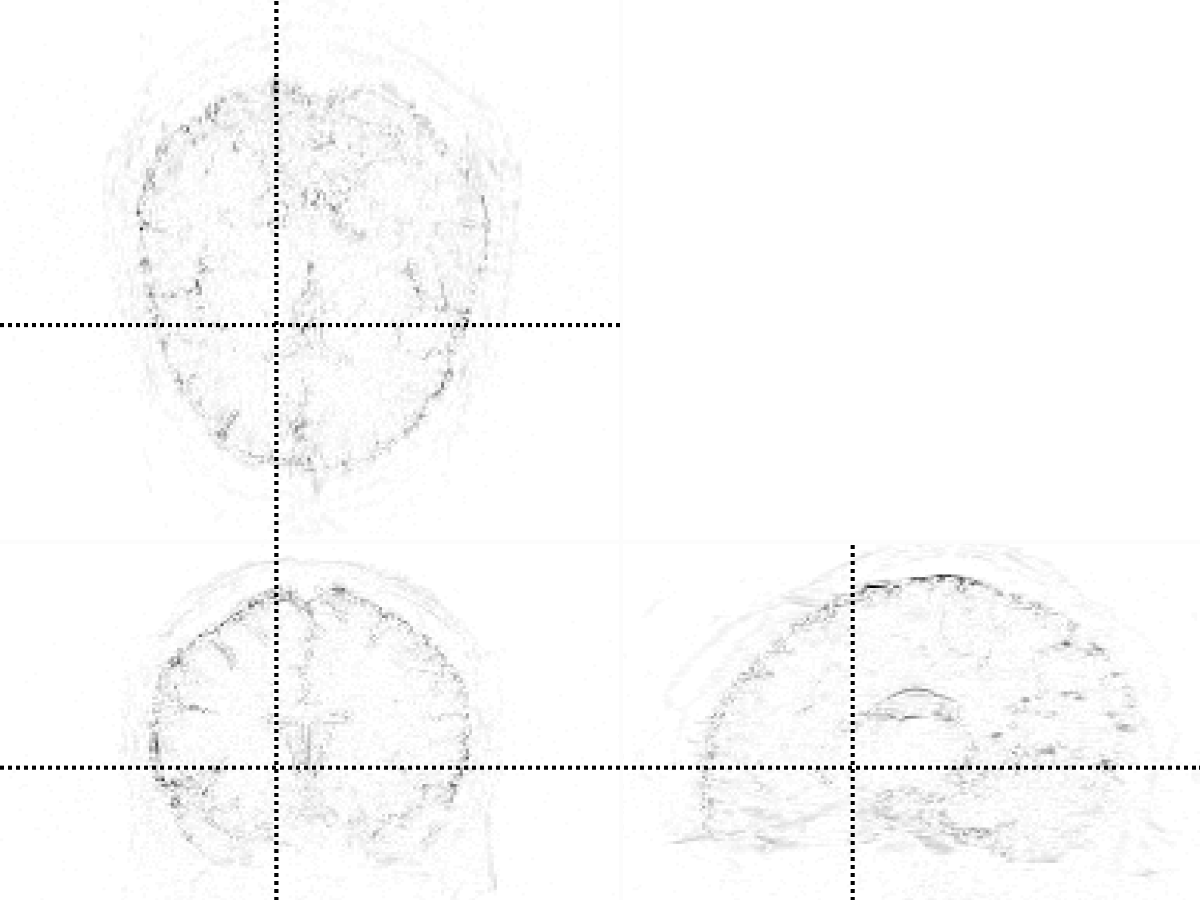} &
        \includegraphics[width=\iwidth]{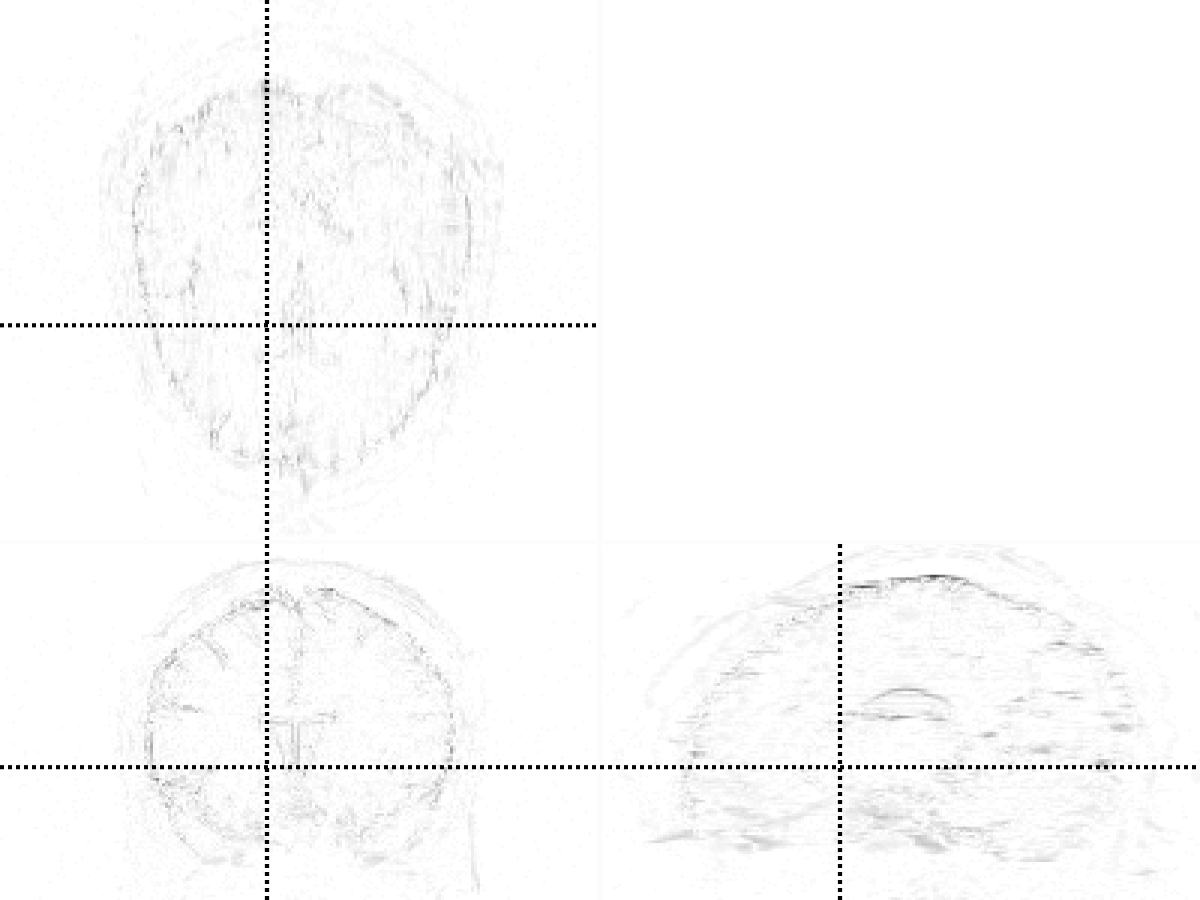} &
      	\includegraphics[width=\iwidth]{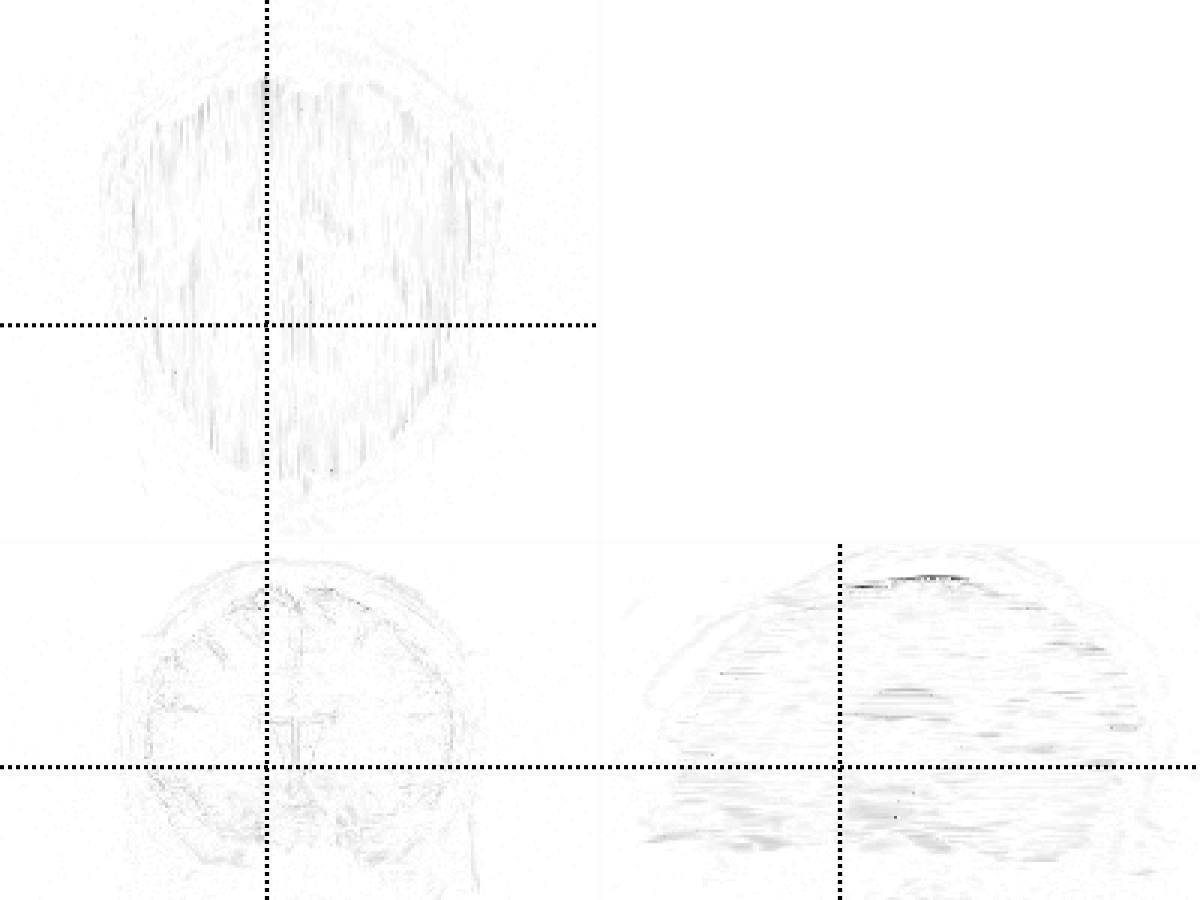} 
    \\
	\rottext{inhomogeneity $\bfb$}&
         &
      	\includegraphics[width=\iwidth]{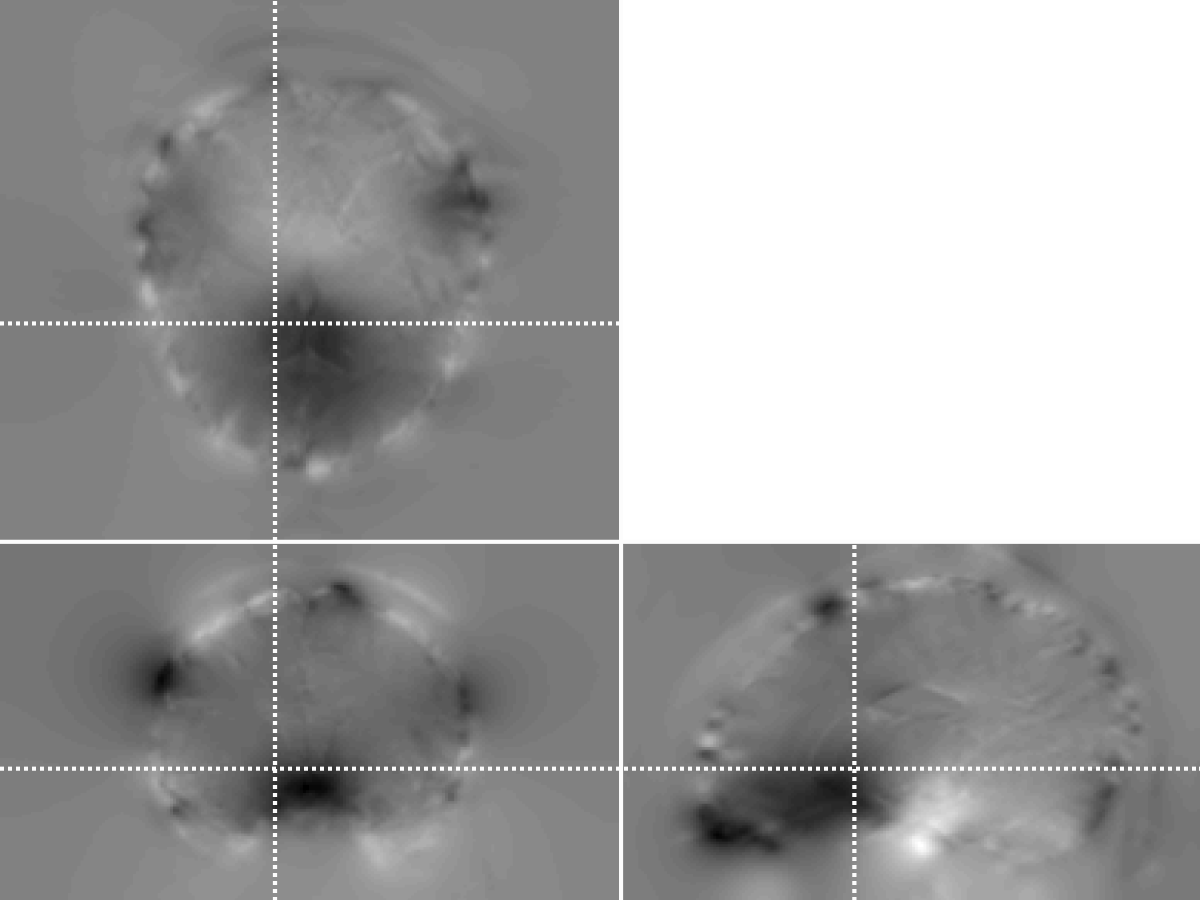} &
		\includegraphics[width=\iwidth]{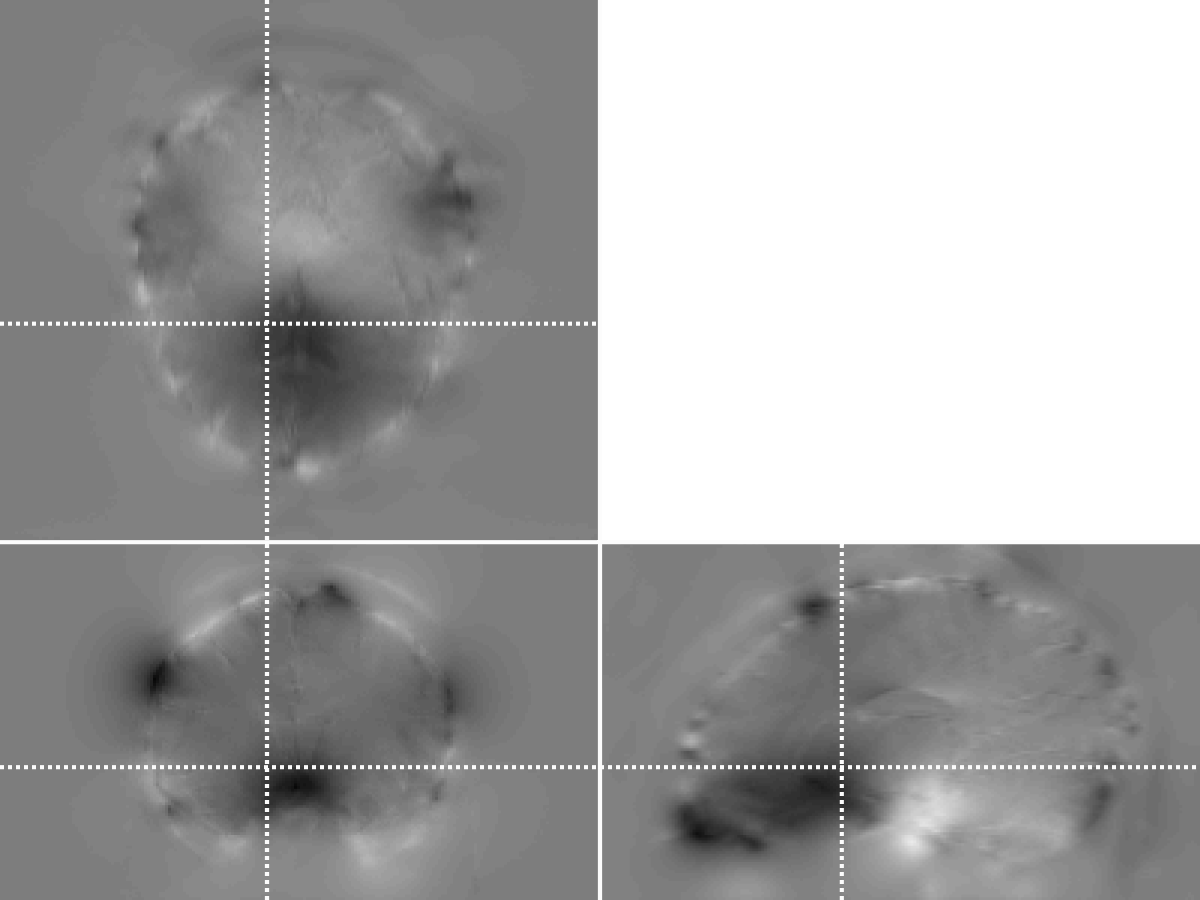} &
      	\includegraphics[width=\iwidth]{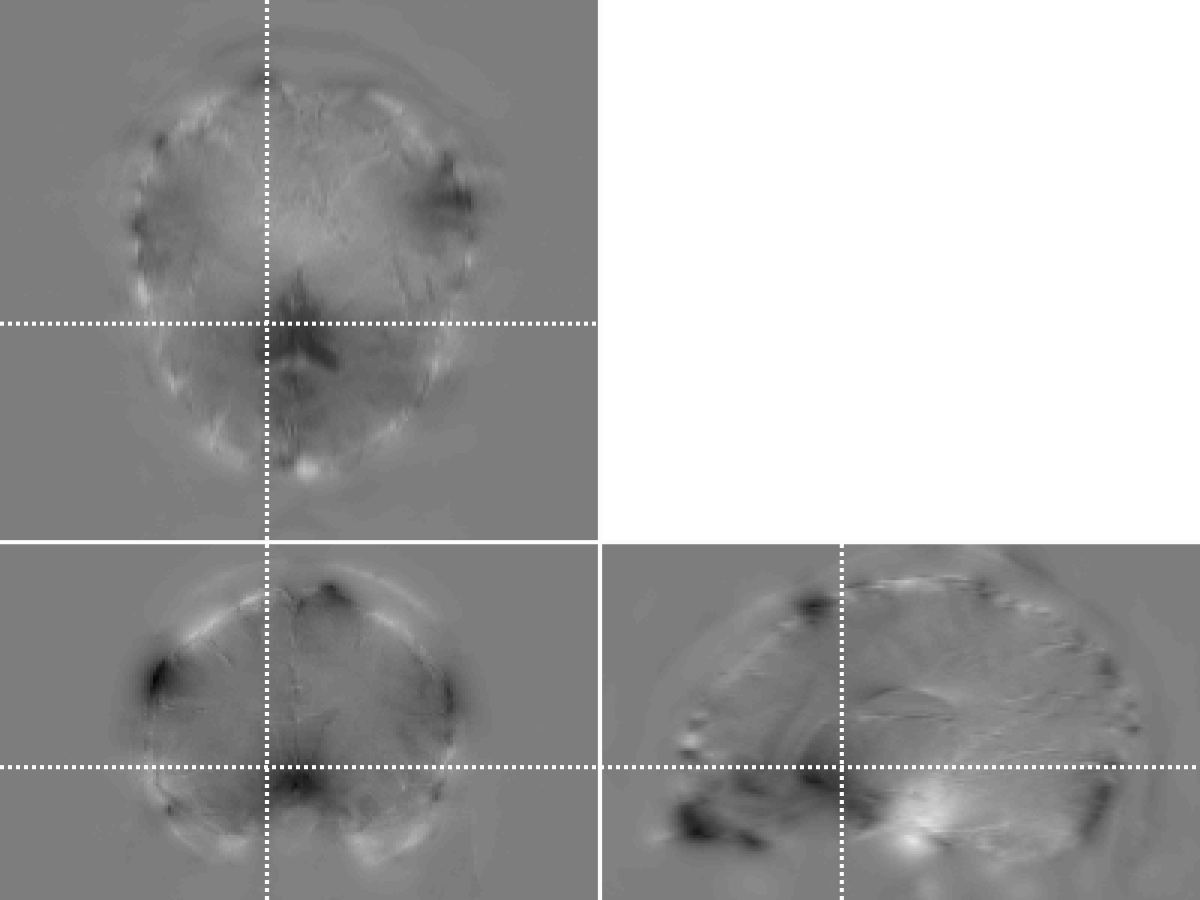} 
      	\\ \hline
      \end{tabular} 
  \caption{3D correction results for high-resolution EPI-MRI brain data provided by the Human Connectome Project~\cite{VanEssen2012}. Initial  data  (left)  and  results  of
three  reversed-gradient  based  correction  methods  are  visualized  using  orthogonal  slice  views. The color axis are chosen identically in each row. First  and  second
rows  show  initial  data  and  corrected  images  for  both  phase-encoding  directions.  Third  row  shows  the  absolute
difference between both images using an inverted color scale for improved visualization.  A significant reduction
of  the  image  difference  can  be  observed  for  all  correction  methods.  The  bottom  row  shows  the  estimated  field
inhomogeneity that is comparable among all correction techniques.}
  \label{fig:compare}
\end{figure}

\begin{table}[t]
 \caption{Comparison of the improvements in image similarity for different reversed-gradient methods applied to the 3D-MRI data with $\alpha=50$ and $\beta=10$ (in GN methods only). Image similarity between the initial and transformed blip-up and blip-down data is assessed using the sum of squared differences (SSD) distance~\cite{Modersitzki2009} (smaller value associated with more similarity) and normalized cross-correlation (NCC)~\cite{Modersitzki2009} (normalized to $[0,1]$ where $1$ is optimal).}
 \label{tab:compare}
 \centering
\scriptsize
 \begin{tabular}{|l|r|r|r|r|}   
  \hline
  & \multicolumn{1}{c|}{HySCO (cubic)~\cite{RuthottoEtAl2013hysco}} & \multicolumn{1}{c|}{HySCO (linear)~\cite{RuthottoEtAl2013hysco}} &
  \multicolumn{1}{c|}{Gauss-Newton-PCG} & \multicolumn{1}{c|}{ADMM}\\
  \hline\hline
  ${\rm SSD}\left(\bfb^0\right)$ & $7.395\cdot10^{10}$ & $7.395\cdot10^{10}$ & $7.395\cdot10^{10}$ & $7.395\cdot10^{10}$\rule{0pt}{2.6ex} 	\\
  ${\rm SSD}\left(\bfb^*\right)$ & $9.540\cdot10^{9}$ & $1.557\cdot10^{10}$ & $6.433\cdot10^{9}$	& $5.058\cdot10^{9}$ 				\\
  \hline
  reduction              & $87.1\%$            & $78.9\%$            & $91.3 \%$	 & $93.1 \%$          							\\
  \hline\hline
  ${\rm NCC}\left(\bfb^0\right)$ & $0.529$ & $0.529$ & $0.529$ & $0.529$\rule{0pt}{2.6ex} 	\\
  ${\rm NCC}\left(\bfb^*\right)$ & $0.924$ & $0.879$ & $0.955$	& $ 0.964$ 				\\
  \hline\hline
  runtime                & 2059 s            & 198 s             & 142 s              & 43 s 								\\
  \hline
\end{tabular}

\end{table}
All methods effectively correct for susceptibility artifacts rendering the corrected image pairs visually more similar with both proposed methods slightly outperforming HySCO; see residual images in
\cref{fig:compare} and \cref{tab:compare}. A notable difference is observed comparing the performance of HySCO with cubic B-spline and linear interpolation of the image data. While the former achieves a higher quality correction, the time-to-solution is around 34 minutes. With less than 3 minutes, the time-to-solution for the linear interpolation model is considerably lower, but it achieves considerably inferior correction results. Changing from the nodal discretization employed in HySCO to the
face-staggered discretization proposed here, leads to a highly accurate correction (see improvement in image similarity), even with a linear interpolation model. As to be expected, the results for GN-PCG and ADMM yield almost identical results.
The most striking difference between the different methods is the reduced time-to-solution. Both newly proposed methods outperform the existing approaches. In our comparison the fastest method is the parallelized implementation of ADMM, which yields a speedup factor of around 50x while the new GN-PCG yields a speedup factor of around 15x as compared to the default settings of HySCO.

\section{Summary And Conclusion} 
\label{sec:summary}
In this paper, we present two efficient methods for susceptibility
artifact correction of EPI-MRI. We consider a variational formulation of
a reversed gradient based correction scheme similar
to~\cite{RuthottoEtAlPMB2012,RuthottoEtAl2013hysco,HollandEtAl2010}. Our method requires one additional EPI-MRI acquisition with opposite phase-encoding direction and, thus, opposite distortion. We follow a discretize-then-optimize paradigm and propose a face-staggered discretization of the field inhomogeneity. This choice leads to a separable discrete distance function and constraints. While the overall optimization problem is, due to the smoothness regularizer, not separable we consider two optimization schemes that exploit the partially separable structure. 
\par
First, we propose a block-Jacobi preconditioner to be used in
Gauss-Newton-PCG optimization schemes. The preconditioner is designed to
exploit the respective block structure of the distance and penalty functions
and also accounts for parts of the smoothness regularizer. The
preconditioner is block-diagonal with tridiagonal structure and thus can
be computed in parallel and with linear complexity. Using a 2D example, we
demonstrate the improved clustering of eigenvalues as compared to the
cheaper Jacobi preconditioner; see \cref{fig:spectrum}. We also show
the effectiveness of the block-Jacobi preconditioner for solving the
Gauss-Newton system with very high accuracy; see \cref{fig:2Dbrain_pcg}. Furthermore, we show that it considerably
reduces the number of PCG-iterations when used in combination with a
GN-PCG using both a 2D and a 3D example; see \cref{fig:2Dbrain_pcg} and \cref{tab:3Dbrain_pcg}.
\par
Second, we split the separable and non-separable parts of the objective
function by adding an artificial variable and apply ADMM to compute a
saddle-point of the associated augmented Lagrangian. The resulting subproblems can be
solved efficiently. The first
subproblem, albeit being non-convex, is separable and can be broken down into several smaller problems with only
a few hundreds of unknowns that can be solved in parallel using sequential quadratic programming. The
second subproblem consists of minimizing a convex quadratic with a block-diagonal Hessian with BTTB
structure and, thus, can be solved efficiently and in parallel using DCTs. 
We provide a detailed convergence result that is similar to~\cite{WangEtAl2015} but exploits the smoothness of our problem.  We also derive a theoretical lower bound for the augmentation parameter in ADMM.
Using a numerical experiment, we 
compare different adaption strategies for the augmentation parameter in
\cref{fig:3Dbrain_admm}. In our experience, adaptive choice of the
parameter is possible, however, in view of the presented convergence result, we recommend using at least an empirically tuned lower bound to ensure convergence. We found that choosing a sufficiently large
lower bound is critical to ensure smoothness of the solution.
\par
The correction quality for both methods is comparable to
state-of-the-art methods as shown using one example in \cref{fig:compare}. The most striking difference is the reduced
time-to-solution for both methods. Using the proposed preconditioner in GN-PCG, runtime is reduced
by a factor of around 15 in this example using a straightforward
implementation in MATLAB.
The implementation of GN-PCG is freely available as version 2 of
\texttt{HySCO} as part of the ACID toolbox
(\url{http://www.diffusiontools.com/documentation/hysco.html}).
\par
A larger speedup factor of around 50 is achieved in our
experiments using a relatively simple parallel implementation of the
proposed ADMM method. Here, the separable structure of the computationally first, and most challenging, subproblem in ADMM is used by parallelizing over all image slices. Given our promising results, the ADMM method is an ideal candidate for implementation on massively parallel hardware such as Graphics Processing Units (GPU) and can be attractive for real-time applications such as~\cite{DagaEtAl2014}.

\section{Acknowledgements} 
\label{sec:acknowledgements}
We wish to thank Harald Kugel from the Department of Clinical Radiology, University Hospital M\"unster, Germany, for suggestions on a first draft of this manuscript and providing the 2D data used in our numerical experiments.
The 3D data were provided by the Human Connectome Project, WU-Minn Consortium (Principal Investigators: David Van Essen and Kamil Ugurbil; 1U54MH091657) funded by the 16 NIH Institutes and Centers that support the NIH Blueprint for Neuroscience Research; and by the McDonnell Center for Systems Neuroscience at Washington University.
We also like to thank Siawoosh Mohammadi, University Hospital Hamburg Eppendorf, Germany, for fruitful discussions and helpful advice.
\bibliographystyle{siamplain}
\footnotesize
\bibliography{EPI}
\end{document}